\documentclass[pdflatex,sn-mathphys-num]{sn-jnl}

\usepackage{graphicx}%
\usepackage{multirow}%
\usepackage{amsmath,amssymb,amsfonts}%
\usepackage{amsthm}%
\usepackage{mathrsfs}%
\usepackage[title]{appendix}%
\usepackage{xcolor}%
\usepackage{textcomp}%
\usepackage{manyfoot}%
\usepackage{booktabs}%
\usepackage{tikz}
\usepackage{multirow}
\usepackage{accents}
\usepackage{algorithm}%
\usepackage{algorithmicx}%
\usepackage{algpseudocode}%
\usepackage{pgfplots}
\pgfplotsset{compat=1.18}

\usepackage{mathtools}
\usepackage{mystyle}
\usepackage{bbm}

\usepackage[capitalize,noabbrev]{cleveref}
\crefname{equation}{eq.}{eqs.}
\usepackage{autonum}

\theoremstyle{thmstyleone}%
\newtheorem{theorem}{Theorem}
\newtheorem{proposition}[theorem]{Proposition}%
\newtheorem{lemma}[theorem]{Lemma}

\theoremstyle{thmstyletwo}%
\newtheorem{remark}{Remark}%

\theoremstyle{thmstylethree}%

\begin{document}

\title[A restart-free accelerated algorithm for nonconvex minimization: continuous- and discrete-time analysis]{A restart-free accelerated algorithm for nonconvex minimization: continuous- and discrete-time analysis}


\author*[1]{\fnm{Kansei} \sur{Ushiyama}}\email{ushiyama.k.1053@m.isct.ac.jp}

\author[2]{\fnm{Shun} \sur{Sato}}\email{shun\_sato@tmu.ac.jp}

\affil*[1]{\orgdiv{Department of Mathematical and Computing Science, School of Computing}, \orgname{Institute of Science Tokyo (former Tokyo Institute of Technology)}, \orgaddress{\street{2-12-1 Ookayama}, \city{Meguro-ku}, \postcode{152-8550}, \state{Tokyo}, \country{Japan}}}

\affil[2]{\orgdiv{Department of Mathematical Sciences, Graduate School of Science}, \orgname{Tokyo Metropolitan University}, \orgaddress{\street{1-1 Minami-Osawa}, \city{Hachioji-shi}, \postcode{192-0397}, \state{Tokyo}, \country{Japan}}}


\abstract{
We propose two deterministic momentum-based first-order methods for minimizing lower-bounded nonconvex functions with Lipschitz-continuous gradients and Hessians.
Both methods attain an $\varepsilon$-approximate first-order stationary point in $\Order{\varepsilon^{-7/4}}$ function and gradient evaluations, and their update parameters do not depend on the target accuracy or a prescribed terminal iteration.
The convergence guarantees concern the best point among weighted averages of the iterates.
Instead of relying on momentum-resetting restarts, both methods employ a velocity-control mechanism that adaptively reduces the velocity.
The first method uses a differentiable velocity-control rule, making its last iterate differentiable with respect to the initial point.
To the best of our knowledge, this is the first deterministic first-order method to combine this differentiability property with the stated complexity.
The second method adaptively estimates the gradient and Hessian Lipschitz constants and requires no prior knowledge of their values.
Numerical experiments show that it outperforms existing parameter-free methods on some test problems.
Both methods are derived by discretizing a new continuous-time model described by an ordinary differential equation.
We develop parallel continuous- and discrete-time convergence analyses within the performance estimation problem framework.
}




\maketitle

\allowdisplaybreaks
\section{Introduction}\label{NCAcc:sec:intro}
In this paper, we consider the smooth nonconvex optimization problem
\begin{equation}
    \min_{x\in\RR^d} f(x), \label{problem}
\end{equation}
where $f$ is twice differentiable and bounded below, and its gradient and Hessian are Lipschitz continuous with constants $L$ and $M$, respectively.
Our goal is to find an $\varepsilon$-approximate first-order stationary point, i.e., a point $x \in \RR^d$ satisfying $\norm{\nabla f(x)} \le \varepsilon$, using only evaluations of $f$ and $\nabla f$.
Such first-order methods are particularly effective for large-scale problems, such as training large machine learning models, since they do not require Hessian information, whose computation is often costly.

For functions $f$ with only Lipschitz continuous gradients, gradient descent achieves the optimal complexity $\Order{\varepsilon^{-2}}$ for finding an $\varepsilon$-stationary point (cf.~\cite{N18b}).
Recently, under the additional assumption that the Hessian is Lipschitz continuous, several first-order methods have achieved the improved complexity bound $\Order{\varepsilon^{-7/4}}$~\cite{LL22,MT24,MT24b}.
The dependence on $\varepsilon$ matches the lower bound for deterministic first-order methods established in~\cite{Z26}.
Along this line of research, we propose two first-order methods that attain this optimal complexity.

To achieve the complexity $\Order{\varepsilon^{-7/4}}$, the existing methods in~\cite{LL22,MT24,MT24b} rely on \emph{restart strategies}, which can have several drawbacks.
One drawback is that restart strategies can introduce nonsmoothness into the map from the initial point to the algorithm's output.
Typically, when a restart occurs, internal variables of the algorithm, such as the velocity, are reset to their initial values.
In these methods, the restart procedure uses if--else statements whose conditions depend on internal variables.
Switching between these if--else branches can introduce discontinuities into the output map, making it non-differentiable.
In contrast, a differentiable output map can be composed with a differentiable outer objective function while preserving differentiability, a property that is often useful (see \cref{sec:rel}).
Another drawback is that restarts discard information accumulated by the algorithm.
Although restarting improves the theoretical worst-case complexity bound, it may result in excessive information loss.
This motivates the search for milder strategies that retain more of the accumulated information while preserving the same improved worst-case convergence rate.

In the proposed algorithms, we introduce an alternative strategy, called \emph{velocity control}.
In this strategy, when the norm of the velocity exceeds a threshold, the velocity is reduced rather than reset to zero.
In the first proposed algorithm, this deceleration is implemented smoothly, making the last iterate of the algorithm differentiable with respect to the initial point.
The second algorithm is parameter-free; it does not require any prior knowledge of $L$, $M$, or $\varepsilon$, although its velocity-control rule is not differentiable.

Both proposed algorithms are derived by discretizing the newly introduced ordinary differential equation (ODE)
\begin{equation}
    \ddot{x}(t) + \frac{6}{7}\frac{\alpha}{t^{1/7}} \dot{x}(t) + \nabla f(x(t)) = 0,\qquad x(0) = x_0 \in \RR^d, \quad \dot{x}(0) = 0, \label{ODE}
\end{equation}
where $\alpha > 0$.
This ODE can be interpreted as the motion of a particle in the potential field $f$ under time-decaying friction.
We first analyze the convergence rate of this ODE using the Performance Estimation Problem (PEP) framework (see \cref{sec:rel}).
Subsequently, we analyze the complexity of the proposed algorithms in parallel with the analysis of the ODE.
Hessian-free inequalities (cf.~\cite{MT24}) enable us to exploit the Lipschitz continuity of $\nabla^2 f$ in the analysis of first-order methods, even though $\nabla^2 f$ does not explicitly appear in the algorithms.

\subsection{Main contributions}
The main contributions of this paper are summarized as follows.
\begin{itemize}
\item We propose a new continuous-time model ODE~\eqref{ODE} for nonconvex minimization and prove the convergence rate
\[
\min_{0 \le t \le T} \norm*{\nabla f(\bar{x}(t))} \le \Order{T^{-4/7}}
\]
for a weighted average $\bar{x}$ of the trajectory.

\item By discretizing the ODE~\eqref{ODE}, we derive two novel deterministic first-order algorithms for nonconvex optimization~\eqref{problem}.
Both algorithms achieve the dimension-independent convergence rate
\[
\min_{1 \le t \le T} \norm*{\nabla f(\bar{x}_t)} \le \Order{T^{-4/7}}
\]
for a weighted average $\bar{x}_t$ of the iterates and thus have oracle complexity $\Order{\varepsilon^{-7/4}}$.
Instead of relying on restart mechanisms, these algorithms employ a velocity-control mechanism, which reduces the velocity.
\begin{itemize}
    \item The first is a differentiable algorithm, which enables differentiation of the last iterate with respect to the initial point. 
    To the best of our knowledge, it is the first deterministic first-order method to combine this property with an oracle complexity of $\Order{\varepsilon^{-7/4}}$.
    \item 
    The second is a parameter-free algorithm whose update parameters do not require the Lipschitz constants $L$ and $M$, the target accuracy, or a prescribed terminal iteration.
    Numerical experiments show that, for some problems, this method outperforms existing parameter-free methods.
\end{itemize}
\end{itemize}

\subsection{Related work}\label{sec:rel}
\bmhead{Nonconvex problems with Lipschitz continuous Hessian}
For nonconvex problems with Lipschitz continuous gradients, gradient descent achieves the optimal oracle complexity $\Order{\varepsilon^{-2}}$ for finding an $\varepsilon$-stationary point, matching the lower bound~\cite{CGT10,CDHS20}.
For such functions, an momentum method was proposed in~\cite{OCBP14}, and the behavior of the accelerated gradient method was analyzed in~\cite{OW19}.

Under the additional assumption that the Hessian is Lipschitz continuous, methods with improved oracle complexity $\Order{\varepsilon^{-7/4}\mathrm{Polylog}(\varepsilon^{-1})}$ have been proposed.
Methods that utilize Hessian-vector products include~\cite{AABHM17,CDHS18,ROW20,RW18}, while methods that use only first-order information include~\cite{XJY17,AL18,JNJ18,CDHS17}.
The polylogarithmic factor was removed in~\cite{LL22}, which established an oracle complexity of $\Order{\varepsilon^{-7/4}}$, followed by parameter-free variants in~\cite{MT24,MT24b}.
A method with oracle complexity $\Order{d^{1/4}\varepsilon^{-13/8}}$ was developed in~\cite{JMP25}, where $d$ denotes the dimension.
This bound improves on the $\varepsilon^{-7/4}$ scaling when $d=o(\varepsilon^{-1/2})$.
In independent work, Hermant et al.~\cite{HADHRW26} developed a restart-free first-order method with randomized momentum parameters. Its update equations imply that the last iterate is differentiable with respect to the initial point for a fixed realization of the randomness.
Their convergence guarantee is expressed as an expectation restricted to a high-probability event: for each fixed $\delta\in(0,1)$, $\mathbb{E}\!\left[\mathbbm{1}_{A_N}\min_{1\le k\le N}\norm{\nabla f(\bar{x}_k)}\right]\le\Order{N^{-4/7}}$, $\mathbb{P}(A_N)\ge1-\delta$, for an algorithmic output $\bar{x}_k$.
For a comparison between the proposed methods and existing first-order algorithms, see \cref{tab:methods}.

For deterministic first-order methods over the function class considered here, the previously known dimension-independent lower bound was $\Omega(\varepsilon^{-12/7})$~\cite{CDHS202}.
A recent study established the improved lower bound $\Omega(\varepsilon^{-7/4})$~\cite{Z26}, thereby closing the gap.

\begin{table}[tbp]
    \centering
    \caption{Comparison of existing and proposed first-order methods}
    \begin{tabular}{cccccccc}
        \toprule
        \multirow{2}{*}{Algorithm} & \multirow{2}{*}{$\Order{\varepsilon^{-7/4}}$} & Differentiable  &\multicolumn{3}{c}{Parameter-free} & Second-order & \multirow{2}{*}{Deterministic} \\
        && (last iterate) & $L$ & $M$ & $\varepsilon,N^\dagger$ &stationary & \\
        \cmidrule(r){1-1}\cmidrule(l){2-8}
        \cite[Theorem~2]{XJY17} &&&&&&\checkmark\\
        \cite[Theorem~4]{AL18} &&&&&&\checkmark\\
        \cite[Theorem~8]{JNJ18} &&&&&&\checkmark\\
        \cite[Theorem~1]{CDHS17} &&&&&&&\checkmark\\
        \cite[Theorem~2.2]{LL22} & \checkmark & & &&&&\checkmark\\
        \cite[Theorem~5.7]{MT24} & \checkmark & & \checkmark & \checkmark& \checkmark&& \checkmark\\
        \cite[Theorem~1]{MT24b}${}^1$ & \checkmark & & \checkmark & \checkmark& \checkmark&& \checkmark\\
        \cite[Theorem~4.1]{JMP25} & $\bigcirc^2$ & & & & & & \\
        \cite[Theorem~19(i)]{HADHRW26} & $(\checkmark)^3$ & ${\checkmark}^4$ & & \checkmark \\
        \midrule
        Algorithm 1 & \checkmark &\checkmark & & & \checkmark & & \checkmark \\
        Algorithm 2 & \checkmark & & \checkmark & \checkmark& \checkmark&& \checkmark\\
        \bottomrule
    \end{tabular}
    ${}^\dagger$~A checkmark in this column means that neither the target accuracy $\varepsilon$ nor a prescribed terminal iteration $N$ is used in updates.\\
    ${}^1$~Universal method for H\"older continuous Hessian.\\
    ${}^2$~$\Order{d^{1/4}\varepsilon^{-13/8}}$, where $d$ is the dimension.\\
    ${}^3$~For fixed $\delta\in(0,1)$, $\mathbb{E}[\mathbbm{1}_{A_N}\min_{1\le k\le N}\norm{\nabla f(\bar{x}_k)}]\le\Order{N^{-4/7}}$, with $\mathbb{P}(A_N)\ge1-\delta$.\\
    ${}^4$~Differentiability with respect to the initial point for a fixed realization of the randomness.
    \label{tab:methods}
\end{table}

\bmhead{Differentiable algorithms}
A \emph{differentiable algorithm} is an optimization routine whose output map is differentiable with respect to its inputs and hyperparameters.
Such algorithms can be embedded as subroutines in other optimization procedures, enabling outer-level objectives to be optimized using gradients computed by automatic differentiation.
Incorporating these algorithms into machine learning systems has been reported to improve performance in a range of applications, including hyperparameter tuning and meta-learning~\cite{ADGHPSSd16,FAL17}, generative adversarial network training~\cite{MPPS17}, unrolled sparse coding methods such as LISTA~\cite{GL10}, and inverse problems addressed using unfolded proximal or primal--dual methods~\cite{YSLX16}.

\bmhead{Continuous-time algorithms}
Continuous-time algorithms are models of optimization methods represented by ODEs.
A classical example is gradient descent, which can be modeled by the gradient flow $\dot{x}=-\nabla f(x)$.
For momentum methods, a continuous-time model was first introduced in~\cite{P64}, where both the heavy-ball method and its ODE counterpart
\begin{equation}
    \ddot{x}(t) + \beta\dot{x}(t) + \nabla f(x(t)) = 0 \label{NCHB}
\end{equation}
were studied.
In~\cite{SBC16}, Nesterov's accelerated gradient method (AGM)~\cite{N83} was modeled by the following ODE with vanishing damping:
\begin{equation}
    \ddot{x}(t) + \frac{3}{t}\dot{x}(t) + \nabla f(x(t)) = 0. \label{NCNAG}
\end{equation}
Its convergence rate has the same order as that of its discrete counterpart: for convex $f$, the solution $x(t)$ satisfies $f(x(t))-f^\star\le\Order{t^{-2}}$.
This formulation was later extended to accelerated mirror descent~\cite{KBB15}.
In~\cite{SDJS22}, AGM was connected to a different continuous-time model, known as a high-resolution ODE, involving a Hessian--velocity product.
This analysis yields the best-iterate gradient-norm bound $\min_{0\le i\le k}\norm{\nabla f(x_i)}\le\Order{k^{-3/2}}$ for AGM.
For $\mu$-strongly convex functions, the ODE~\eqref{NCHB} with $\beta=2\sqrt{\mu}$ corresponds to a variant of AGM for strongly convex functions~\cite{N18b}.
Its solution satisfies $f(x(t))-f^\star=\Order{\e^{-\sqrt{\mu}t}}$~\cite{WRJ21}.

For nonconvex functions with an $M$-Lipschitz continuous Hessian, the heavy-ball ODE~\eqref{NCHB} was studied in~\cite{OMT26}.
The authors showed that setting $\beta=(3M)^{2/7}(\Delta_f/T)^{1/7}$, where $\Delta_f=f(x_0)-\inf_{x\in\RR^d}f(x)$, yields the bound $\min_{0\le t\le T}\norm{\nabla f(\bar{x}(t))}\le\Order{T^{-4/7}}$ for a suitable weighted average $\bar{x}(t)$ of the trajectory.
However, this choice of damping coefficient requires prior knowledge of the terminal time $T$, and the convergence bound is established for that prescribed horizon.

Our ODE~\eqref{ODE} is inspired by this model and replaces the horizon-dependent constant damping $\beta=\Theta(T^{-1/7})$ with the time-dependent damping $\beta(t)=\Theta(t^{-1/7})$, where $t$ denotes the current time.
This modification eliminates the need to specify the terminal time $T$ in advance.

In general, discretizing continuous-time models can yield implementable optimization algorithms.
Although systematic approaches to rate-preserving discretization have been studied (e.g.,~\cite{USM23}), the construction of a suitable discretization generally depends on the particular ODE.
Accordingly, this paper adopts a specialized discretization tailored to~\eqref{ODE}.

\bmhead{Analysis strategy: Lyapunov functions vs.~PEP}
The standard approach to analyzing both discrete- and continuous-time algorithms relies on \emph{Lyapunov functions}, also known as \emph{potential functions}.
Typically, once an appropriate Lyapunov function has been constructed, the convergence rate can be readily derived (cf.~\cite{ADA21}).
However, identifying such functions is largely heuristic.
Several studies have explored automated approaches to constructing Lyapunov functions~\cite{TVL18b,TB19,UBT25,SRR22,MTB23,KSUMT24}.

The \emph{Performance Estimation Problem} (PEP) framework, originally proposed in~\cite{DT14}, provides a systematic approach to analyzing and designing first-order methods for convex optimization without requiring an a priori choice of a Lyapunov function.
This framework characterizes the worst-case performance after $K$ iterations as the optimal value of an optimization problem over objective functions and iterates.
A typical formulation is
\begin{align}
    \maximize_{f,\,x^{(0)},\dots,x^{(K)}} &\quad f(x^{(K)})-f^\star,\\
    \subjectto &\quad f \text{ is convex and $L$-smooth},\\
    &\quad x^{(k)} \text{ is generated by a prescribed first-order method},\\
    &\quad \text{conditions on the initial data}.
\end{align}
An upper bound on the optimal value of this problem yields a worst-case convergence guarantee for the method.
The main challenge is that $f$ ranges over an infinite-dimensional function space.
In suitable settings, convex interpolation theory~\cite{THG17} allows this problem to be reformulated as a finite-dimensional semidefinite program (SDP).
This reformulation enables computer-assisted numerical analysis, supported by open-source software packages~\cite{THG17c,GMGH22}, as well as rigorous analytical convergence proofs based on the PEP framework.

Several methods with exact worst-case optimality guarantees have been derived using the PEP framework, including the Optimized Gradient Method (OGM)~\cite{DT14,KF16}, the Information-Theoretic Exact Method (ITEM)~\cite{DT221}, and OptISTA~\cite{JGR25}.
The framework has also led to the development of many other optimal or efficient algorithms and to tight worst-case convergence analyses; see, e.g.,~\cite{DT16,K21,L21b,PR22,KF21,YR21,KF18,dGT17,DS20,DFST25,DTdB22,RTB20,DGT20,BTB23,DVPGR24}.

A continuous-time version of PEP for analyzing ODE models was first proposed in~\cite{KY23b}, followed by an alternative formulation in~\cite{USM24}.
In this paper, we extend the latter formulation to the nonconvex setting and apply it to our analysis.

\bmhead{Hessian-free inequalities}
Our analysis relies on the following Hessian-free inequalities from~\cite{MT24,OMT26}.
Although these inequalities follow from the Lipschitz continuity of the Hessian, they do not explicitly involve the Hessian itself.
As in~\cite{MT24,MT24b}, they enable us to develop fast Hessian-free methods for minimizing functions with Lipschitz continuous Hessians.

\begin{proposition}[Discrete Hessian-free inequality~{\cite[Lemma~3.1]{MT24}}]\label{prop:HFI}
    Let $f \colon \RR^d \to \RR$ be a twice differentiable function with an $M$-Lipschitz continuous Hessian.
    Then, for any $z_1,\ldots,z_n\in\RR^d$ and $\lambda_1,\ldots,\lambda_n\ge0$ satisfying $\sum_{i=1}^n\lambda_i=1$, we have
    \begin{equation}
        \norm*{\nabla f\paren*{\sum_{i=1}^n\lambda_i z_i}-\sum_{i=1}^n\lambda_i\nabla f(z_i)} \le \frac{M}{2}\sum_{1\le i<j\le n}\lambda_i\lambda_j\norm*{z_i-z_j}^2. \label{eq:HFI}
    \end{equation}
\end{proposition}

\begin{proposition}[{\cite[Lemma~3.2]{MT24}}]
    Let $f \colon \RR^d \to \RR$ be a twice differentiable function with an $M$-Lipschitz continuous Hessian.
    Then, for any $x,y\in\RR^d$, we have
    \begin{equation}
        f(x)-f(y) \le \frac12\inpr{\nabla f(x)+\nabla f(y)}{x-y}+\frac{M}{12}\norm*{x-y}^3. \label{HF2}
    \end{equation}
\end{proposition}

\begin{proposition}[Continuous Hessian-free inequality~{\cite[proof of Lemma~1]{OMT26}}]
    Let $f \colon \RR^d \to \RR$ be a twice differentiable function with an $M$-Lipschitz continuous Hessian.
    Fix $t>0$, and let $z \colon [0,t]\to\RR^d$ be continuous and $w \colon [0,t]\to\RR_{\ge0}$ be integrable with $\int_0^t w(s)\dd s=1$.
    Then,
    \begin{equation}
        \norm*{\nabla f\paren*{\int_0^t w(s)z(s)\dd s}-\int_0^t w(s)\nabla f(z(s))\dd s} \le \frac{M}{2}\int_A w(s_1)w(s_2)\norm*{z(s_1)-z(s_2)}^2\dd s_1\dd s_2, \label{HFd}
    \end{equation}
    where $A=\setI{(s_1,s_2)\in\RR^2}{0\le s_1\le s_2\le t}$.
\end{proposition}

The converse of~\cref{prop:HFI} was established in~\cite[Corollary~1]{BDLLM25}: if a differentiable function $f\colon\RR^d\to\RR$ satisfies~\eqref{eq:HFI} for every positive integer $n$, every choice of points $z_1,\ldots,z_n\in\RR^d$, and all nonnegative weights $\lambda_1,\ldots,\lambda_n$ summing to one, then $f$ is twice differentiable with an $M$-Lipschitz continuous Hessian.

\subsection{Notation, assumptions and organization}
\bmhead{Notation}
In our notation, $\inpr{\cdot}{\cdot}$ and $\norm{\cdot}$ denote the inner product and norm in the Euclidean space $\RR^d$, respectively.
For $x\colon [0,\infty) \to \RR^d$, the first- and second-order time derivatives of $x$ are denoted by $\dot{x}$ and $\ddot{x}$, respectively.
For $a,b \in \RR$, $a\wedge b$ and $a\vee b$ denote $\min(a,b)$ and $\max(a,b)$, respectively.
For an integer $k$, $\undertilde{k}$ denotes $k/7$.

\bmhead{Assumptions}
Let $\mathcal{F}$ denote the class of twice differentiable functions $f \colon \RR^d \to \RR$ satisfying the following conditions:
\begin{itemize}
    \item Lower boundedness: $\inf_{x\in\RR^d} f(x) > -\infty$.
    \item $L$-Lipschitz continuity of the gradient: $\norm{\nabla f(x) - \nabla f(y)} \le L \norm{x-y}$ for any $x, y \in \RR^d$.
    \item $M$-Lipschitz continuity of the Hessian: $\norm{\nabla^2 f(x) - \nabla^2 f(y)} \le M \norm{x-y}$ for any $x, y \in \RR^d$, where the norm of matrices is the operator norm.
\end{itemize}

\bmhead{Organization}
The remainder of this paper is organized as follows.
\cref{sec:NCPEP} analyzes the convergence rate of the corresponding ODE.
In \cref{sec:DiscreAlg}, we discretize the ODE and present the proposed algorithms.
\cref{sec:DPEP} provides convergence analyses of the algorithms, in parallel with the analysis in \cref{sec:NCPEP}.
We discuss directions for future research in \cref{sec:conclusion}.
In \cref{sec:exp}, we present numerical results.
Finally, \cref{sec:lemc} and \cref{sec:lemd} contain the proofs of key lemmas used in \cref{sec:NCPEP} and \cref{sec:DPEP}, respectively.

\section{Continuous-time convergence analysis}\label{sec:NCPEP}
In this section, we derive a convergence rate for ODE~\eqref{ODE} using a continuous-time version of PEP.
Specifically, we establish an asymptotic bound on $\min_{0\le t\le T}\norm*{\nabla f(\bar{x}(t))}$, where $\bar{x}$ is a weighted moving average of the solution $x$.
We define $\bar{x}(0)=x(0)$ and, for $t>0$,
\begin{equation}
    \bar{x}(t) \coloneqq \int_{t/2}^t w_t(\tau)x(\tau)\dd\tau, \qquad w_t(\tau) \coloneqq \frac{\e^{\alpha\tau^{6/7}}}{\int_{t/2}^t \e^{\alpha s^{6/7}}\dd s}, \quad \tau\in[t/2,t]. \label{awt}
\end{equation}
We consider this averaged trajectory because the Hessian-free inequality~\eqref{HFd} allows us to bound the gradient norm at the averaged point in terms of the norm of the averaged gradient and a quadratic error term.
The choice of $w_t$ is motivated by the identity
\begin{equation}
    \frac{1}{w_t(\tau)}\dv{w_t(\tau)}{\tau} = \frac{6\alpha}{7\tau^{1/7}} \eqqcolon a_\tau, \label{weidampc}
\end{equation}
where $a_\tau$ is the damping coefficient in ODE~\eqref{ODE}.
This identity plays a crucial role in the proof of \cref{lem:upp}.

We establish the following convergence bound.

\begin{theorem}
    Let $f \colon \RR^d \to \RR$ be a twice differentiable function, and let $x_0 \in \RR^d$.
    Suppose that $f(x_0) - \inf_{x \in \RR^d} f(x) \le R$ and that $\nabla^2 f$ is $M$-Lipschitz continuous.
    Let $x \colon \RR_{\ge 0} \to \RR^d$ be a solution of~\eqref{ODE}.
    Then, for $\bar{x}$ defined by~\eqref{awt}, we have
    \begin{equation}
        \min_{0 \le t \le T} \norm*{\nabla f(\bar{x}(t))}
        \le
        \frac{7}{6}\paren*{
            3\sqrt{\alpha R}
            +
            \frac{2\cdot7^4MR}{6^4\alpha^3}
        }T^{-4/7}
        + \Order{T^{-11/7}}.
    \end{equation}
\end{theorem}

\begin{remark}
    The leading constant is minimized at
    $\alpha = \paren*{\frac{4\cdot7^4M\sqrt{R}}{6^4}}^{2/7}$,
    and the minimum value is
    $\frac{49}{6}\paren*{\frac{7^4MR^4}{2^56^4}}^{1/7}$.
\end{remark}

\begin{remark}\label{rmk:const}
    We can replace the averaging window $[t/2,t]$ with $[bt,t]$ for any $b \in [0,1)$.
    In fact, using the full window $[0,t]$ yields the sharper bound
    \[
        \frac{7}{6}\paren*{
            3\sqrt{\alpha R}
            +
            \frac{7^4 M R}{2\cdot 6^4 \alpha^3}
        } T^{-4/7}
        + \Order{T^{-11/7}}.
    \]
    However, in the discrete-time algorithm, the full window cannot guarantee this rate.
    To clarify the correspondence between the continuous- and discrete-time settings,
    the theorem is therefore stated for the window $[t/2,t]$.
\end{remark}

\begin{proof}
We begin with the following elementary bound.
For any fixed constant $\varepsilon_0 > 0$ and any $T > \varepsilon_0$,
\[
\min_{0 \le t \le T} \norm*{\nabla f\paren*{\bar{x}(t)}}
\le
\min_{\varepsilon_0 \le t \le T} \norm*{\nabla f\paren*{\bar{x}(t)}}
\le
\frac{1}{T-\varepsilon_0}
\int_{\varepsilon_0}^T
\norm*{\nabla f\paren*{\bar{x}(t)}} \dd t.
\]
Since $\varepsilon_0$ is independent of $T$, it only affects lower-order terms in the convergence rate.
Therefore, it suffices to prove that
\[
    \int_{\varepsilon_0}^T \norm*{\nabla f(\bar{x}(t))} \dd t
    = \Order{T^{3/7}}.
\]

We analyze this quantity using a continuous-time version of PEP:
\begin{align}
    \maximize_{\substack{
        f,\\
        x \in C^2([0,T];\RR^d),\\
        \phi \in C^1([0,T];\RR)
    }} \label{eq:PEP}
    &\quad
    \int_{\varepsilon_0}^T \norm*{\nabla f(\bar{x}(t))} \dd t, \\
    \subjectto
    &\quad
    f \text{ is twice differentiable and } \nabla^2 f
    \text{ is } M\text{-Lipschitz continuous}, \\
    &\quad
    \phi(t) = f \paren*{ x(t)} - f(x(T)), \\
    &\quad
    x \text{ is a solution of the ODE } \eqref{ODE},\\
    &\quad
    \phi(0) \le R, \quad \dot{x}(0) = 0.
\end{align}
By deriving an upper bound for this problem, we obtain the desired convergence rate.
The bound is established through the following chain of inequalities, which is standard in PEP:
\[
    \fbox{PEP}
    \le
    \fbox{relaxed problem}
    \le
    \fbox{Lagrange dual problem}
    \le
    \fbox{value at a feasible dual solution}.
\]

\bmhead{Step 1: Relaxation}
The first step is to relax~\eqref{eq:PEP} so that its dual problem can be handled explicitly.
We apply two relaxations.

The first relaxation replaces the objective function in~\eqref{eq:PEP}
with an upper bound involving only the velocity, as stated in the following lemma.
\begin{lemma}\label{lem:upp}
    Define $w_{t}(\tau)$ by~\eqref{awt}.
    Let $x$ be the solution of ODE~\eqref{ODE}, and let $\varepsilon_0$ be a positive constant satisfying
    $\e^{(1 - 2^{-6/7})\alpha \varepsilon_0^{6/7}} \ge 1+ 2^{6/7}$.
    Then, for any $T \ge \varepsilon_0$,
    \begin{equation}
        \int_{\varepsilon_0}^T \norm*{\nabla f (\bar{x}(t))} \dd t
        \le
        \int_{{\varepsilon_0}/2}^T
        \paren*{
            \frac{3\alpha}{t^{1/7}} \norm{\dot{x}(t)}
            +
            \frac{2\cdot7^4M}{6^4\alpha^2}
            t^{1/7}T^{1/7} \norm{\dot{x}(t)}^2
        }\dd t. \label{eq:lemupp}
    \end{equation}
\end{lemma}

\begin{remark}
    When $\alpha = 0.1$, the assumption on $\varepsilon_0$ is satisfied by $\varepsilon_0 \ge 40$.
\end{remark}

The purpose of \cref{lem:upp} is to bound the integral of
$\norm{\nabla f(\bar{x}(t))}$, which is difficult to handle directly in the PEP framework,
by an integral involving $\norm{\dot{x}(t)}$, which is much easier to treat.
This conversion is straightforward if one ignores the tightness of the coefficients of
$\norm{\dot{x}(t)}$ and $\norm{\dot{x}(t)}^2$.
However, establishing the desired convergence rate requires sufficiently tight estimates of these coefficients.
Since this estimate involves a technical calculation, the proof is deferred to \cref{sec:lemc}.

The second relaxation replaces the two constraints
$\phi(t) = f(x(t)) - f(x(T))$ and ODE~\eqref{ODE}
with alternative constraints that eliminate $f$ from the optimization variables.
We focus on the time evolution of $\phi(t)$, into which the information from ODE~\eqref{ODE}
can be incorporated as follows:
\begin{equation}
    \dot{\phi}(t)
    = \inpr{\nabla f \paren*{x(t)}}{\dot{x}(t)}
    =
    -\inpr*{
        \ddot{x}(t) + \frac{6}{7}\frac{\alpha}{t^{1/7}} \dot{x}(t)
    }{\dot{x}(t)}
    =
    -\frac{1}{2}\dv{}{t}\normm{\dot{x}(t)}
    -
    \frac{6\alpha}{7t^{1/7}} \normm{\dot{x}(t)}.
\end{equation}
We retain this equation as a new constraint, together with the terminal condition
$\phi(T) = 0$.

The resulting relaxed PEP is
\begin{align}
    \maximize_{\substack{
        x \in \setI{C^2([0,T];\RR^d)}{\dot{x}(0)=0},\\
        \phi \in \setI{C^1([0,T];\RR)}{\phi(T)=0}
    }}
    &\quad
    \int_{{\varepsilon_0}/2}^T
    \paren*{
        \frac{3\alpha}{t^{1/7}} \norm{\dot{x}(t)}
        +
        \frac{2\cdot7^4M}{6^4\alpha^2}
        t^{1/7}T^{1/7} \norm{\dot{x}(t)}^2
    }\dd t
    \\
    \subjectto
    &\quad
    \dot{\phi}(t)
    +
    \frac{1}{2}\dv{}{t}\normm{\dot{x}(t)}
    +
    \frac{6\alpha}{7t^{1/7}} \normm{\dot{x}(t)}
    = 0,  \label{acconst1}\\
    &\quad
    \phi(0) \le R. \label{acconst2}
\end{align}

\bmhead{Step 2: Dual problem}
We next derive the Lagrange dual problem to obtain a further upper bound by weak duality.
Let the dual variables be $\lambda \in C^1([0,T])$ and $\eta \ge 0$.
The Lagrange function $\mathcal{L}$ is given by
\begin{align}
    \mathcal{L}(x,\phi; \lambda, \eta)
    &=
    (\lambda(0) - \eta)\phi(0)
    +
    \int_0^T \dot{\lambda}(t) \phi(t) \dd t\\
    &\quad
    -
    \frac{1}{2}\lambda(T)\normm{\dot{x}(T)}
    +
    \frac{1}{2}\int_0^T \dot{\lambda}(t) \normm{\dot{x}(t)} \dd t
    -
    \int_0^{{\varepsilon_0}/2}
    \frac{6\alpha}{7t^{1/7}}\lambda(t) \normm{\dot{x}(t)} \dd t\\
    &\quad
    -
    \int_{{\varepsilon_0}/2}^T
    \sqparen*{
        \paren*{
            \frac{6\alpha}{7t^{1/7}} \lambda(t)
            -
            \frac{2\cdot7^4M}{6^4\alpha^2}t^{1/7}T^{1/7}
        }\normm{\dot{x}(t)}
        -
        \frac{3\alpha}{t^{1/7}} \norm{\dot{x}(t)}
    }\dd t
    +
    \eta R.
\end{align}
Thus, the dual problem reads
\[
\minimize_{\substack{\lambda \in C^1([0,T]),\\\eta \ge 0}}
\max_{\substack{
    x \in \setI{C^2([0,T];\RR^d)}{\dot{x}(0)=0},\\
    \phi \in \setI{C^1([0,T];\RR)}{\phi(T)=0}
}}
\mathcal{L}(x,\phi; \lambda, \eta).
\]
For $\max_{x,\phi}\mathcal{L}$ to be finite, the first term requires
$\eta = \lambda(0)$, the second term requires $\lambda(t)$ to be constant,
and the third term requires $\lambda(T) \ge 0$.
Thus, slightly abusing notation by writing $\lambda = \lambda(0)$,
and discarding the nonpositive third and fifth terms, we obtain the following upper bound:
\begin{align}
    \minimize_{\lambda \ge 0}
    &\quad
    \max_{x}
    \bigg\{
    \int_{{\varepsilon_0}/2}^T
    \bigg[
    -
    \paren*{
        \frac{6\alpha}{7t^{1/7}}\lambda
        -
        \frac{2\cdot7^4M}{6^4\alpha^2}t^{1/7}T^{1/7}
    }\normm{\dot{x}(t)}
    +
    \frac{3\alpha}{t^{1/7}} \norm{\dot{x}(t)}
    \bigg]\dd t
    \bigg\}
    +
    \lambda R.
    \label{acvarcontdual}
\end{align}
The maximized functional above depends on $x$ only through $\norm{\dot{x}(t)}$.
Hence, the maximization over $x$ can be reduced to a pointwise maximization over
$\norm{\dot{x}(t)}$.

\bmhead{Step 3: Feasible solution}
Finally, we evaluate the dual objective at the feasible choice
\[
    \lambda =
    \paren*{C+\frac{2\cdot7^4}{6^4}}
    \frac{7M}{6\alpha^{3}}T^{3/7},
\]
where $C>0$ is a constant to be optimized later.
This choice balances the maximization term and the $\lambda R$ term in~\eqref{acvarcontdual}.
With this choice of $\lambda$, and using $t\le T$ in the integral, the coefficient of
$\norm{\dot{x}(t)}^2$ satisfies
\begin{align}
    &-
    \paren*{
        \paren*{C+\frac{2\cdot7^4}{6^4}}
        \frac{M}{\alpha^2}\frac{T^{3/7}}{t^{1/7}}
        -
        \frac{2\cdot7^4M}{6^4\alpha^2}t^{1/7}T^{1/7}
    } \\
    &\qquad =
    -\frac{M}{\alpha^{2}}\frac{T^{3/7}}{t^{1/7}}
    \paren*{
        C
        +
        \frac{2\cdot 7^4}{6^4}
        \paren*{1-\frac{t^{2/7}}{T^{2/7}}}
    } \le
    -\frac{CM}{\alpha^{2}}\frac{T^{3/7}}{t^{1/7}}.
\end{align}
Therefore, the objective function in~\eqref{acvarcontdual} is bounded above by
\begin{align}
    \MoveEqLeft
    \int_{\varepsilon_0/2}^T
    \max_{u \ge 0}
    \bigg\{
        -\frac{CM}{\alpha^{2}}\frac{T^{3/7}}{t^{1/7}}u^2
        +
        \frac{3\alpha}{t^{1/7}} u
    \bigg\} \dd t
    +
    \frac{7M}{6\alpha^{3}}
    \paren*{C + \frac{2\cdot7^4}{6^4}}T^{3/7} R\\
    &=
    \int_{\varepsilon_0/2}^T
    \frac{
        \displaystyle\paren*{\frac{3\alpha}{t^{1/7}}}^2
    }{
        \displaystyle
        4\frac{CM}{\alpha^{2}}\frac{T^{3/7}}{t^{1/7}}
    } \dd t
    +
    \frac{7M}{6\alpha^{3}}
    \paren*{C + \frac{2\cdot7^4}{6^4}}T^{3/7} R\\
    &=
    \int_{\varepsilon_0/2}^T
    \frac{3^2\alpha^4}{4CMT^{3/7}}\frac{1}{t^{1/7}} \dd t
    +
    \frac{7M}{6\alpha^{3}}
    \paren*{C + \frac{2\cdot7^4}{6^4}}T^{3/7} R\\
    &\le
    \frac{7}{6}
    \paren*{
        \frac{3^2\alpha^4}{4CM}
        +
        \frac{M}{\alpha^{3}}
        \paren*{C + \frac{2\cdot7^4}{6^4}} R
    }T^{3/7}.
\end{align}
Choosing $C$ to minimize the right-hand side gives the upper bound
\[
    \frac{7}{6}
    \paren*{
        3\sqrt{\alpha R}
        +
        \frac{2\cdot7^4MR}{6^4\alpha^3}
    }T^{3/7}
\]
for~\eqref{eq:PEP}.

Combining this estimate with the initial averaging argument and using the fact that
$\varepsilon_0$ is independent of $T$, we obtain
\begin{align}
    \MoveEqLeft
    \min_{0 \le t \le T} \norm*{\nabla f(\bar{x}(t))}
    \le
    \frac{1}{T-\varepsilon_0}
    \int_{\varepsilon_0}^T
    \norm*{\nabla f(\bar{x}(t))} \dd t\\
    &\le
    \frac{7}{6}\paren*{
        3\sqrt{\alpha R}
        +
        \frac{2\cdot7^4MR}{6^4\alpha^3}
    }T^{-4/7}
    +
    \Order{T^{-11/7}}.
\end{align}
This completes the proof.
\end{proof}

\section{Discretization with velocity control}\label{sec:DiscreAlg}
In \cref{sec:discretization}, we first introduce an update scheme and observe that it can be regarded as a discretization of ODE~\eqref{ODE}.
We then describe the proposed algorithms based on this discretization with a velocity-control strategy in \cref{sec:alg}.
This discretization allows us to conduct a convergence analysis in parallel with the continuous-time setting.
We first present a differentiable algorithm and then introduce a parameter-free algorithm by incorporating an adaptation process.

\subsection{Discretization}\label{sec:discretization}
We now discretize ODE~\eqref{ODE} to derive optimization algorithms.
Hereafter, $t\in\mathbb{Z}_{\ge0}$ denotes the iteration counter, which we also refer to as time.
We reuse the symbols $w_t(\tau)$ and $a_\tau$ for the discrete-time weight and damping coefficient defined below.
We define $v_t\coloneqq x_t-x_{t-1}$ for $t\ge1$ and refer to it as the velocity.
Note that $v_t$ represents the displacement over one iteration; unlike the usual finite-difference approximation of continuous-time velocity, it is not divided by the step size.
Starting from an initial point $x_0 \in \RR^d$, we discretize the ODE as
\begin{equation}
    (2-\hat{r}_{t+1} + a_{t+1}) \frac{v_{t+1}}{h_{t+1}^2}
    =
    \hat{r}_{t}\frac{2+a_{t+1}}{2+a_{t}}\frac{v_{t}}{h_{t+1}^2}
    - \nabla f(x_{t}),
    \qquad v_0 = 0,\label{scheme}
\end{equation}
where $\hat{r}_{t} \in [0,1]$ and the step size $h_t > 0$ are defined in the algorithms below.

For iteration $t \ge 2$, we define the weight $w_t(\tau)$ for $t_0 \le \tau < t$ and the output $\bar{x}_t$ by
\begin{align}
    w_{t}(\tau) \coloneqq
    \frac{\e^{\alpha \tau^{6/7}}}
    {\sum_{s=t_0}^{t-1} \e^{\alpha s^{6/7}}},
    \qquad
    \bar{x}_t \coloneqq \sum_{\tau=t_0}^{t-1} w_{t}(\tau)x_\tau,
    \label{wt}
\end{align}
where $t_0$ is defined for $t \ge 2$ by
\begin{equation}
    t_0 \coloneqq 2^{i-1}
    \quad \text{if} \quad
    2^i \le t < 2^{i+1}
    \quad (i=1,2,\ldots).\label{t0}
\end{equation}
Formally, we set $\bar{x}_1 = x_0$.
The averaging window $[t_0,t-1]$ is connected to the continuous-time counterpart $[t/2,t]$ in the sense that, for $t\ge 2$, it holds that $t/4 < t_0 \le t/2$.
Although $\bar{x}_t$ appears to require the history
$\setE{x_\tau}_{t_0 \le \tau < t}$, the definition of $t_0$ enables us to compute
$\bar{x}_t$ recursively without storing the entire history, as detailed in \cref{sec:output}.

The damping coefficient is defined as $a_0 = 0$ and
\begin{align}
    a_\tau
    &\coloneqq
    \frac{\e^{\alpha \tau^{6/7}} - \e^{\alpha(\tau-1)^{6/7}}}
    {\e^{\alpha(\tau-1)^{6/7}}}
    =
    \frac{1}{w_{t}(\tau-1)}\frac{w_{t}(\tau)-w_{t}(\tau-1)}{1}
    \label{at}
\end{align}
for $\tau \ge 1$.
This definition is connected to the continuous-time damping coefficient through~\eqref{weidampc}.

The next lemma shows that $a_\tau$ approximates the damping coefficient of ODE~\eqref{ODE}.
\begin{lemma}\label{lem:ad0}
    It holds that
    \begin{equation}
        a_\tau \ge \frac{6\alpha}{7\tau^{1/7}},
        \quad \text{and} \quad
        a_\tau = \Order{\tau^{-1/7}}.
    \end{equation}
\end{lemma}
\begin{proof}
    Since $(1-x)^{6/7} \le 1 - \frac{6}{7}x$ holds for $x \le 1$,
    \begin{equation}
        a_{\tau}
        =
        \e^{\alpha (\tau^{6/7}- (\tau-1)^{6/7})} - 1
        \ge
        \alpha(\tau^{6/7}- (\tau-1)^{6/7})
        =
        \alpha\tau^{6/7}
        \paren*{1-\paren*{1-\frac{1}{\tau}}^{6/7}}
        \ge
        \frac{6\alpha}{7\tau^{1/7}}.
    \end{equation}
    These inequalities are asymptotically tight, and hence $a_\tau = \Order{\tau^{-1/7}}$.
\end{proof}

\bmhead{Correspondence to ODE}
We observe that scheme~\eqref{scheme} can be regarded as a discretization of~\eqref{ODE} when $\hat{r}_t = 1$.
Consider a fixed step size $h_t = h$.
For large $t$, we use the approximations
$a_t \simeq \frac{6\alpha}{7t^{1/7}}$ and
$\frac{2+a_{t+1}}{2+a_{t}} \simeq 1$.
Using these approximations, the scheme becomes approximately
\[
    \frac{x_{t} - 2x_{t-1} + x_{t-2}}{h^2}
    +
    \frac{6}{7}\frac{\alpha h^{-6/7}}{(th)^{1/7}}
    \frac{x_{t}-x_{t-1}}{h}
    +
    \nabla f(x_{t-1})
    = 0.
\]
Here, $(x_{t}-x_{t-1})/{h}$ and
$(x_{t} - 2x_{t-1} + x_{t-2})/{h^2}$ can be interpreted as discretizations of
$\dot{x}(th)$ and $\ddot{x}(th)$, respectively.
Thus, this scheme can be regarded as a discretization of ODE~\eqref{ODE}, except for the discrepancy between $\alpha$ in ODE~\eqref{ODE} and $\alpha h^{-6/7}$ in the above equation.

This discrepancy arises because the definition of $a_t$ does not incorporate the step size $h$.
If we instead defined
\[
    a_t
    =
    \frac{\e^{\alpha (th)^{6/7}} - \e^{\alpha((t-1)h)^{6/7}}}
    {\e^{\alpha((t-1)h)^{6/7}}},
\]
then we would have the approximation
$a_t \simeq \frac{6\alpha h}{7(th)^{1/7}}$, resolving the discrepancy.
Although this modified definition of $a_t$ involving $h$ is more natural than~\eqref{at}, we retain the original definition for the analysis of the parameter-free algorithm, where $h$ varies over iterations.
In the subsequent analysis, we estimate quantities involving $a_t$ and $w_t$ in a manner parallel to the continuous-time analysis.
Such estimates would become difficult if the varying step sizes were incorporated into $a_t$ and $w_t$.

\subsection{Algorithms}\label{sec:alg}
\subsubsection{Differentiable algorithm}
First, we describe a differentiable algorithm based on the discretization~\eqref{scheme}.
We begin by defining a function used in the algorithm.
For fixed $r \in [0,1)$ and a time-dependent threshold $m_t > 0$,
define $\sigma(\,\cdot\,;{m_t}^2) \colon [0,\infty) \to [0,1]$ so that it satisfies
\begin{equation}
\begin{alignedat}{2}
    &\sigma(u;{m_t}^2) = 1
    &&\quad \text{for} \quad 0 \le u \le {m_t}^2,\\
    \text{and} \quad
    &\sigma(u;{m_t}^2) \le r
    &&\quad \text{for} \quad (2{m_t})^2 \le u.
\end{alignedat}\label{sigma}
\end{equation}
For instance, with $S(u) \coloneqq 3u^2 - 2u^3$, we can take
\begin{equation}
\sigma(u; {m_t}^2) =
\begin{cases}
1, & u \le {m_t}^2, \\
1 - (1-r)\,S\paren*{\tfrac13\paren*{\tfrac{u}{{m_t}^2}-1}},
& {m_t}^2 \le u \le (2{m_t})^2, \\
r, & u \ge (2{m_t})^2,
\end{cases}\label{sigma1}
\end{equation}
as depicted in \cref{fig:rplot}.
This function is differentiable, and hence makes the resulting algorithm differentiable.
If piecewise smoothness is sufficient, we can instead use a piecewise-smooth $\sigma$ such as
\begin{equation}
    \sigma(u;{m_t}^2)
    =
    \max\setE{
        \min\setE{1,\, 1 - \tfrac13(1 - r)(u/{m_t}^2 - 1)},
        \, r
    }.\label{minmax}
\end{equation}
Using such a function $\sigma$, the proposed algorithm is defined in \cref{alg:velocityC};
the algorithm inherits the smoothness of $\sigma$.

\begin{figure}[htbp]
\centering
\begin{tikzpicture}[scale=1]
  \def\m{1.3}
  \def\r{0.2}

  \draw[->] (0,0) -- (3*\m,0) node[right] {$u$};
  \draw[->] (0,0) -- (0,1.5);

  \draw (0,1)    node[left] {$1$}   --++ (-0.1,0);
  \draw (0,\r)  node[left] {$r$} --++ (-0.1,0);

  \draw[thick,blue,dashed]
    (\m,1)
    -- (2*\m,\r);

  \draw[thick,blue]
    (0,1) -- (\m,1) to[out=0,in=180] (2*\m,\r)-- (3*\m,\r);

  \draw[dotted] (\m,0) -- (\m,1);
  \draw[dotted] (2*\m,0) -- (2*\m,\r);

  \node[below] at (0,0) {0};
  \node[below] at (\m,0) {${m_t}^2$};
  \node[below] at (2*\m,0) {$(2m_t)^2$};

  \node at (2*\m,1.2) {$\sigma(u;{m_t}^2)$};
\end{tikzpicture}
\caption{The blue solid curve shows the example of $\sigma$ defined by~\eqref{sigma1}. The dashed line represents the piecewise-smooth function~\eqref{minmax}.}
\label{fig:rplot}
\end{figure}

\begin{algorithm}[htbp]
    \caption{Velocity control method (fixed step size, differentiable algorithm)}
    \label{alg:velocityC}
    \begin{algorithmic}[1]
    
        \State \textbf{Input:} $x_0, L, M, r, \alpha, T$
        \State \Comment{$r \in [0,1)$ and $\alpha > 0$ can be chosen arbitrarily; recommended values: $r=0.5, \alpha=0.1$}
        \State Initialize $v_0 = 0$, $\hat{r}_0 = 1$, $t=0$, $\rmax = \max\paren*{r,1/2}$, $h^2 = 4(1-\rmax)/L$
        
        \For{$t = 1,\ldots,T$}
        
            \State $v^1_t \gets \dfrac{1}{1 + a_t}
            \paren*{\hat{r}_{t-1}\dfrac{2+a_{t}}{2+a_{t-1}} v_{t-1} - h^2 \nabla f(x_{t-1})}$
            \Comment{$a_t$ is defined by \eqref{at}} \label{line:upd}

            \State $m_t \gets \dfrac{6\rmax \alpha}{7h^2 M\, t^{1/7}}$
            \State $\hat{r}_t \gets \sigma(\norm{v^1_t}^2;{m_t}^2)$
            \Comment{$\sigma$ satisfies \eqref{sigma}} \label{line:sigma}
        
            \State $v_t \gets \dfrac{1 + a_t}{2 - \hat{r}_t + a_t} v^1_t$ \label{line:corr}
        
            \State $x_t \gets x_{t-1} + v_t$ \label{line:x}
        
            \State Compute output $\bar{x}_t$
            \Comment{$\bar{x}_t$ is defined by \eqref{wt} and calculated by \cref{alg:output}}
        
        \EndFor
        
        \State \textbf{Output:} $\bar{x}_T$
    
    \end{algorithmic}
\end{algorithm}

\bmhead{Behavior of algorithm} \cref{alg:velocityC} can be understood as follows.
We introduce an auxiliary variable~$\hat{r}_t$ upon discretization, which implements the velocity-control mechanism, a newly introduced alternative to restart.
This variable is determined adaptively.
First, Line~\ref{line:upd} calculates a trial velocity~$v^1_t$ using~\eqref{scheme} with the tentative value $\hat{r}_{t} = 1$.
Then, Line~\ref{line:sigma} computes the actual value of $\hat{r}_{t}$ by applying $\sigma$ to ${\norm{v^1_t}}^2$.
By the definition of~$\sigma$, we keep $\hat{r}_t = 1$ if $\norm{v^1_{t}} \le m_t$, that is, if the velocity does not exceed the threshold~$m_t$.
Otherwise, we adopt a smaller value $\hat{r}_t \in [0,1)$.
Line~\ref{line:corr} then corrects the velocity by multiplying $v^1_t$ by the factor
$\frac{1+a_t}{2-\hat{r}_t + a_t}$, which equals $1$ if $\hat{r}_t = 1$ and is less than $1$ otherwise.
Finally, Line~\ref{line:x} computes~$x_t$ from the corrected velocity~$v_t$ and~$x_{t-1}$.
The output of \cref{alg:velocityC} is the weighted average of past iterations defined by~\eqref{wt}, which can be computed efficiently without storing the full history by \cref{alg:output}, detailed in \cref{sec:output}.

\bmhead{Properties used in the convergence analysis}
The velocity-control procedure restricts the possible values of $\hat{r}_t$ and the range of $\norm{v_t}$ according to the tentative velocity $\norm{v_t^1}$, as summarized in \cref{tab:case1}.
These case-wise properties will be used in the convergence analysis.

\begin{table}[htbp]
    \centering
    \caption{Cases for \cref{alg:velocityC}. The resulting bound on $\norm{v_t}$ follows from the relation $v_t = \frac{1 + a_t}{2 - \hat r_t + a_t}v^1_t$. The applied chain rules are explained in \cref{sec:DPEP}.}
    \begin{tabular}{lcccc}
        \toprule
        & \multirow{2}{*}{Condition} & \multicolumn{2}{c}{Resulting} & Applied \\
        & & $\hat r_t$ & $\norm{v_t}$ & Chain Rule \\
        \cmidrule(r){1-2}\cmidrule(lr){3-4}\cmidrule(l){5-5}
        \textbf{Case 1-1a} & $\norm{v^1_t} \le m_t$ & $\hat{r}_t = 1$ & $\norm{v_t} \le m_t$ & \multirow{2}{*}{\eqref{cr1a}} \\
        \textbf{Case 1-1b} & $m_t < \norm{v^1_t} \le 2m_t$ & $0 \le \hat{r}_t \le 1$ & $\tfrac12 m_t < \norm{v_t} \le 2m_t$ & \\[3pt]
        \textbf{Case 1-2} & $2m_t < \norm{v^1_t}$ & $\hat{r}_t \le r$ & $m_t < \norm{v_t}$ & \eqref{cr2a} \\
        \bottomrule
    \end{tabular}
    \label{tab:case1}
\end{table}

\subsubsection{Adaptive algorithm}
Next, we modify \cref{alg:velocityC} to obtain \cref{alg:velocity}, an adaptive version that estimates $L$ and $M$.
This modification removes the need for prior knowledge of these constants.
However, as a trade-off, the possible values of $\hat{r}_t$ are restricted to $\setE{1,r,0}$, and the algorithm is not differentiable in general because the branch decisions depend on the iterates.

In \cref{alg:velocity}, we use the following three quantities to estimate $L$ and $M$, with the estimates denoted by $L_t$ and $M_t$:
\begin{align}
    E_{L_t}^-(x,y)
    &\coloneqq
    f(x) - f(y) - \inpr{\nabla f(y)}{x-y} - \frac{L_t}{2}\normm{x-y}, \label{E-xy}\\
    E_{L_t}^+(x,y)
    &\coloneqq
    f(x) - f(y) - \inpr{\nabla f(x)}{x-y} - \frac{L_t}{2}\normm{x-y}, \label{E+xy}\\
    M_{\mathrm{est}}(x,y)
    &\coloneqq
    \frac{
        12\paren*{
            f(x) - f(y)
            - \frac{1}{2}\inpr{\nabla f(x) + \nabla f(y)}{x-y}
        }
    }{\norm{x-y}^3}, \label{Mxy}
\end{align}
where $M_{\mathrm{est}}(x,x)$ is defined to be 0.

\begin{algorithm}[htbp]
\caption{Velocity control method (adaptive step size, parameter-free algorithm)}
\label{alg:velocity}
\begin{algorithmic}[1]

\State \textbf{Input:} $x_0, L_0>0, \alpha > 0, r \in (0,1), \hmax > 0, \beta_\mathrm{inc} >1, \beta_\mathrm{dec} \in (0,1)$

\State Initialize $t = 1$, $x_0^0 = x_0$, $v_0 = v_0^0 = 0$, $\hat{r}_0 = 1$, $L_1 = L_0$, $\rmax = \max\paren*{r,1/2}$

\While{stopping criterion is not satisfied}

    \While{true}
    \Comment{Backtracking of step size}

        \State $h_t^2 \gets \min\paren*{{4(1-\rmax)}/{L_t}, \hmax}$

        \State $v_{\mathrm{pre}} \gets \hat{r}_{t-1}\frac{2+a_t}{2+a_{t-1}} v_{t-1} - h_t^2 \nabla f(x_{t-1})$
        \Comment{$a_t$ is defined by \eqref{at}} \label{line:pre}

        \State $v^R_t \gets \frac{v_{\mathrm{pre}}}{2-R+a_t}$ \quad for $R \in \setE{0, r, 1}$

        \State $x^R_t \gets x_{t-1} + v^R_t$ \quad for $R \in \setE{0, r, 1}$ \label{line:xcand}

        \If{$E_{L_t}^-(x^R_t,x_{t-1}) \le 0$ for all $R \in \setE{0, r, 1}$} \label{line:estL}\Comment{$E_{L_t}^-$ is defined by~\eqref{E-xy}}
            \State \textbf{break}
        \EndIf

        \State $L_t \gets \beta_\mathrm{inc} L_t$ \label{line:inc}

        \If{$L_t > L_{t-1}$}
            \State $(\hat{r}_{t-1},v_{t-1},x_{t-1}) \gets (0, v^0_{t-1}, x^0_{t-1})$ \label{line:rewrite}
        \EndIf

    \EndWhile

    \State $M_t \gets \max(M_\mathrm{est}(x^1_t, x_{t-1}),0)$ \label{line:Mest}\Comment{$M_\mathrm{est}$ is defined by~\eqref{Mxy}}
    \State $m_t \gets \frac{6\rmax \alpha}{7\hmax M_t t^{1/7}}$ \Comment{if $M_t = 0$, formally $m_t = +\infty$}

    \If{$\norm{v_t^1} \le m_t$} 

        \State $(\hat{r}_t,v_t,x_t) \gets (1, v^1_t, x^1_t)$, $L_{t+1} \gets L_t$ \label{line:r1}

    \ElsIf{$E^+_{L_t}(x^{r}_t,x_{t-1}) \le 0$}\Comment{$E^+_{L_t}$ is defined by~\eqref{E+xy}}

        \State $(\hat{r}_t,v_t,x_t) \gets (r, v^r_t, x^r_t)$, $L_{t+1} \gets \beta_\mathrm{dec}L_t$ \label{line:rr}

    \Else

        \State $(\hat{r}_t,v_t,x_t) \gets (0, v^0_t, x^0_t)$, $L_{t+1} \gets \beta_\mathrm{dec}L_t$ \label{line:r0}

    \EndIf

    \State Compute output $\bar{x}_t$
    \Comment{$\bar{x}_t$ is defined by \eqref{wt} and calculated by \cref{alg:output}}

    \State $t \gets t+1$

\EndWhile

\State \textbf{Output:} $\bar{x}_{t^\star}$ such that $t^\star \in \argmin_{1\le s\le {t-1}} \norm{\nabla f(\bar{x}_s)}$

\end{algorithmic}
\end{algorithm}

\bmhead{Behavior of algorithm} The algorithm proceeds as follows.
\begin{enumerate}
    \item \textbf{Candidate generation.}
    Lines~\ref{line:pre}--\ref{line:xcand} compute all possible candidates for the next iterate $(v_t,x_t)$, namely,
    $(v^1_t,x^1_t)$, $(v^r_t,x^r_t)$, and $(v^0_t,x^0_t)$, corresponding to
    $\hat{r}_t =1,r,0$, respectively.

    \item \textbf{Estimation of $L_t$.}
    Line~\ref{line:estL} attempts to ensure
    $E_{L_t}^-(x_t,x_{t-1}) \le 0$ for all the candidates of $x_t$, which means
    \begin{equation}
        f(x_{t}) - f(x_{t-1})
        \le
        \inpr{\nabla f(x_{t-1})}{v_t}
        +
        \frac{L_t}{2}\norm{v_t}^2. \label{Lest1}
    \end{equation}
    If this condition fails, the algorithm increases the estimate $L_t$ and recomputes $(v_t,x_t)$ in order to retry~\eqref{Lest1}.
    This process terminates after finitely many trials because~\eqref{Lest1} is satisfied for any $L_t \ge L$, since $\nabla f$ is $L$-Lipschitz continuous.
    Therefore, $L_t$ does not exceed $\Lmax \coloneqq \max\setE{\beta_\text{inc}L,L_0}$.
    If $L_t > L_{t-1}$ after finalizing the estimate $L_t$, Line~\ref{line:rewrite} rewrites the previous iteration with $\hat{r}_{t-1} = 0$.

    \item \textbf{Estimation of $M_t$.}
    Line~\ref{line:Mest} computes the estimate
    $M_t = \max \setE{M_{\mathrm{est}}(x^1_t,x_{t-1}),0}$, which is the minimum nonnegative value satisfying
    \begin{align}
        f(x_t) - f(x_{t-1})
        &\le
        \frac{1}{2}\inpr{\nabla f(x_t) + \nabla f(x_{t-1})}{v_t}
        +
        \frac{M_t}{12}\norm{v_t}^3 \label{Mest}
    \end{align}
    if $\hat{r}_t = 1$ is accepted.
    This estimate is used to set the threshold $m_t$ for velocity control.
    Note that~\eqref{Mest} is automatically satisfied by setting $M_t = M$ in view of~\eqref{HF2}.
    Thus, $M_t$ does not exceed $M$.

    \item \textbf{Velocity control and computation of $(v_t,x_t)$.}
    After estimating $L$ and $M$, the algorithm performs velocity control.
    If the trial velocity satisfies $\norm{v^1_t} \le m_t$, Line~\ref{line:r1} finalizes $\hat{r}_t =1$ and adopts $(v^1_t,x^1_t)$ as $(v_t,x_t)$.
    Otherwise, the algorithm attempts to ensure the additional inequality
    $E_{L_t}^+(x_t,x_{t-1}) \le 0$, i.e.,
    \begin{equation}
        f(x_{t}) - f(x_{t-1})
        \le
        \inpr{\nabla f(x_{t})}{v_t}
        +
        \frac{L_t}{2}\norm{v_t}^2, \label{Lest2}
    \end{equation}
    which is used in the convergence analysis.
    If this condition is satisfied, $\hat{r}_t =r$ is adopted and
    $(v^r_t,x^r_t)$ is used as $(v_t,x_t)$.
    Otherwise, $\hat{r}_t =0$ is adopted and
    $(v^0_t,x^0_t)$ is used as $(v_t,x_t)$.
\end{enumerate}

\bmhead{Properties used in the convergence analysis}
We summarize the properties of \cref{alg:velocity} that will be used in the convergence analysis.
Since $\hat{r}_t$ and $x_t$ can be rewritten at iteration $t+1$ by Line~\ref{line:rewrite}, hereafter $(\hat{r}_t, x_t, v_t)$ denotes the element in the final history $\setE{(\hat{r}_t, x_t, v_t)}_{0\le t\le T}$ at the time the algorithm terminates.
The velocity-control procedure restricts the possible values of $\hat{r}_t$ and the range of $\norm{v_t}$ according to the tentative velocity $\norm{v_t^1}$, as summarized in \cref{tab:case2}.
At each iteration, the guaranteed inequalities depend on the value of $\hat{r}_t$.
In the subsequent convergence analysis, we use the following implications:
\begin{equation}
    \hat{r}_t = 1 \Longrightarrow \eqref{Lest1}, \eqref{Mest},
    \qquad
    \hat{r}_t = r \Longrightarrow \eqref{Lest1}, \eqref{Lest2},
    \qquad
    \hat{r}_t = 0 \Longrightarrow \eqref{Lest1}.\label{rtoineq}
\end{equation}
Moreover, the estimates of $L$ and $M$ are bounded as
\begin{equation}
    L_t \le \Lmax = \max\setE{\beta_\text{inc}L,L_0}, \quad M_t \le M. \label{ests}
\end{equation}
The variable $\hat{r}_t$ is also related to the step sizes.
Since Line~\ref{line:r1} tentatively sets $L_{t+1} = L_t$ and in the next iteration, Line~\ref{line:rewrite} rewrites $\hat{r}_{t-1}$ as 0 whenever $L_t > L_{t-1}$, we have the following implications:
\begin{align}
    h_{t+1}^2 = h_{t}^2 \quad &\text{if} \quad \hat{r}_{t} = 1. \label{ifr1}\\
    \hat{r}_{t} = 0 \quad &\text{if} \quad h_{t+1}^2 < h_{t}^2, \label{ifr0}
\end{align}

\begin{table}[htbp]
    \centering
    \caption{Cases for \cref{alg:velocity}. Case 2-2a is an exceptional case that occurs when $\hat{r}_t = 1$ at time $t$ but is changed to $\hat{r}_t = 0$ at time $t+1$, that is, when $\hat{r}_t = 1$ before Line~\ref{line:rewrite} is invoked at time $t+1$. The applied chain rules are explained in \cref{sec:DPEP}.}
    \begin{tabular}{lcccc}
        \toprule
        & \multirow{2}{*}{Condition} & \multicolumn{2}{c}{Resulting} & Applied \\
        & & $\hat r_t$ & $\norm{v_t}$ & Chain Rule \\
        \cmidrule(r){1-2}\cmidrule(lr){3-4}\cmidrule(l){5-5}
        \textbf{Case 2-1} & $\norm{v^1_t} \le m_t$ and Case 2-2a does not occur & $\hat{r}_t = 1$ & $\norm{v_t} \le m_t$ & \eqref{cr1a} \\[5pt]
        \textbf{Case 2-2a} & Described in Caption. & $\hat{r}_t = 0$ & $\norm{v_t} \le m_t$ & \multirow{2}{*}{\eqref{cr2a}} \\
        \textbf{Case 2-2b} & $m_t < \norm{v^1_t}$ & $\hat{r}_t = r$ or 0 & $\frac12m_t < \norm{v_t}$ & \\
        \bottomrule
    \end{tabular}
    \label{tab:case2}
\end{table}

The convergence rates of these two algorithms are established in the next section.

\section{Discrete-time convergence analysis}\label{sec:DPEP}
In this section, we analyze the convergence rates of \cref{alg:velocityC} and \cref{alg:velocity} for
$\min_{1\le t \le T-1}\norm*{\nabla f(\bar{x}_t)}$.
The proof follows the same three-step strategy as in the continuous-time analysis:
we first relax the discrete-time PEP, then derive a Lagrange dual problem, and finally evaluate it at a feasible dual solution.

\begin{theorem}\label{thm:rateC}
    Suppose that $f \in \mathcal{F}$ and $f(x_0) - \inf f \le R$.
    Let $\bar{x}_t$ be the output of \cref{alg:velocityC}.
    Then, it holds that
    \begin{equation}
        \min_{1 \le t \le T-1} \norm*{\nabla f\paren*{\bar{x}_t}}
        \le
        \paren*{
            \frac{112(1-\rmax) MR}{\rmax \alpha^3L}
            +
            \frac{1183\alpha^4L^2}{1296(1-\rmax)^2M}
        }T^{-4/7}
        +
        \Order{T^{-6/7}}.
    \end{equation}
\end{theorem}

\begin{theorem}\label{thm:rate}
    Suppose that $f \in \mathcal{F}$ and $f(x_0) - \inf f \le R$.
    Let $\bar{x}_t$ be the output of \cref{alg:velocity} and $\Lmax = \max\setE{\beta_\text{inc}L,L_0}$.
    Then, it holds that
    \begin{equation}
        \min_{1 \le t \le T-1} \norm*{\nabla f\paren*{\bar{x}_t}}
        \le
        \paren*{
            \frac{28\hmax MR}{\rmax \alpha^3}
            +
            \frac{1183\alpha^4}{81M}
            \paren*{
                \frac{\Lmax}{4(1-\rmax)}
                \vee
                \frac{1}{\hmax}
            }^2
        }T^{-4/7}
        +
        \Order{T^{-6/7}}.
    \end{equation}
\end{theorem}

We prove both theorems simultaneously, since the two algorithms are based on the same discretization of the ODE~\eqref{ODE} and differ only in the schedules of $L_t$, $M_t$, $h_t$, and $\hat{r}_t$.

\begin{proof}
Recall that $(\hat{r}_t, x_t, v_t)$ denotes the element in the final history
$\setE{(\hat{r}_t, x_t, v_t)}_{0\le t\le T}$ at the time the algorithm terminates.
Since for any fixed integer $\varepsilon_0$ such that $1 < \varepsilon_0 < T$,
\[
\min_{1 \le t \le T-1} \norm*{\nabla f\paren*{\bar{x}_t}}
\le
\min_{\varepsilon_0 \le t \le T-1} \norm*{\nabla f\paren*{\bar{x}_t}}
\le
\frac{1}{T-\varepsilon_0}
\sum_{t=\varepsilon_0}^{T-1}
\norm*{\nabla f \paren*{\bar{x}_t}},
\]
we consider the following PEP:
\begin{align}
    \maximize_{\substack{
        f \in \mathcal{F},\ \setE{\phi_t}_{0\le t \le T},\\
        \setE{x_t}_{0\le t \le T},\ \setE{v_t}_{0\le t \le T}
    }} \label{eq:dPEPv}
    &\quad
    \sum_{t=\varepsilon_0}^{T-1}\norm*{\nabla f \paren*{\bar{x}_t}}\\
    \subjectto
    &\quad
    \phi_t
    =
    f \paren*{ x_t}
    -
    \frac{2-\hat{r}_T + a_T}{2+a_T}f(x_T)
    -
    \frac{\hat{r}_T}{2+a_T}f(x_{T-1}), \label{dconst1}\\
    &\quad
    \setE{x_t}_t
    \text{ is generated by \cref{alg:velocityC} or \cref{alg:velocity}}, \label{dconst2}\\
    &\quad
    \phi_0 \le R, \quad v_0 = 0.
\end{align}
We derive an upper bound for this problem by following the same procedure as in the continuous-time analysis.

\bmhead{Step 1: Relaxation}
We begin by relaxing this problem.

\noindent
\textbf{Bounding the objective.}
First, we replace the objective function using the following lemma, whose proof is deferred to \cref{sec:lemd}.
\begin{lemma}[Corresponding to \cref{lem:upp}]\label{lem:uppd}
    Let $w_{t}(\tau)$ and $a_t$ be defined as in \cref{sec:discretization}, and let $\varepsilon_0 \ge 2$ be a sufficiently large integer such that, for all $t \ge \varepsilon_0$, the following conditions hold:
    \begin{align}
        \frac{7^2}{6^2}
        \paren*{\frac{7}{6}+ \frac{\alpha}{\floor{t/4}^{1/7}}}^2
        \frac{\e^{\alpha \floor{t/4}^{6/7}} - 1}
        {\e^{\alpha (\floor{t/4}-1)^{6/7}} - 1}
        &\le 3,\label{asm91}\\
        \e^{-\alpha (t - 1)^{6/7}/2} \le \frac12, \qquad
        \frac{\e^{-\alpha \paren*{(t-1)^{6/7} - (t/2+1)^{6/7}}}}
        {1 - \e^{-\alpha (t-1)^{6/7}/2}}
        &\le \frac{23\cdot2^{1/7}}{t},\label{asm92}\\
        -\frac{7}{3}\e^{\alpha (t-1)^{6/7}}
        \log(1-\e^{-\alpha (t-2)^{6/7}})
        &\le \frac{13}{5},\label{asm93}\\
        2a_{\floor{t/4}} + \frac{26}{49(\floor{t/4}-1)}
        &\le 1.\label{asm94}
    \end{align}
    Then, for the iterates and algorithmic parameters generated by \cref{alg:velocityC} or \cref{alg:velocity}, we have
    \begin{align}
        &\sum_{t=\varepsilon_0}^{T-1}
        \norm*{\nabla f \paren*{\bar{x}_t}}\label{Ghat}\\
        &\le
        \sum_{t=\floor{\varepsilon_0/4}}^{T-1}
        \left[
        \frac{1}{h_t^2}
        \paren*{
            \frac{2\alpha(14-\hat{r}_{t})}{(t-1)^{1/7}}
            +
            \frac{13}{5}(1 - \hat{r}_t)
            +
            \frac{13}{5}
            \paren*{1 - \frac{h_{t}^2}{h_{t+1}^2}\hat{r}_t}
        }
        \norm{v_t}
        +
        6\frac{M}{\alpha^2} t^{1/7} T^{1/7} \norm{v_{t}}^2
        \right].
    \end{align}
\end{lemma}

\begin{remark}
    When $\alpha = 0.1$, the assumption on $\varepsilon_0$ is satisfied by $\varepsilon_0 \ge 40$.
    More specifically, inequalities~\eqref{asm91}, \eqref{asm92}, \eqref{asm93}, and \eqref{asm94} are satisfied by $t \ge 16, 28, 40, 8$, respectively.
    The assumptions~\eqref{asm91}, \eqref{asm92}, and \eqref{asm93} are introduced to guarantee \eqref{asm1}, \eqref{asm2}, and \eqref{asm3}, respectively, which are assumptions of lemmas in \cref{sec:lemd}; \eqref{asm94} is used directly in the proof of \cref{lem:uppd}.
\end{remark}

\noindent
\textbf{Discrete chain-rule substitutes.}
Second, we replace the two constraints~\eqref{dconst1} and \eqref{dconst2}.
In the continuous-time analysis, we used the chain rule
$\dot{\phi}(t) = \inpr{\nabla f(x(t))}{\dot{x}(t)}$
and then incorporated the ODE to eliminate $f$ from the optimization variables.
In discrete time, however, we cannot use the chain rule because $\phi_t$ is defined on $\ZZ_{\ge 0}$.
We therefore construct discrete substitutes for the chain rule.
Depending on the cases described in \cref{tab:case1} and \cref{tab:case2}, we selectively use one of the following approximations at each iteration.

\begin{lemma}[Discrete chain-rule substitutes]\label{lem:dcr}
    For the algorithmic variables at iteration $t$ $(1 \le t \le T)$ and $\phi$ defined by~\eqref{dconst1}, let
    \[
        l_t \coloneqq \frac{\hat{r}_{t}}{2 + a_{t}} \le \frac12.
    \]
    Consider the following two discrete chain-rule substitutes:
    \begin{align}
        \phi_t - \phi_{t-1}
        &\le
        \inpr{l_t\nabla f(x_t) + (1-l_t)\nabla f(x_{t-1})}{v_t}
        +
        (1-2l_t) \frac{L_t}{2} \normm{v_t}
        +
        2l_t \frac{M_t}{12}\norm{v_t}^3, \label{cr1a}\\
        \phi_t - \phi_{t-1}
        &\le
        \inpr{l_t\nabla f(x_t) + (1-l_t)\nabla f(x_{t-1})}{v_t}
        +
        \frac{L_t}{2} \normm{v_t}, \label{cr2a}
    \end{align}
    where $L_t = L$ and $M_t = M$ for \cref{alg:velocityC}.
    Then the following statements hold.
    \begin{enumerate}
        \item If iteration $t$ falls under Case 1-1a or Case 1-1b in \cref{tab:case1}, or Case 2-1 in \cref{tab:case2}, then~\eqref{cr1a} holds.
        \item If iteration $t$ falls under Case 1-2 in \cref{tab:case1}, or Case 2-2a or Case 2-2b in \cref{tab:case2}, then~\eqref{cr2a} holds.
    \end{enumerate}
\end{lemma}

\begin{remark}
    We use~\eqref{cr1a} when $\norm{v_t}$ is sufficiently small so that the cubic error term $\norm{v_t}^3$ can be exploited.
    Otherwise, we use~\eqref{cr2a}.
\end{remark}

\begin{proof}
    Inequality~\eqref{cr1a} follows by adding~\eqref{Lest1} multiplied by $1-2l_t$ and~\eqref{Mest} multiplied by $2l_t$.
    In \cref{alg:velocityC}, inequalities~\eqref{Lest1} and~\eqref{Mest} always hold by setting $L_t = L$ and $M_t = M$.
    In \cref{alg:velocity}, they hold by~\eqref{rtoineq}.

    Inequality~\eqref{cr2a} follows by adding~\eqref{Lest1} multiplied by $1-l_t$ and~\eqref{Lest2} multiplied by $l_t$.
    In \cref{alg:velocityC}, inequalities~\eqref{Lest1} and~\eqref{Lest2} always hold by setting $L_t = L$.
    In \cref{alg:velocity}, if $\hat{r}_t = r$, they hold by~\eqref{rtoineq}.
    If $\hat{r}_t = 0$, then $l_t = 0$, and hence~\eqref{cr2a} reduces to~\eqref{Lest1}, which also holds by~\eqref{rtoineq}.
\end{proof}

\noindent
\textbf{Discrete counterpart of~\eqref{acconst1}.}
For each iteration $t$, we now evaluate $\phi_t - \phi_{t-1}$ by applying one of the above chain-rule substitutes, depending on the cases described in \cref{tab:case1,tab:case2}.

\begin{lemma}[Discrete counterpart of~\eqref{acconst1}]\label{lem:dacconst}
    Let $L_t = L$, $M_t = M$, and $h_t = h$ for \cref{alg:velocityC}.
    For iteration $t$ $(1 \le t \le T)$, it holds that
    \begin{equation}
        \phi_t - \phi_{t-1} \le
        \begin{cases}
            - (2(1-\hat{r}_t) + a_t)
            \paren*{\dfrac{1}{h_{t}^2}-\dfrac{L_t}{4}}
            \norm{v_t}^2
            +
            \hat{r}_t\dfrac{M_t}{12}\norm{v_t}^3
            -\rho_t + \rho_{t-1}\\
            \hspace{80pt}
            \text{if $t \neq T$ and iteration $t$ is Case 1-1a, Case 1-1b, or Case 2-1},\\
            - \paren*{
                \dfrac{2(1-\hat{r}_t)+a_t}{h_{t}^2}
                - \dfrac{L_t}{2}
            }
            \norm{v_t}^2
            -\rho_t + \rho_{t-1}\\
            \hspace{80pt}
            \text{if $t \neq T$ and iteration $t$ is Case 1-2, Case 2-2a, or Case 2-2b},\\
            -\paren*{
                \dfrac{2-\hat{r}_T+a_T}{h_T^2}
                -\dfrac{L_T}{2}
            }
            \normm{v_T}
            +
            \dfrac{2+a_T}{2-\hat{r}_T+a_T}\rho_{T-1}\\
            \hspace{80pt}
            \text{if $t = T$},
        \end{cases}\label{crfa}
    \end{equation}
    where
    \[
        \rho_t \coloneqq
        \hat{r}_t
        \frac{2-\hat{r}_{t+1} + a_{t+1}}{h_{t+1}^2(2 + a_t)}
        \inpr{v_{t+1}}{v_{t}}.
    \]
\end{lemma}

\begin{proof}
    We begin by calculating the term common to~\eqref{cr1a} and~\eqref{cr2a}.
    Applying the scheme~\eqref{scheme}, we have
    \begin{align}
        \MoveEqLeft
        \inpr{l_t\nabla f(x_{t}) + (1-l_t)\nabla f(x_{t-1})}{v_{t}}\\
        &=
        -\left\langle
        \frac{\hat{r}_t}{2 + a_t}
        \paren*{
            (2-\hat{r}_{t+1} + a_{t+1}) \frac{v_{t+1}}{h_{t+1}^2}
            -
            \hat{r}_t\frac{2+a_{t+1}}{2+a_t}\frac{v_{t}}{h_{t+1}^2}
        }
        \right.\\
        &\qquad
        \left.
        +
        \frac{2-\hat{r}_t + a_t}{2 + a_t}
        \paren*{
            (2-\hat{r}_t + a_t) \frac{v_{t}}{h_{t}^2}
            -
            \hat{r}_{t-1}\frac{2+a_t}{2+a_{t-1}}\frac{v_{t-1}}{h_{t}^2}
        },
        {v_{t}}
        \right\rangle\\
        &=
        -
        \frac{1}{2 + a_t}
        \paren*{
            \frac{(2 - \hat{r}_t + a_t)^2}{h_{t}^2}
            -
            \hat{r}_t^2
            \frac{2 + a_{t+1}}{h_{t+1}^2(2+a_t)}
        }\norm{v_t}^2\\
        &\qquad
        -
        \hat{r}_t
        \frac{2-\hat{r}_{t+1} + a_{t+1}}{h_{t+1}^2(2 + a_t)}
        \inpr{v_{t+1}}{v_{t}}
        +
        \hat{r}_{t-1}
        \frac{2-\hat{r}_t + a_t}{h_{t}^2(2 + a_{t-1})}
        \inpr{v_t}{v_{t-1}}.
    \end{align}
    Concerning the coefficient of $\normm{v_t}$, since $a_t$ is decreasing and $\hat{r}_t = 0$ if $h_{t+1} < h_t$ by~\eqref{ifr0}, it follows that
    \begin{align}
        \frac{1}{2 + a_t}
        \paren*{
            \frac{(2 - \hat{r}_t + a_t)^2}{h_{t}^2}
            -
            \hat{r}_t^2
            \frac{2 + a_{t+1}}{h_{t+1}^2(2+a_t)}
        }
        &\ge
        \frac{1}{2 + a_t}
        \paren*{
            \frac{(2 - \hat{r}_t + a_t)^2}{h_{t}^2}
            -
            \frac{\hat{r}_t^2}{h_{t}^2}
        }\\
        &=
        \frac{2(1-\hat{r}_t ) + a_t}{h_{t}^2}.
    \end{align}
    Hence, we have
    \begin{equation}
        \inpr{l_t\nabla f(x_{t}) + (1-l_t)\nabla f(x_{t-1})}{v_{t}}
        \le
        -
        \frac{2(1-\hat{r}_t ) + a_t}{h_{t}^2}
        \norm{v_t}^2
        -\rho_t + \rho_{t-1}.
    \end{equation}

    Therefore, for iteration $t\neq T$ in Case 1-1a, Case 1-1b, or Case 2-1, applying~\eqref{cr1a} gives
    \begin{align}
        \MoveEqLeft
        \phi_t - \phi_{t-1}\\
        &\le
        -\frac{2(1-\hat{r}_t ) + a_t}{h_{t}^2}\norm{v_t}^2
        -\rho_t +\rho_{t-1}
        +
        \frac{2(1-\hat{r}_t) + a_t}{2+a_t}\frac{L_t}{2}\normm{v_t}
        +
        \frac{2\hat{r}_t}{2+a_t}\frac{M_t}{12}\norm{v_t}^3\\
        &\le
        -
        (2(1-\hat{r}_t) + a_t)
        \paren*{\frac{1}{h_{t}^2}-\frac{L_t}{4}}
        \norm{v_t}^2
        +
        \hat{r}_t\frac{M_t}{12}\norm{v_t}^3
        -\rho_t + \rho_{t-1}.
        \label{crf1a}
    \end{align}
    For iteration $t\neq T$ in Case 1-2, Case 2-2a, or Case 2-2b, applying~\eqref{cr2a} gives
    \begin{equation}
        \phi_t - \phi_{t-1}
        \le
        -
        \paren*{
            \frac{2(1-\hat{r}_t)+a_t}{h_{t}^2}
            -
            \frac{L_t}{2}
        }\norm{v_t}^2
        -\rho_t + \rho_{t-1}.
        \label{crf2a}
    \end{equation}
    For iteration $T$, we simply use~\eqref{Lest1}, which always holds, and obtain
    \begin{align}
        \MoveEqLeft
        \phi_t - \phi_{t-1}
        \le
        \inpr{\nabla f(x_{T-1})}{v_{T}}
        +
        \frac{L_T}{2}\normm{v_T} \\
        &\le
        -
        \paren*{
            \frac{2-\hat{r}_T+a_T}{h_T^2}
            -
            \frac{L_T}{2}
        }\normm{v_T}
        +
        \frac{2+a_T}{2-\hat{r}_T+a_T}\rho_{T-1}.
        \label{crf3a}
    \end{align}
\end{proof}

\noindent
\textbf{Relaxed PEP.}
We replace the objective using~\cref{lem:uppd} and replace the constraints~\eqref{dconst1} and~\eqref{dconst2} by~\eqref{crfa}, together with the terminal condition
\begin{equation}
    \frac{2-\hat{r}_T + a_T}{2+a_T}\phi_T
    +
    \frac{\hat{r}_T}{2+a_T}\phi_{T-1}
    = 0,\label{edge1}
\end{equation}
which follows directly from the definition of $\phi$ and corresponds to $\phi(T)=0$ in the continuous-time PEP.
Denoting the right-hand side of~\eqref{Ghat} by $\hat{G}$, the PEP is relaxed to
\begin{align}
    \maximize_{\substack{
        \setE{v_t}_{0\le t \le T} \text{ s.t. } v_0 = 0,\\
        \setE{\phi_t}_{0\le t \le T} \text{ s.t. } \eqref{edge1}
    }} \label{eq:dPEPvr}
    &\quad
    \hat{G} \\
    \subjectto
    &\quad
    \eqref{crfa} \quad (t = 1,\ldots, T),\\
    &\quad
    \phi_0 \le R.
\end{align}

\bmhead{Step 2: Dual problem}
We next construct a dual problem to obtain an upper bound by weak duality.
Instead of deriving the full dual problem, we restrict the dual variables to the following simple ansatz, which mirrors $\dot{\lambda}(t) = 0$ in the continuous-time proof.
Since this restriction gives a feasible dual solution, it provides an upper bound on the relaxed PEP.

For $1 \le t \le T$, let $\lambda_t$ be the dual variable corresponding to the constraint~\eqref{crfa} at iteration $t$.
We use the ansatz that $\lambda_t$ is a constant $\lambda \ge 0$ for $1 \le t \le T-1$, and
$
    \lambda_T = \frac{2-\hat{r}_T + a_T}{2+a_T}\lambda.
$
Then, 
the dual problem is bounded above by
\[
    \min_{\lambda\ge 0, \eta\ge 0}
    \max_{\substack{
        \setE{v_t}_{0\le t \le T} \text{ s.t. } v_0 = 0,\\
        \setE{\phi_t}_{0\le t \le T} \text{ s.t. } \eqref{edge1}
    }}
    \mathcal{L}(v,\phi;\lambda, \eta),
\]
where
\begin{align}
    &\mathcal{L}(v,\phi;\lambda, \eta)\\
    &=
    \sum_{t=\floor{\varepsilon_0/4}}^{T-1}
    \left[
    \frac{1}{h_t^2}
    \paren*{
        \frac{2\alpha(14-\hat{r}_{t})}{(t-1)^{1/7}}
        +
        \frac{13}{5}(1 - \hat{r}_t)
        +
        \frac{13}{5}
        \paren*{1 - \frac{h_{t}^2}{h_{t+1}^2}\hat{r}_t}
    }
    \norm{v_t}
    +
    6\frac{M}{\alpha^2} t^{1/7} T^{1/7} \norm{v_{t}}^2
    \right]\\
    &\quad
    +
    \lambda
    \frac{\hat{r}_{0}}{h_{1}^2}
    \frac{2-\hat{r}_1 + a_1}{2 + a_{0}}
    \inpr{v_1}{v_{0}}
    -
    \eta(\phi_0 - R)
    -
    \lambda
    \paren*{
        \frac{2-\hat{r}_T + a_T}{2+a_T}\phi_T
        +
        \frac{\hat{r}_T}{2+a_T}\phi_{T-1}
        -
        \phi_0
    } \\
    &\quad
    +
    \lambda \sum_{t=1}^{T-1}
    \begin{cases}
        -(2(1-\hat{r}_t) + a_t)
        \paren*{\frac{1}{h_t^2} - \frac{L_t}{4}}\normm{v_t}
        +
        \hat{r}_t\frac{M_t}{12}\norm{v_t}^3
        & \text{if Case 1-1a, 1-1b, or 2-1},\\
        -\paren*{
            \frac{2(1-\hat{r}_t)+a_t}{h_t^2}
            -
            \frac{L_t}{2}
        }\normm{v_t}
        & \text{otherwise},
    \end{cases}\\
    &\quad
    -
    \lambda
    \frac{2-\hat{r}_T + a_T}{2+a_T}
    \paren*{
        \frac{2-\hat{r}_T+a_T}{h_T^2}
        -
        \frac{L_T}{2}
    }\normm{v_T}.
\end{align}
By the condition $v_0 = 0$, the terminal condition~\eqref{edge1}, and the choice $\eta = \lambda$, the second line reduces to $\lambda R$.

To simplify the expression, define
\begin{align}
    F^{t-}_1(u)
    &\coloneqq
    \lambda\hat{r}_t\frac{M_t}{12}u^3
    -
    \lambda (2(1-\hat{r}_t) + a_t)
    \paren*{\frac{1}{h_t^2} - \frac{L_t}{4}}u^2,\\
    F^{t-}_2(u)
    &\coloneqq
    -\lambda
    \paren*{
        \frac{2(1-\hat{r}_t)+a_t}{h_t^2}
        -
        \frac{L_t}{2}
    }u^2,\\
    Q^t(u)
    &\coloneqq
    6\frac{M}{\alpha^2} t^{1/7} T^{1/7}u^2
    +
    \frac{1}{h_t^2}
    \paren*{
        \frac{2\alpha(14-\hat{r}_{t})}{(t-1)^{1/7}}
        +
        \frac{13}{5}(1 - \hat{r}_t)
        +
        \frac{13}{5}
        \paren*{1 - \frac{h_{t}^2}{h_{t+1}^2}\hat{r}_t}
    } u,\\
    F^{t}_1(u)
    &\coloneqq F^{t-}_1(u) + Q^t(u),
    \qquad
    F^{t}_2(u) \coloneqq F^{t-}_2(u) + Q^t(u).
\end{align}
We use $F_1^t$ for Case 1-1a, Case 1-1b, and Case 2-1, and $F_2^t$ for Case 1-2, Case 2-2a, and Case 2-2b.

Hereafter, we use the abbreviations defined in the Notations section, such as
$\svnon = 1/7$, $\svntw = 2/7$, and $\svnth = 3/7$.
Observing that $h_{t+1}^2 = h_t^2$ if $\hat{r}_t = 1$ by~\eqref{ifr1}, we have
\begin{alignat}{3}
    Q^t(u)
    &=
    6\frac{M}{\alpha^2} t^{\svnon} T^{\svnon}u^2
    +
    \frac{1}{h_t^2}\frac{26\alpha}{(t-1)^{\svnon}}u
    &&& \text{if Case 1-1a or Case 2-1},\label{q1}\\
    Q^t(u)
    &\le
    6\frac{M}{\alpha^2} t^{\svnon} T^{\svnon}u^2
    +
    \frac{1}{h_t^2}
    \paren*{
        \frac{28\alpha}{(t-1)^{\svnon}}
        +
        \frac{26}{5}
    }u
    &&& \text{otherwise}.\label{q2}
\end{alignat}
Using these notations, we obtain the following upper bound on the dual problem:
\begin{align}
    \MoveEqLeft
    \min_{\lambda\ge 0, \eta\ge 0}
    \max_{\substack{
        \setE{v_t}_{0\le t \le T} \text{ s.t. } v_0 = 0,\\
        \setE{\phi_t}_{0\le t \le T} \text{ s.t. } \eqref{edge1}
    }}
    \mathcal{L}(v,\phi;\lambda, \eta)\\
    &\le
    \min_{\lambda\ge 0}
    \left[
    \lambda R
    +
    \max_{\norm{v_T}}
    \left\{
        -\lambda
        \paren*{
            \frac{2-\hat{r}_T+a_T}{h_T^2}
            -
            \frac{L_T}{2}
        }\normm{v_T}
    \right\}
    \right.\\
    &\qquad
    +
    \sum_{t=\floor{\varepsilon_0/4}}^{T-1}
    \begin{cases}
        \max_{\norm{v_t} \le m_t} F^t_1(\norm{v_t})
        & \text{if Case 1-1a or Case 2-1},\\
        \max_{\frac{1}{2}m_t \le\norm{v_t} \le 2m_t} F^t_1(\norm{v_t})
        & \text{if Case 1-1b},\\
        \max_{\norm{v_t} \le m_t} F^t_2(\norm{v_t})
        & \text{if Case 2-2a},\\
        \max_{\frac{1}{2}m_t \le \norm{v_t}} F^t_2(\norm{v_t})
        & \text{if Case 1-2 or Case 2-2b}
    \end{cases}\\
    &\qquad
    \left.
    +
    \sum_{t=1}^{\floor{\varepsilon_0/4}-1}
    \begin{cases}
        \max_{\norm{v_t} \le m_t} F^{t-}_1(\norm{v_t})
        & \text{if Case 1-1a or Case 2-1},\\
        \max_{\frac{1}{2}m_t \le\norm{v_t} \le 2m_t} F^{t-}_1(\norm{v_t})
        & \text{if Case 1-1b},\\
        \max_{\norm{v_t} \le m_t} F^{t-}_2(\norm{v_t})
        & \text{if Case 2-2a},\\
        \max_{\frac{1}{2}m_t \le \norm{v_t}} F^{t-}_2(\norm{v_t})
        & \text{if Case 1-2 or Case 2-2b}
    \end{cases}
    \right].
    \label{duald}
\end{align}

\bmhead{Step 3: Feasible dual solution}
Finally, we evaluate the case-wise maximization problems in~\eqref{duald} by substituting the algorithmic parameters and a feasible dual solution.
As in \cref{alg:velocityC} and \cref{alg:velocity}, let $\rmax = r \vee \frac12$.
The step size and the threshold are given by
\begin{gather}
    h_t^2
    =
    \frac{4(1-\rmax)}{L_t} \wedge \hmax,
    \qquad
    m_t
    =
    \frac{6\rmax \alpha}{7\hmax M_tt^{\svnon}},
\end{gather}
where, in \cref{alg:velocityC}, $L_t = L$, $M_t = M$ and $\hmax = {4(1-\rmax)}/{L}$.
If $M_t = 0$ in \cref{alg:velocity}, we regard $m_t$ as $+\infty$.
We use the feasible dual solution
\begin{equation}
    \lambda = \frac{28\hmax MT^{\svnth}}{\rmax \alpha^3}.
\end{equation}

We first handle the terminal term in~\eqref{duald}.
It is equal to zero, since
\begin{align}
    \frac{2-\hat{r}_T+a_T}{h_T^2} - \frac{L_T}{2}
    >
    \frac{1}{h_T^2} - \frac{L_T}{2}
    \ge
    \frac{L_T}{4(1-\rmax)} - \frac{L_T}{2}
    \ge 0.
\end{align}

We next show that the last sum in~\eqref{duald} is nonpositive.
For the coefficient of $u^2$ in $F_1^{t-}$, since
$2(1-\hat{r}_t) + a_t \ge a_t \ge ({6\alpha})/({7t^{\svnon}})$ by \cref{lem:ad0}, and since
\begin{equation}
    \frac{1}{h_t^2} - \frac{L_t}{4}
    \ge
    \max\paren*{
        \frac{L_t}{4}\paren*{\frac{1}{1-\rmax}-1},
        \frac{1}{\hmax} - \frac{L_t}{4}
    }
    \ge
    \frac{\rmax}{\hmax}, \label{hL4}
\end{equation}
we have
\begin{equation}
    \lambda(2(1-\hat{r}_t) + a_t)
    \paren*{\frac{1}{h_t^2} - \frac{L_t}{4}}
    \ge
    \frac{28\hmax MT^{\svnth}}{\rmax \alpha^3}
    \frac{6\alpha}{7 t^{\svnon}}
    \frac{\rmax}{\hmax}
    =
    \frac{24MT^{\svnth}}{\alpha^2t^{\svnon}}. \label{F1q}
\end{equation}
Thus, since $\hat{r}_t \le 1$, it follows that for $u \le 2m_t$,
\begin{equation}
    \frac{F^{t-}_1(u)}{u^2}
    \le
    \lambda\hat{r}_t\frac{M_t}{12} \cdot 2m_t
    -
    \frac{24MT^{\svnth}}{\alpha^2t^{\svnon}}
    \le
    \frac{4MT^{\svnth}}{\alpha^2t^{\svnon}}
    -
    \frac{24MT^{\svnth}}{\alpha^2t^{\svnon}}
    \le
    -\frac{20MT^{\svnth}}{\alpha^2t^{\svnon}}
    < 0.
\end{equation}
For the coefficient of $u^2$ in $F_2^{t-}$, since $F_2^{t-}$ is used only when $\hat{r}_t \le r \le \rmax$, we have
\begin{equation}
    \frac{2(1-\hat{r}_t)+a_t}{h_t^2} - \frac{L_t}{2}
    \ge
    \max\paren*{
        \frac{L_t}{2}\paren*{\frac{1-\hat{r}_t}{1-\rmax}-1},
        \frac{2(1-\hat{r}_t)}{\hmax}- \frac{L_t}{2}
    }
    +
    \frac{a_t}{h_t^2}
    \ge
    \frac{6\alpha}{7h_t^2t^{\svnon}}. \label{F2q}
\end{equation}
Thus, for all $u$,
\begin{align}
    F^{t-}_2(u)
    \le
    -\lambda \frac{6\alpha}{7h_t^2t^{\svnon}} u^2
    \le 0.
\end{align}
Therefore, the last sum in~\eqref{duald} is nonpositive.

It remains to evaluate the main case-wise maximization terms in~\eqref{duald}.
The proof of the following lemma is deferred to \cref{sec:lemcases}.
\begin{lemma}\label{lem:lemcases}
    For the third term of~\eqref{duald}, the following bounds hold:
    there exists a constant $c > 0$ such that, for any sufficiently large $T$,
    \begin{align}
        &\text{for Case 1-1a and Case 2-1,}
        &\quad
        \max_{\norm{v_t} \le m_t} F^t_1(\norm{v_t})
        &\le
        \frac{338\alpha^4}{27h_t^4M t^{\svnon}T^{\svnth}}
        +
        \Order{T^{-\svnsi}\vee t^{-\svnei}T^{-\svnth}},\\
        &\text{for Case 1-1b,}
        &\quad
        \max_{\frac{1}{2}m_t \le\norm{v_t} \le 2m_t} F^{t}_1(\norm{v_t})
        &\le 0,\\
        &\text{for Case 2-2a,}
        &\quad
        \max_{\norm{v_t} \le m_t} F^{t}_2(\norm{v_t})
        &= \Order{T^{-\svnth}},\\
        &\text{for Case 1-2 and Case 2-2b,}
        &\quad
        \max_{\frac{1}{2}m_t \le \norm{v_t}} F^t_2(\norm{v_t})
        &\le -c.
    \end{align}
\end{lemma}


Defining
\[
U_T =
\sum_{t=\floor{\varepsilon_0/4}}^{T-1}
    \begin{cases}
        \Order{T^{-3/7}} & \text{if Case 2-2a occurs},\\
        -\Omega(1) & \text{if Case 2-2b occurs},\\
        0 & \text{otherwise},
    \end{cases}
\]
and combining these evaluations, we obtain the upper bound
\begin{align}
    \MoveEqLeft
    \min_{\lambda\ge 0, \eta\ge 0}
    \max_{\substack{
        \setE{v_t}_{0\le t \le T} \text{ s.t. } v_0 = 0,\\
        \setE{\phi_t}_{0\le t \le T} \text{ s.t. } \eqref{edge1}
    }}
    \mathcal{L}(v,\phi;\lambda, \eta)\\
    &\le
    \frac{28\hmax MT^{3/7}}{\rmax \alpha^3} R
    +
    \sum_{t=\floor{\varepsilon_0/4}}^{T-1}
    \paren*{
        \dfrac{338\alpha^4}{27h_t^4M t^{1/7}T^{3/7}}
        +
        \Order{T^{-6/7}\vee t^{-8/7}T^{-3/7}}
    }
    +
    U_T\label{finish}\\
    &\le
    \frac{28\hmax MT^{3/7}}{\rmax \alpha^3} R  + \frac{1183\alpha^4T^{3/7}}{81M} H + U_T + \Order{T^{1/7}},
\end{align}
where 
\begin{equation}
    H = \sup_{1\le t \le T} h_t^{-4} \le \begin{cases}
        \paren*{\frac{L}{4(1-\rmax)}}^2 & \text{for \cref{alg:velocityC}},\\
        \paren*{\frac{\Lmax}{4(1-\rmax)} \vee \frac{1}{\hmax}}^2 & \text{for \cref{alg:velocity}}.
    \end{cases} \label{suph}
\end{equation}

The positive contribution from Case 2-2a in $U_T$ would be too large to yield the rate $\Order{T^{-4/7}}$ if it occurred $\Theta(T)$ times.
However, this contribution is offset by Case 2-2b, as stated in the following lemma.
The proof is deferred to \cref{sec:lemUT}.
\begin{lemma}\label{lem:UT}
    It holds that $U_T \le \Order{T^{-3/7}}$.
\end{lemma}

Therefore, we obtain the PEP upper bound as
\begin{align}
    \paren*{\frac{28\hmax M}{\rmax \alpha^3} R +  \frac{1183\alpha^4}{81M}H}T^{3/7} + \Order{T^{1/7}},
\end{align}
and hence
\begin{align}
    \MoveEqLeft
    \min_{1\le t \le T-1}\norm*{\nabla f(\bar{x}_t)}
    \le
    \frac{1}{T-\varepsilon_0}
    \sum_{t=\varepsilon_0}^{T-1}
    \norm*{\nabla f(\bar{x}_t)}\\
    &\le  \paren*{\frac{28\hmax M}{\rmax \alpha^3} R +  \frac{1183\alpha^4}{81M}H}T^{-4/7} + \Order{T^{-6/7}}.
\end{align}
Using~\eqref{suph}, \cref{thm:rateC} and \cref{thm:rate} follow.

\end{proof}

\section{Discussion}\label{sec:conclusion}
\bmhead{Damping schedules across function classes}
For continuous-time models of the form
$
    \ddot{x} + \beta(t)\dot{x} + \nabla f(x) = 0,
$
the optimal damping coefficient $\beta$ depends on the convexity of the objective function.
For strongly convex and convex functions, the optimal damping is $\beta(t) = \Theta(1)$ and $\beta(t) = \Theta(1/t)$, respectively. In this work, we show that, for the smooth nonconvex class considered here, the damping schedule $\beta(t) = \frac{6\alpha}{7}t^{-1/7}$ yields
a continuous-time convergence rate corresponding to the optimal oracle complexity in discrete time. The decay exponents of these damping schedules do not vary monotonically as the convexity assumptions are weakened, which is somewhat surprising.

For $\mu$-strongly convex and convex functions, the Information-Theoretic Exact Method~\cite{DT221} achieves the optimal convergence rates for both function classes, and the damping coefficient of its corresponding ODE interpolates between $\Theta(1)$ and $\Theta(1/t)$ through the parameter $\mu$.
It would be an interesting direction for future research to develop an optimal method that simultaneously covers strongly convex, convex, and nonconvex functions.

\bmhead{Last-iterate vs.~best-iterate guarantees}
The last iterate generated by~\cref{alg:velocityC} is differentiable with respect to the initial point.
However, the convergence guarantee in~\cref{thm:rateC} applies to the best point among the weighted averages of the iterates, rather than to the last iterate itself.
This gap reflects a broader challenge in nonconvex optimization: establishing last-iterate stationarity rates comparable to best-iterate rates is generally difficult, even for gradient descent.
In particular, a decrease in the objective value does not imply a decrease in the gradient norm, and standard descent arguments do not directly yield comparable last-iterate guarantees.
Investigating whether such guarantees can be obtained for a differentiable terminal output, possibly under additional structural assumptions, is an important direction for future research.

\bmhead{Optimal choice of $\alpha$ and $r$}
The velocity-control mechanism has two main parameters, $\alpha$ and $r$.
In the following numerical experiments, we observe that $\alpha = 0.1$ and $r = 0.5$ work well.
By minimizing the constant factors in the convergence rates in \cref{thm:rateC} and \cref{thm:rate}, one can obtain reasonable parameter choices, although these constants are not tight.
How to choose these parameters optimally remains to be investigated.

\bmhead{Cost of adapting $L$}
In \cref{alg:velocity}, the Lipschitz constant $L$ is estimated so that the inequality used in the convergence analysis is guaranteed to hold for all possible next points $x_t^1$, $x_t^r$, and $x_t^0$.
This requires three function evaluations per backtracking trial, rather than one.
This adaptation can be implemented efficiently, for example, by checking the three inequalities in parallel or by terminating the remaining checks once any of the candidate points is found to violate the corresponding inequality.
In \cref{app:predictive}, we propose a more efficient adaptation method and demonstrate its practical effectiveness through numerical experiments.
However, there may still be room for further improving the adaptation procedure.

\section{Numerical experiments}\label{sec:exp}
In this section, we examine the numerical performance of the proposed algorithms.
In \cref{sec:expdiff}, we conduct experiments on \cref{alg:velocityC}, focusing on its performance as a differentiable algorithm, and in \cref{sec:exppf}, we compare \cref{alg:velocity} with existing parameter-free methods on artificial nonconvex and convex problems.

The algorithms were implemented in Python 3.14.7 using JAX 0.11.1. All experiments were run on a Mac equipped with an Apple M3 processor (8 CPU cores) and 24 GB of unified memory, running macOS 26.6.2. 

\subsection{Differentiable algorithm}\label{sec:expdiff}
We compare \cref{alg:velocityC} ($\mathtt{VC}$), which is a differentiable algorithm, with vanilla gradient descent ($\mathtt{GD}$) and Nesterov's accelerated gradient method ($\mathtt{NAG}$)~\cite{N83}
\[
    y_{k} = x_k + \frac{k-1}{k+2}(x_{k} - x_{k-1}), \qquad x_{k+1} = y_k - \frac1L\nabla f(y_{k}),
\]
which are also differentiable algorithms.
Although the standard convex convergence guarantee for $\mathtt{NAG}$ does not apply in this nonconvex setting, we use it as a reference method.

\subsubsection{Minimizing an artificial problem}
First, we compare their performance in approximating the derivative of a solution with respect to a hyperparameter of the objective function.
In this section, we use the vectorized notation $\cos x = (\cos x_1 ,\ldots,\cos x_d)^\top$ and $\sin x = (\sin x_1 ,\ldots,\sin x_d)^\top$.

We consider the following family of minimization problems.
Let $\lambda \in \RR$ be a hyperparameter, and define the objective function $f_\lambda \colon \RR^d \to \RR$ by
\[
    f_\lambda(x) = \frac12 x^\top A x + \rho^\top (\mathbbm{1} - \cos x) + \lambda b^\top x,
\]
where $A = \mathrm{diag}(a_1,\ldots, a_d) \in \RR^{d\times d}$ with $A \succ O$, $\rho \in \RR_{\ge0}^d$, and $b \in \RR^d$.
Since $\nabla f_0(0) = 0$ and $\nabla^2 f_0(0) = A + \mathrm{diag}(\rho) \succ O$, the origin $0$ is a strict local minimizer of $f_0$.
For this function, $L = \max_i (a_i + \rho_i)$, $M = \max_i \rho_i$.

For this problem, we consider the derivative of the minimizer with respect to $\lambda$.
Let $x^\star(\lambda)$ be the smooth function\footnote{The existence of such a function follows from the implicit function theorem. Indeed, let $F(x, \lambda) = \nabla f_\lambda(x)$. Then, $F(0,0) = 0$ and $(\partial F/ \partial x)(0,0) = \nabla^2 f_0 (0) = A + \mathrm{diag}(\rho) \succ O$.
} defined for $\lambda$ near $0$ such that $\nabla f_\lambda(x^\star(\lambda)) = 0$ and $x^\star(0) = 0$, i.e., the local minimizer of $f_\lambda$ obtained by perturbing the local minimizer at the origin. Then, we have\footnote{This follows from the following calculation. By differentiating $0 = \nabla f_\lambda(x^\star(\lambda))$, we have $0 = \partial_\lambda \nabla f_\lambda(x)\mid_{x =x^\star(\lambda)} +  \nabla^2 f_\lambda(x^\star(\lambda))\dv{x^\star}{\lambda} = b + (A + \mathrm{diag}(\rho\odot \cos(x^\star(\lambda))))\dv{x^\star}{\lambda}$. Thus, by substituting $\lambda =0$, we have $b + (A + \mathrm{diag}(\rho))\dv{x^\star}{\lambda}(0)=0$}
\[
D^\star \coloneqq \dv{x^\star}{\lambda}(0) = \paren*{-\frac{b_1}{a_1 + \rho_1},\ldots,-\frac{b_d}{a_d + \rho_d}}^\top.
\]
Letting $x_t(\lambda)$ be the output at the $t$-th iteration of each minimization algorithm applied to $f_\lambda$, we approximate $D^\star$ by
\[
    D_t \coloneqq \dv{x_t}{\lambda}(0),
\]
using automatic differentiation.

We compare the convergence of the relative error in the Jacobian $\norm{D_t - D^\star}/\norm{D^\star}$ among the three differentiable algorithms $\mathtt{GD}$, $\mathtt{NAG}$, and $\mathtt{VC}$. The results are shown in \cref{fig:diff1}.
$\mathtt{VC}$ decreases the Jacobian relative error fastest during the first few iterations, before $\mathtt{NAG}$ outperforms it for a certain period. However, $\mathtt{VC}$ eventually becomes the fastest method by a substantial margin.
The same trend is observed for the distance to the local minimum.

\begin{figure}
    \centering
    \includegraphics[width=0.49\linewidth]{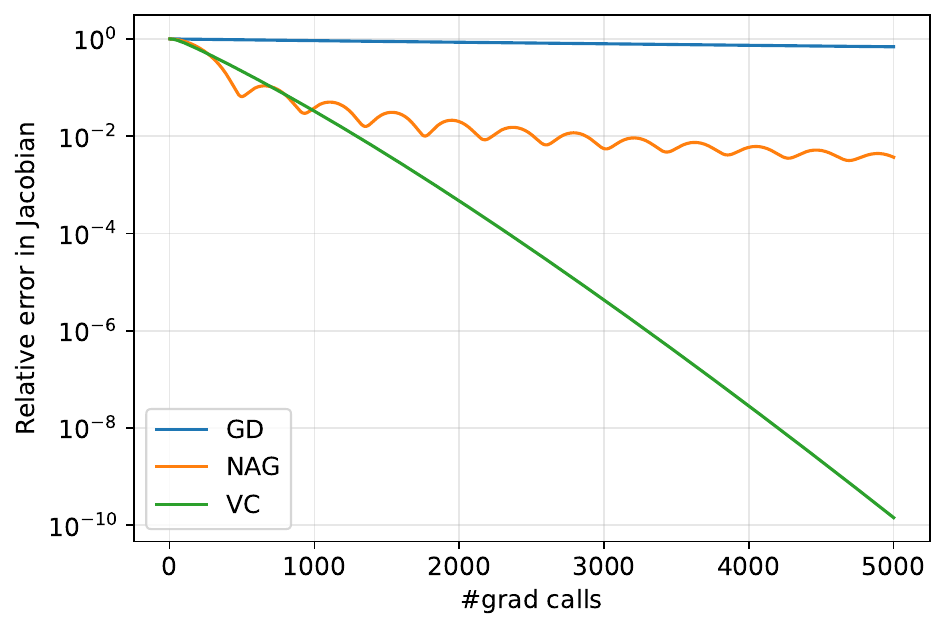}
    \includegraphics[width=0.49\linewidth]{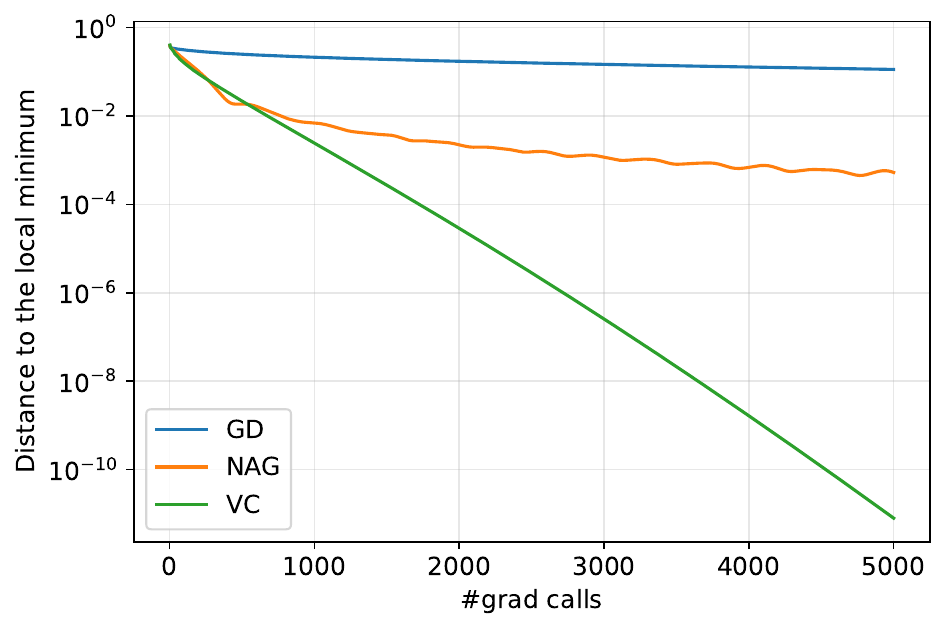}
    \caption{Convergence of the relative error in the Jacobian $\norm{D_t - D^\star}/\norm{D^\star}$ and the distance to the local minimum $\norm{x_t - 0}$. The parameters of $\mathtt{VC}$ are set to $\alpha = 0.1$ and $r = 0.5$.}
    \label{fig:diff1}
\end{figure}

\subsubsection{Meta-learning}
Next, we compare the performance of the differentiable algorithms in the model-agnostic meta-learning framework proposed in~\cite{FAL17}.
In this meta-learning framework, we optimize the initial parameters of machine learning models so that the loss after a few training iterations starting from the initial parameters becomes small, thereby enabling the models to adapt quickly to new tasks.

In this experiment, we consider sinusoidal regression tasks.
For each task $\mathcal{D}$, the target function
$g_{\mathcal{D}} \colon \RR \to \RR$ is defined by
\[
    g_{\mathcal{D}}(x)
    \coloneqq a_{\mathcal{D}}\sin(x+\phi_{\mathcal{D}}),
\]
where the task-specific parameters $a_{\mathcal{D}}$ and
$\phi_{\mathcal{D}}$ are independently sampled as
$a_{\mathcal{D}}\sim\mathrm{Unif}([0.1,5])$ and
$\phi_{\mathcal{D}}\sim\mathrm{Unif}([0,\pi])$.
For each task, we generate two regression datasets,
$\mathcal{D}_{\mathrm{inner}}$ and $\mathcal{D}_{\mathrm{meta}}$,
each containing $10$ input-output pairs $(x,y)$ with
$x\sim\mathrm{Unif}([-5,5])$ and $y=g_{\mathcal{D}}(x)$.

The prediction model is a neural network with one hidden layer of
$40$ units and the hyperbolic tangent activation function:
\[
    f_\theta(x)
    = b_2 + \sum_{i=1}^{40}
      w_{2,i}\tanh(w_{1,i}x+b_{1,i}),
\]
where $\theta=(w_1,b_1,w_2,b_2)\in\RR^{121}$ collects all trainable
weights and biases.

For each task $\mathcal{D}$, we consider the regression problem
\begin{equation}
    \minimize_{\theta\in\RR^{121}} \quad
    \frac{1}{2|\mathcal{D}_{\mathrm{inner}}|} \sum_{(x,y)\in\mathcal{D}_{\mathrm{inner}}} \abs{f_\theta(x)-y}^2,
    \label{innerreg}
\end{equation}
and let
$\theta^{\mathtt{alg}}_{\mathcal{D}}(\theta^{\mathrm{init}})$
denote the output of the inner optimization algorithm
$\mathtt{alg}$ applied to this problem from the initial point
$\theta^{\mathrm{init}}$.
In this meta-learning framework, we optimize the initialization
$\theta^{\mathrm{init}}$ by solving
\[
    \minimize_{\theta^{\mathrm{init}}\in\RR^{121}} \quad
    \mathbb{E}_{\mathcal{D}}\left[
        \frac{1}{2|\mathcal{D}_{\mathrm{meta}}|} \sum_{(x,y)\in\mathcal{D}_{\mathrm{meta}}} \abs{f_{\theta^{\mathtt{alg}}_{\mathcal{D}}(\theta^{\mathrm{init}})}(x)-y}^2
    \right].
\]
At each meta-learning iteration, we approximate this expectation by
\begin{equation}
    \frac{1}{|\mathcal{B}|}
    \sum_{\mathcal{D}\in\mathcal{B}}
    \paren*{
        \frac{1}{2|\mathcal{D}_{\mathrm{meta}}|}\sum_{(x,y)\in\mathcal{D}_{\mathrm{meta}}}\abs{f_{\theta^{\mathtt{alg}}_{\mathcal{D}} (\theta^{\mathrm{init}})}(x)-y}^2
    },
    \label{metaloss}
\end{equation}
where $\mathcal{B}$ is a batch of $16$ sampled tasks.
Using a differentiable inner algorithm makes this batch
meta-objective differentiable with respect to
$\theta^{\mathrm{init}}$.
We compute its gradient by automatic differentiation through
the inner optimization steps.

In this experiment, we use the three differentiable algorithms $\mathtt{GD}$, $\mathtt{NAG}$, and $\mathtt{VC}$ for $\mathtt{alg}$, and set their outputs to the iterates after $5$ or $20$ iterations.
The fixed parameters $L$ and $M$ are chosen from $L \in \setE{30, 100, 300}$ and $M \in \setE{30, 100, 300}$.
As the optimizer for the meta-loss~\eqref{metaloss}, we use Adam~\cite{KB15}.

The results are shown in~\cref{fig:maml}.
For the 5-iteration setting, $\mathtt{VC}$ with the best choice of $L$ and $M$ achieves a lower meta-loss after $40$ seconds than the other methods with their best choices of $L$.
Each inner iteration of $\mathtt{VC}$ is more computationally expensive than that of $\mathtt{GD}$ or $\mathtt{NAG}$, resulting in more time per outer optimization iteration.
This additional cost may explain the slower initial decrease in the meta-loss with respect to elapsed time.
Nevertheless, $\mathtt{VC}$ subsequently overtakes the other methods.
Moreover, when the methods are compared at the same $L$, $\mathtt{VC}$ tends to outperform the other methods for every tested value of $M$, except when $L=300$.

In the 20-iteration setting, all methods attain lower final meta-losses than in the 5-iteration setting.
With its best parameter choices, $\mathtt{VC}$ still achieves the lowest final meta-loss after $80$ seconds, although its advantage over $\mathtt{NAG}$ is smaller than in the 5-iteration setting.
However, when the methods are compared at the same $L$, $\mathtt{VC}$ no longer consistently outperforms the other methods.

\begin{figure}
    \centering
    \includegraphics[width=0.99\linewidth]{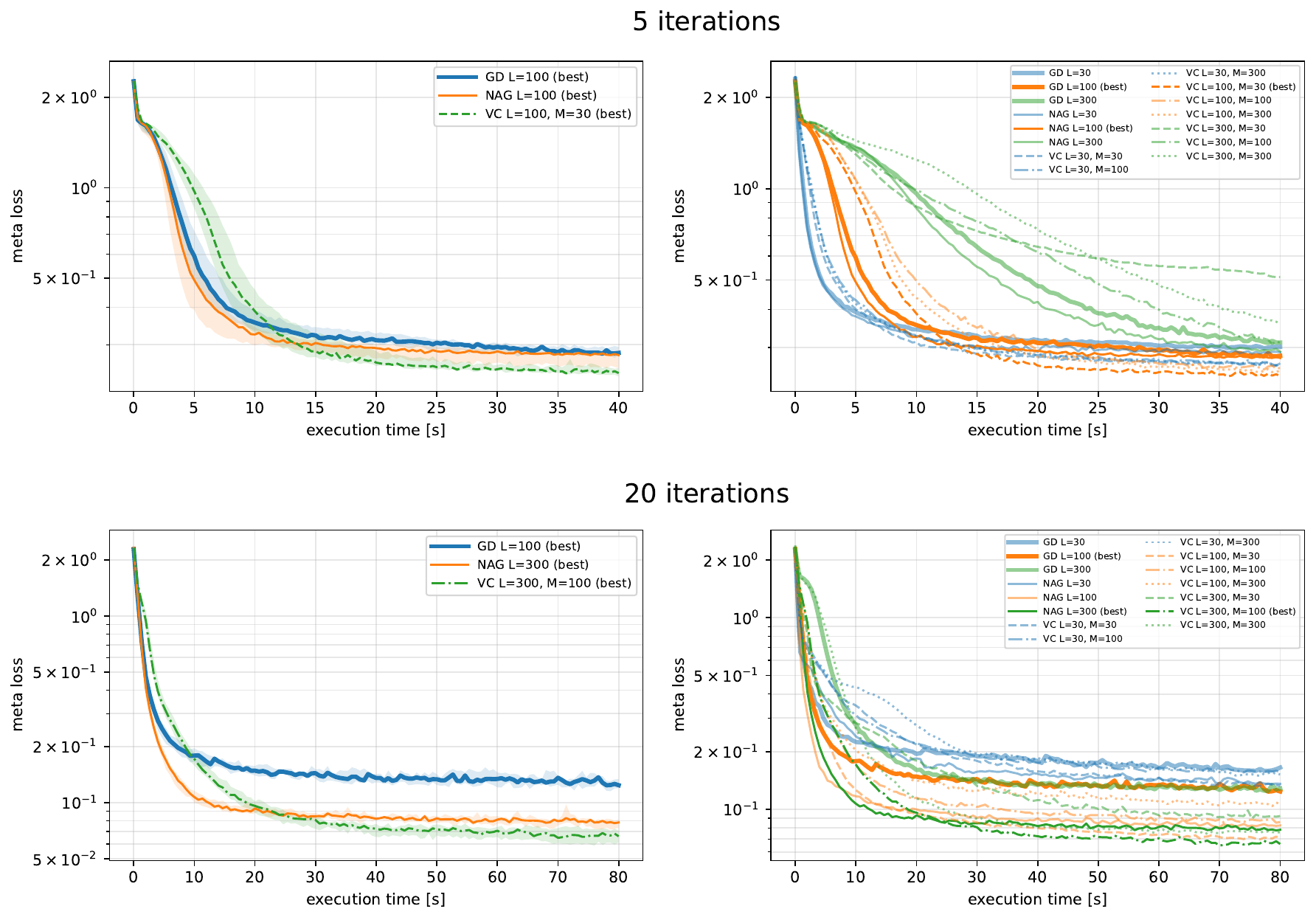}
    \caption{Convergence of the meta-loss, computed using the expression in~\eqref{metaloss} with $\mathcal{B}$ replaced by a separate set of evaluation tasks fixed throughout each run.
    The curves show the medians over $10$ runs.
    The left panels show the results obtained with the best parameter choices, together with the 25th--75th percentiles, while the right panels show all tested parameter settings.
    For $\mathtt{VC}$, the parameters are fixed at $\alpha=0.1$ and $r=0.5$.}
    \label{fig:maml}
\end{figure}

\subsection{Parameter-free algorithm}\label{sec:exppf}
In this section, we numerically compare the proposed parameter-free method, \cref{alg:velocity} ($\mathtt{VC}$), with existing parameter-free methods for nonconvex minimization: adaptive gradient descent ($\mathtt{GD}$), the accelerated gradient method with restart ($\mathtt{AGD}$)~\cite{MT24}, and the heavy-ball method with restart ($\mathtt{HB}$)~\cite{MT24b}.
During the adaptation of $L$ in $\mathtt{VC}$, the function values $f(x^1_t)$, $f(x^r_t)$, and $f(x^0_t)$ are computed in parallel.
In addition to these methods, we compare $\mathtt{Predictive VC}$, a variant of $\mathtt{VC}$ that reduces the cost of estimating $L$ while maintaining the convergence guarantee.
The algorithm is detailed in \cref{app:predictive}, but the main modification is that, if $\hat{r}_t = 1$ persists for a certain period, $L$ is estimated using only $f(x^1_t)$, without evaluating $f(x^r_t)$ and $f(x^0_t)$.
We evaluate these algorithms on five test functions: the Rosenbrock function, Dixon--Price function, Powell function, Qing function, and a quadratic function.
The first four functions are taken from~\cite{JY13}. Although these functions do not satisfy the globally Lipschitz assumptions of our theory, we use them to investigate empirical behavior.

All methods are run for 50 seconds, except for the Qing function, for which they are run for 15 seconds.
We evaluate the decrease in the objective function value against three measures of computational cost: the number of gradient calls, the number of oracle calls (counting each gradient-only, function-value-only, or joint function-and-gradient evaluation as one call), and cumulative execution time. For each method, all five runs follow the same optimization trajectory, and the time axis shows the median cumulative execution time at corresponding points along that trajectory.
In these plots, the parameters of $\mathtt{VC}$ are fixed at the recommended values $\alpha = 0.1$ and $r = 0.5$.
To investigate the effects of $\alpha$ and $r$, we also present separate plots comparing different choices of these parameters.

\begin{itemize}
    \item Rosenbrock function: $d = 10^6$,
    \[
        \sum_{i=1}^{d-1} \left((1 - x_i)^2 + 100(x_{i+1} - x_i^2)^2\right).
    \]
    This function is nonconvex, with a unique minimizer.
    The result is shown in \cref{fig:RB}.
    The proposed methods ($\mathtt{VC}$ and $\mathtt{Predictive VC}$) outperform the existing methods with respect to all three measures.
    $\mathtt{Predictive VC}$ requires fewer oracle calls than $\mathtt{VC}$, leading to a shorter execution time.
    Velocity control occurs during the initial slow-convergence phase but no longer occurs after the method enters the fast-convergence phase.
    The comparison of different values of $\alpha$ and $r$ indicates that the convergence behavior is more sensitive to $\alpha$ than to $r$.
    The recommended choice $\alpha = 0.1$ and $r = 0.5$ is not the best among the tested parameter settings, but performs well.
    
    \item Dixon--Price function: $d = 10^6$,
    \[
        (x_1 - 1)^2 + \sum_{i=2}^{d} i(2x_i^2 - x_{i-1})^2.
    \]
    This function is nonconvex, with global minimum value zero.
    The result is shown in \cref{fig:DP}.
    $\mathtt{HB}$ performs best, followed by $\mathtt{VC}$ and $\mathtt{Predictive VC}$, although the objective function value still remains at approximately $10^8$ for all methods.
    Velocity control occurs frequently throughout the execution.
    For $\mathtt{VC}$, all tested choices of $\alpha$ and $r$ show similar performance.

    \item Powell function: $d = 10^6$,
    \[
        \sum_{i=1}^{d/4} \paren*{
        (x_{4i-3} + 10x_{4i-2})^2
        + 5(x_{4i-1} - x_{4i})^2
        + (x_{4i-2} - 2x_{4i-1})^4
        + 10(x_{4i-3} - x_{4i})^4}.
    \]
    This function is convex, with the origin as its unique minimizer.
    For this function, we additionally compare Nesterov's accelerated gradient method for convex functions ($\mathtt{NAG\mathchar`- C}$)~\cite{N18b}.
    The result is shown in \cref{fig:Pow}.
    $\mathtt{HB}$ performs best among all the methods.
    $\mathtt{Predictive VC}$ performs second best in the initial regime, whereas NAG-C exhibits a faster asymptotic decrease than $\mathtt{Predictive VC}$.
    Regarding the parameters of $\mathtt{VC}$, the choice of $\alpha$ has a substantial effect on the performance.
    The best performance is obtained with $\alpha = 0.05$, although $\alpha = 0.1$ performs nearly as well.

    \item Qing function: $d = 10^5$,
    \[
        \sum_{i=1}^{d} (x_i^2 - i)^2.
    \]
    This function is nonconvex.
    For the Qing function, we set the dimension to $d = 10^5$ so that the method can reach a solution within a reasonable amount of time.
    The result is shown in \cref{fig:Qing}.
    $\mathtt{Predictive VC}$ performs best, substantially reducing the number of oracle calls compared with $\mathtt{VC}$.
    For $\mathtt{VC}$, the performance depends strongly on the parameter choice.
    The recommended parameter setting performs well in the initial phase, whereas $\alpha = 0.05$ and $r = 0.25$ eventually exhibit the fastest convergence.

    \item Strongly convex quadratic function: $d = 10^6$,
    \[
        \frac{1}{2}\sum_{i=1}^{d}\left(1+\frac{i-1}{d-1}(L-1)\right)(x_i-1)^2,
    \]
    where $L=10^4$.
    This function is strongly convex with condition number $\kappa = 10^4$.
    For this function, we additionally compare Nesterov's accelerated gradient method for strongly convex functions ($\mathtt{NAG\mathchar`-SC}$)~\cite{N18b}, for which the strong convexity parameter $\mu = 1$ is given.
    The result is shown in \cref{fig:Quad}.
    $\mathtt{Predictive VC}$ performs best, despite not using the strong convexity parameter.
    For $\mathtt{VC}$, the recommended parameter setting performs best.
\end{itemize}

Overall, $\mathtt{Predictive VC}$ outperforms the existing methods on some test problems, whereas the original $\mathtt{VC}$ can be slower because of the computational overhead of adapting $L$.
The observed improvements over $\mathtt{AGD}$ and $\mathtt{HB}$ may be partly attributable to two features of the proposed methods.
First, $\mathtt{AGD}$ and $\mathtt{HB}$ rely on restart strategies, whereas the proposed methods employ velocity control without resetting the velocity to zero.
Second, the Hessian-smoothness estimates used by $\mathtt{AGD}$ and $\mathtt{HB}$ are nondecreasing between consecutive restarts, whereas the proposed methods update $M_t$ at each iteration without imposing monotonicity.
The recommended parameter setting $\alpha=0.1$ and $r=0.5$ is not always the best among those tested, but yields performance close to that of the best tested setting across the problems considered.

\begin{figure}
    \centering
    Rosenbrock function\\
    \includegraphics[width=0.49\linewidth]{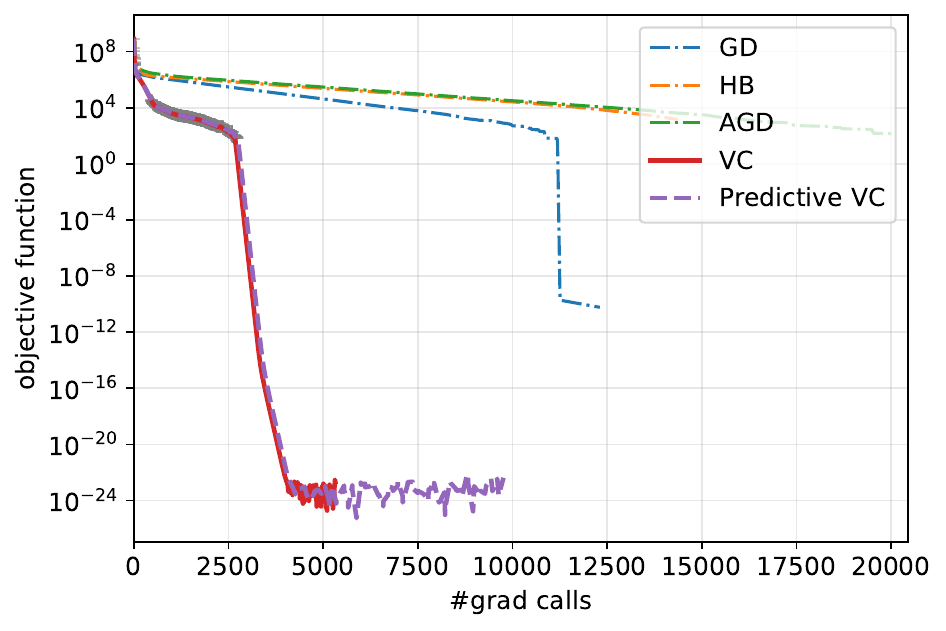}
    \includegraphics[width=0.49\linewidth]{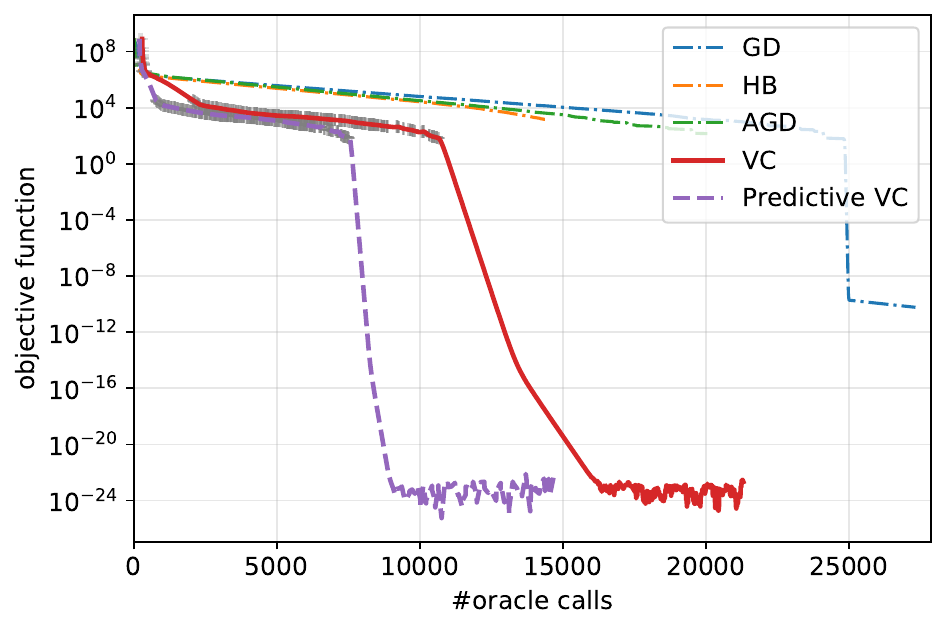}
    \includegraphics[width=0.49\linewidth]{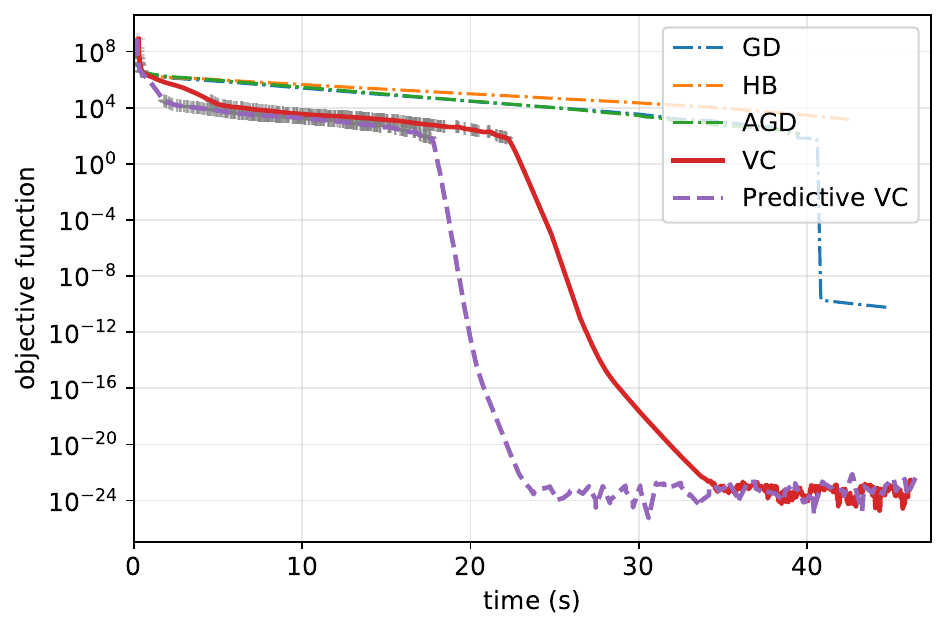}
    \includegraphics[width=0.49\linewidth]{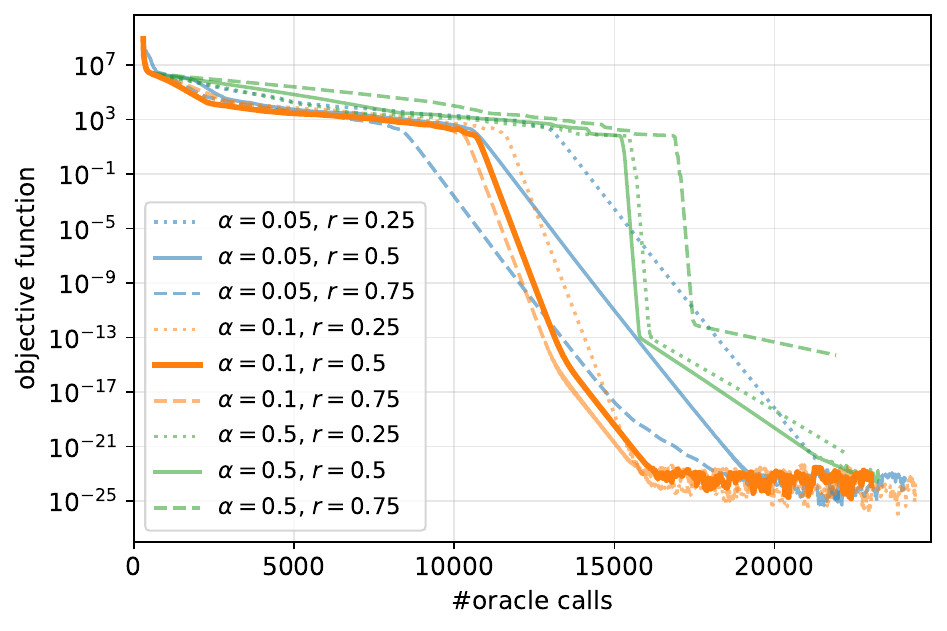}
    \caption{The convergence of $f(x_k) - f^\star$ on the Rosenbrock function.
    All panels except the lower-right panel compare the proposed and existing algorithms.
    For $\mathtt{VC}$ and $\mathtt{Predictive VC}$, the parameters are set to $\alpha = 0.1$, $r = 0.5$, $\hmax = 1$, $\beta_\mathrm{inc} = 1.1$, and $\beta_\mathrm{dec} = 0.9$.
    The lower-right panel compares VC with various choices of $\alpha$ and $r$, while fixing the other parameters at the values given above.}
    \label{fig:RB}
\end{figure}

\begin{figure}
    \centering
    Dixon--Price function\\
    \includegraphics[width=0.49\linewidth]{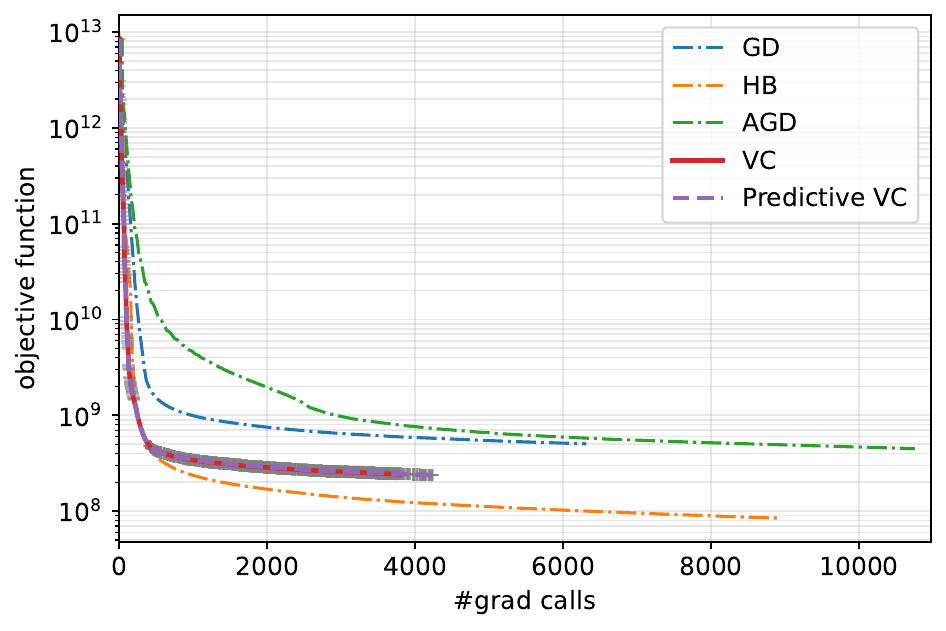}
    \includegraphics[width=0.49\linewidth]{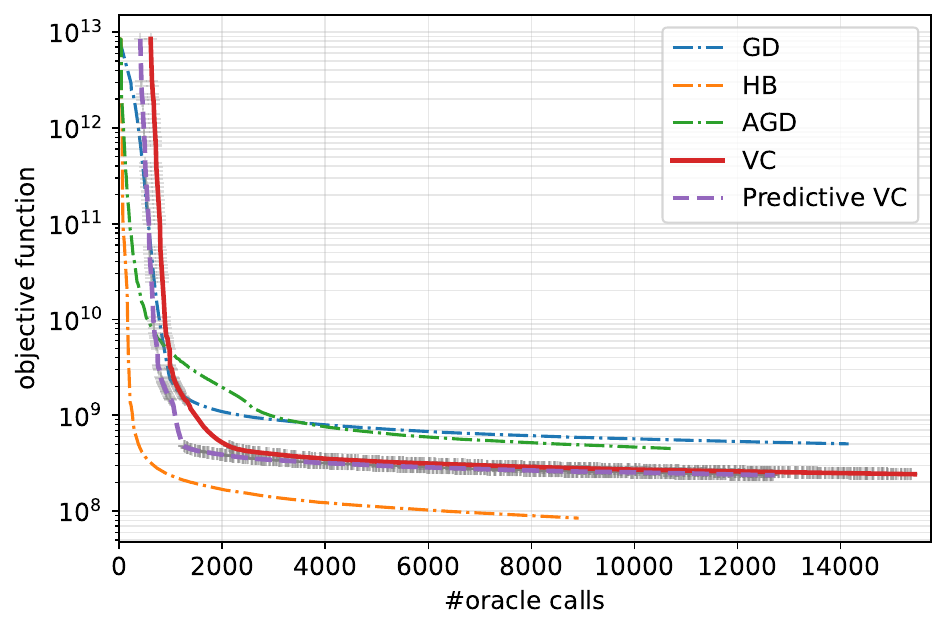}
    \includegraphics[width=0.49\linewidth]{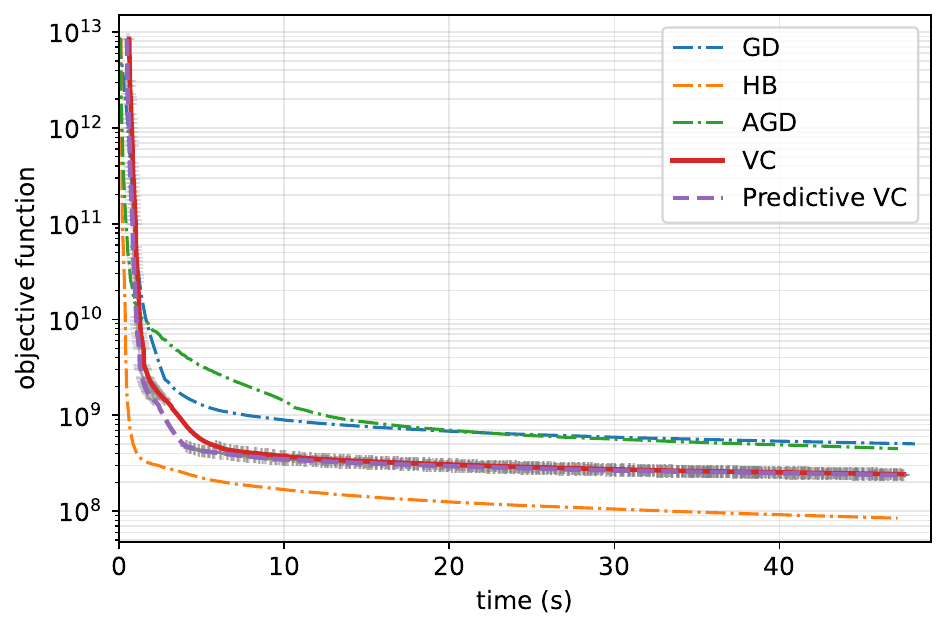}
    \includegraphics[width=0.49\linewidth]{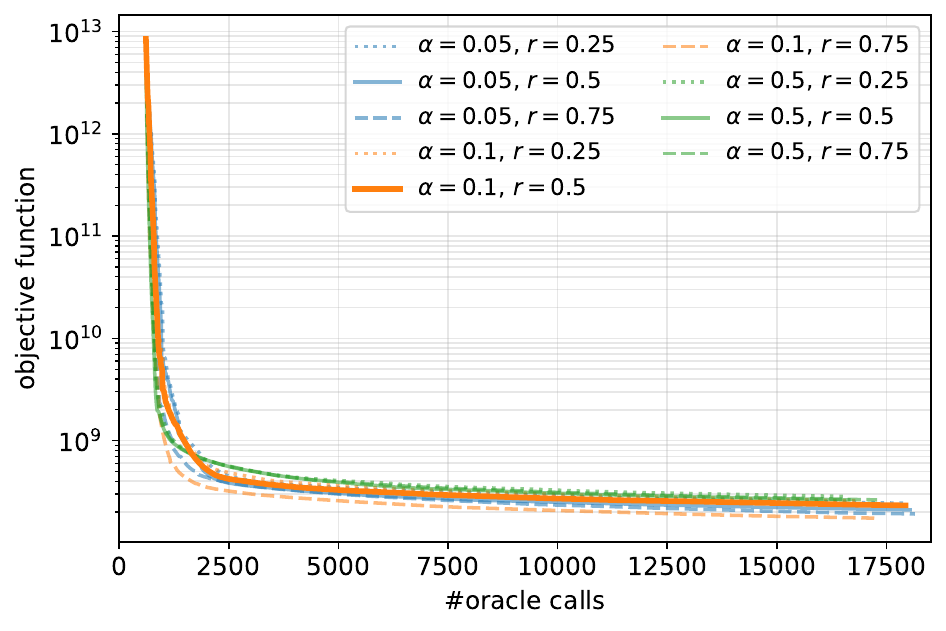}
    \caption{The convergence of $f(x_k) - f^\star$ on the Dixon--Price function.}
    \label{fig:DP}
\end{figure}

\begin{figure}
    \centering
    Powell function\\
    \includegraphics[width=0.49\linewidth]{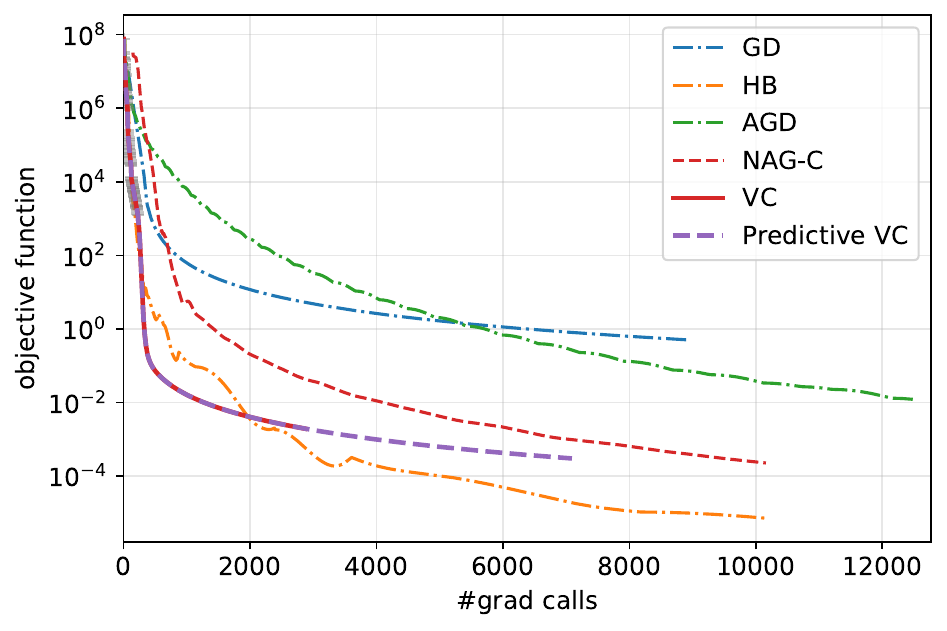}
    \includegraphics[width=0.49\linewidth]{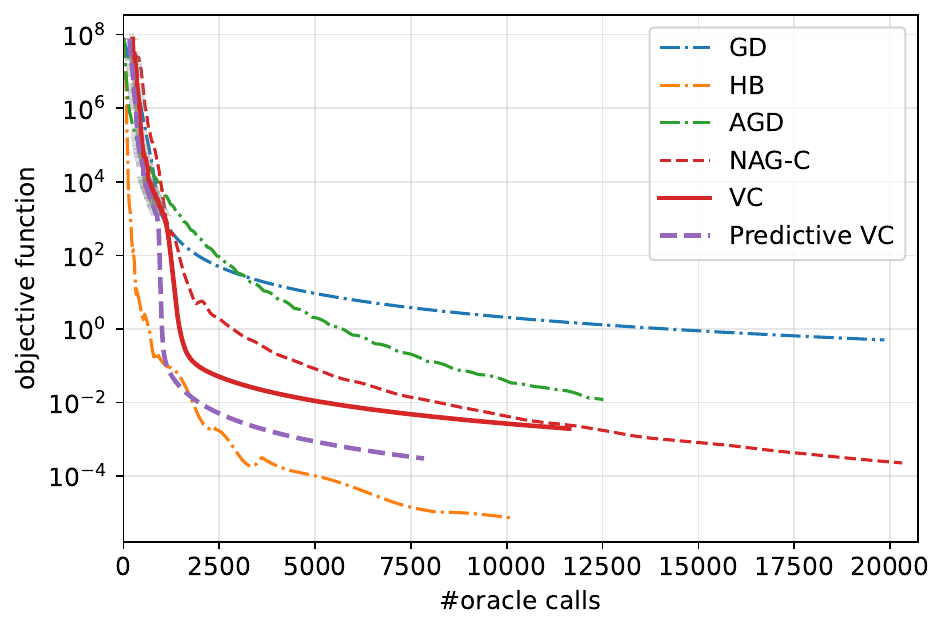}
    \includegraphics[width=0.49\linewidth]{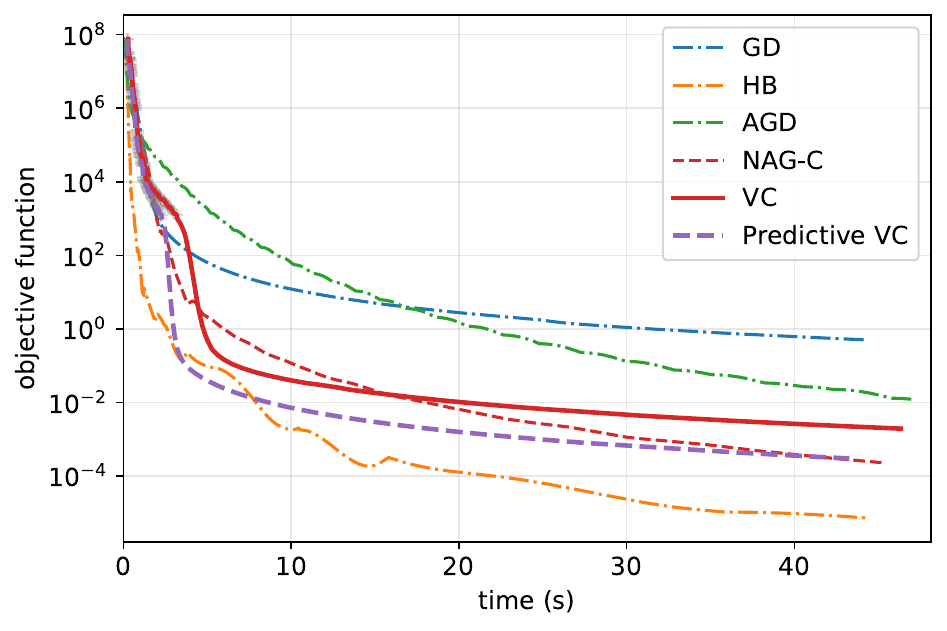}
    \includegraphics[width=0.49\linewidth]{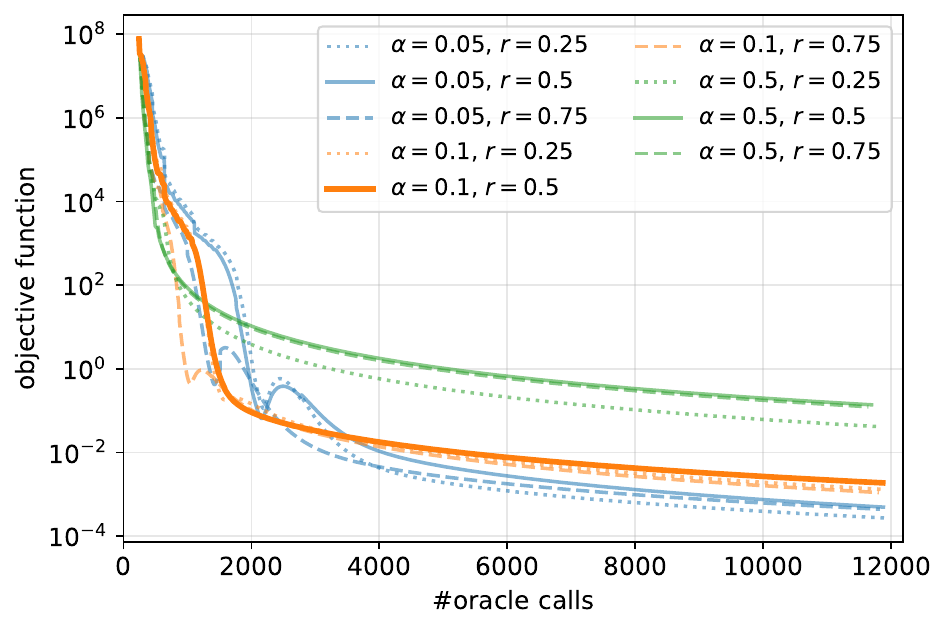}
    \caption{The convergence of $f(x_k) - f^\star$ on the Powell function. }
    \label{fig:Pow}
\end{figure}

\begin{figure}
    \centering
    Qing function\\
    \includegraphics[width=0.49\linewidth]{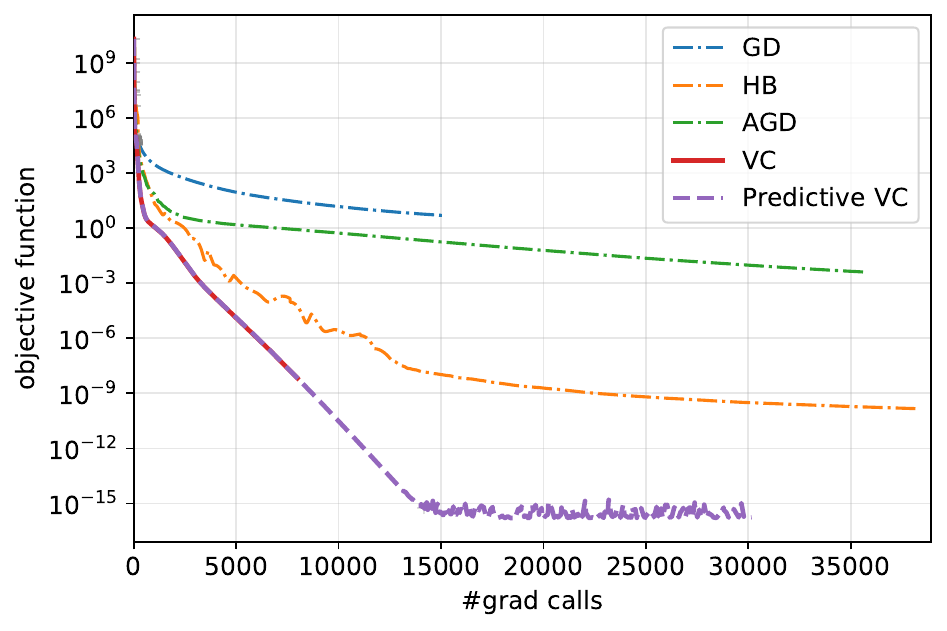}
    \includegraphics[width=0.49\linewidth]{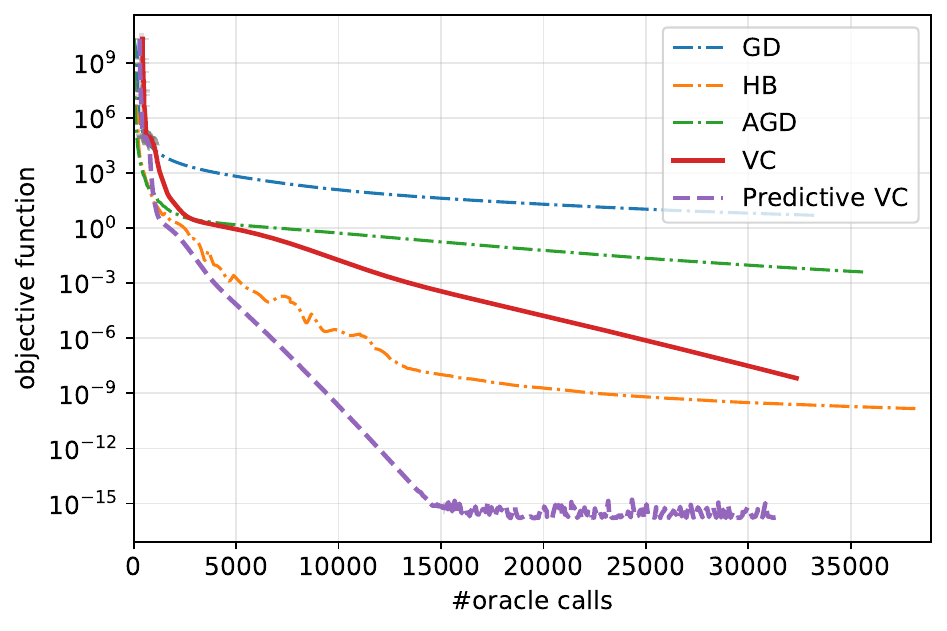}
    \includegraphics[width=0.49\linewidth]{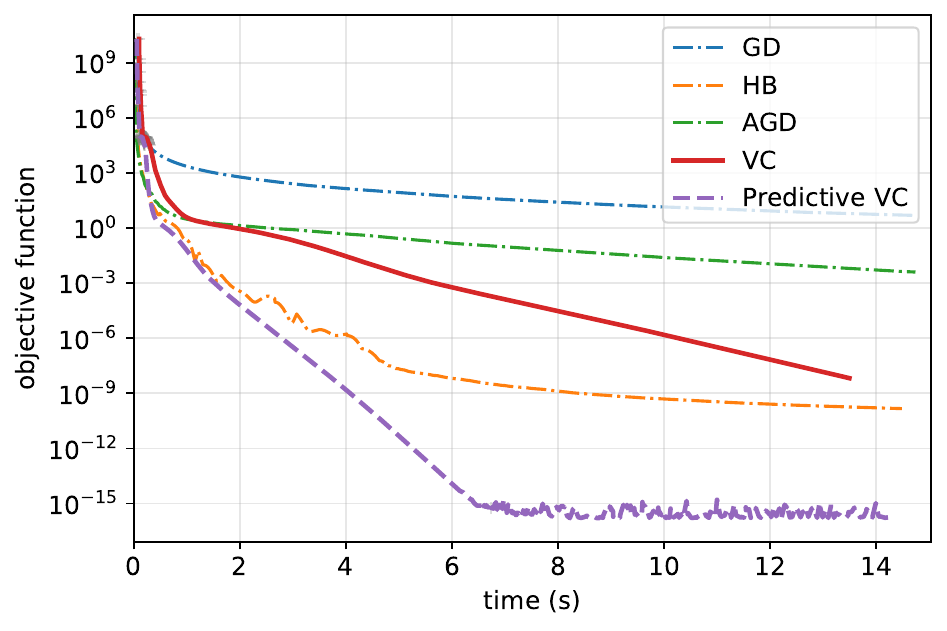}
    \includegraphics[width=0.49\linewidth]{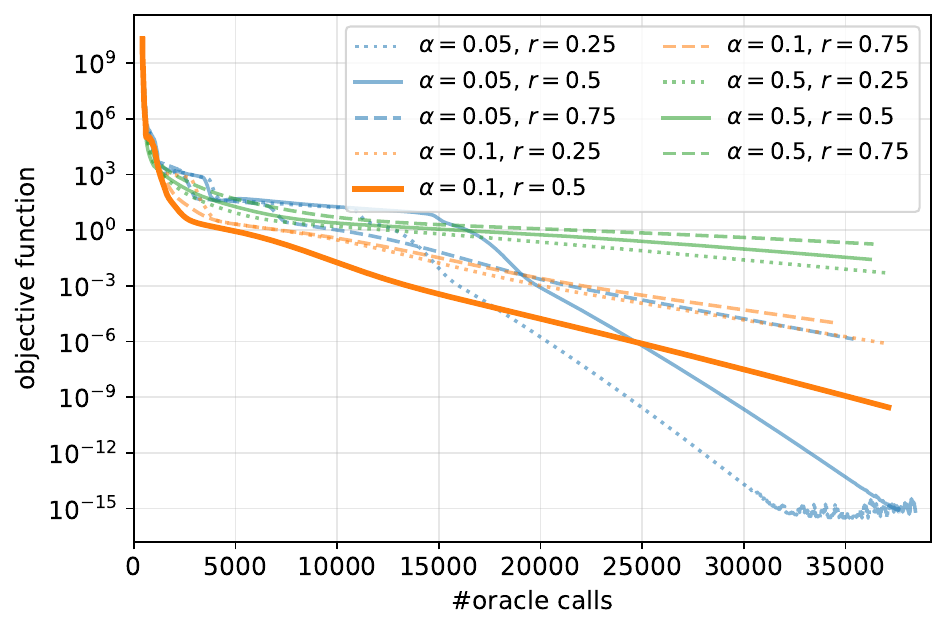}
    \caption{The convergence of $f(x_k) - f^\star$ on the Qing function.}
    \label{fig:Qing}
\end{figure}

\begin{figure}
    \centering
    Quadratic function\\
    \includegraphics[width=0.49\linewidth]{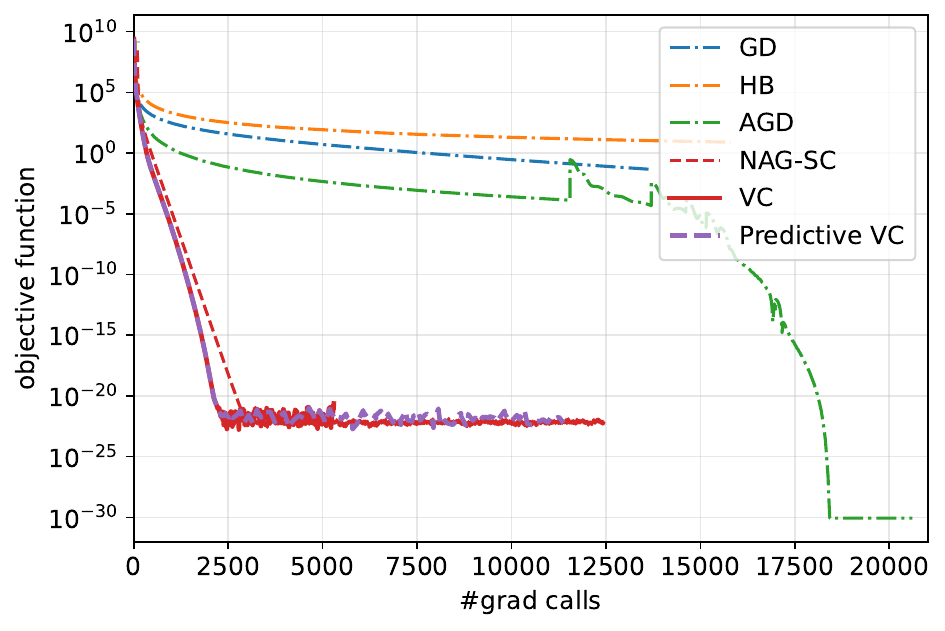}
    \includegraphics[width=0.49\linewidth]{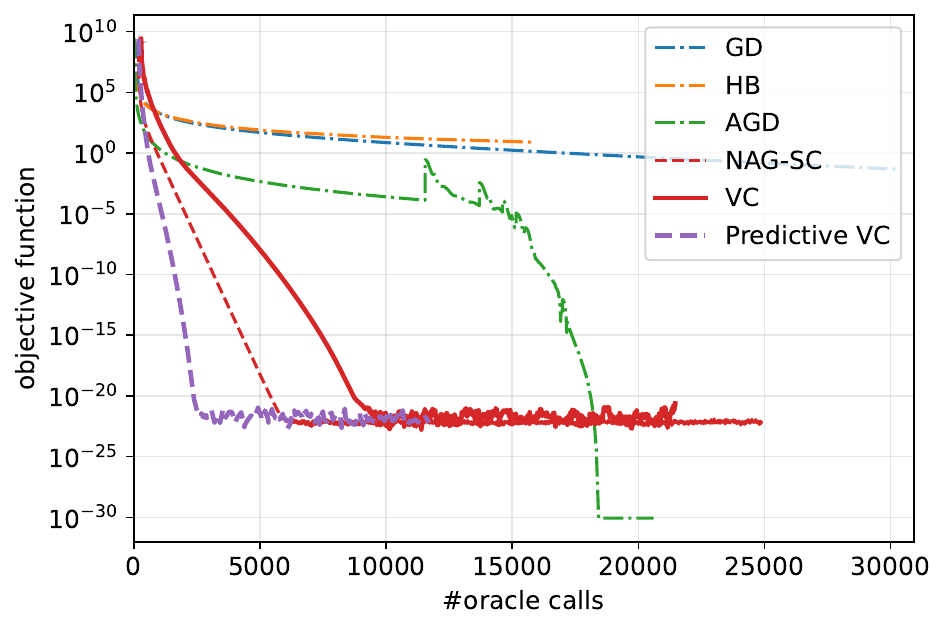}
    \includegraphics[width=0.49\linewidth]{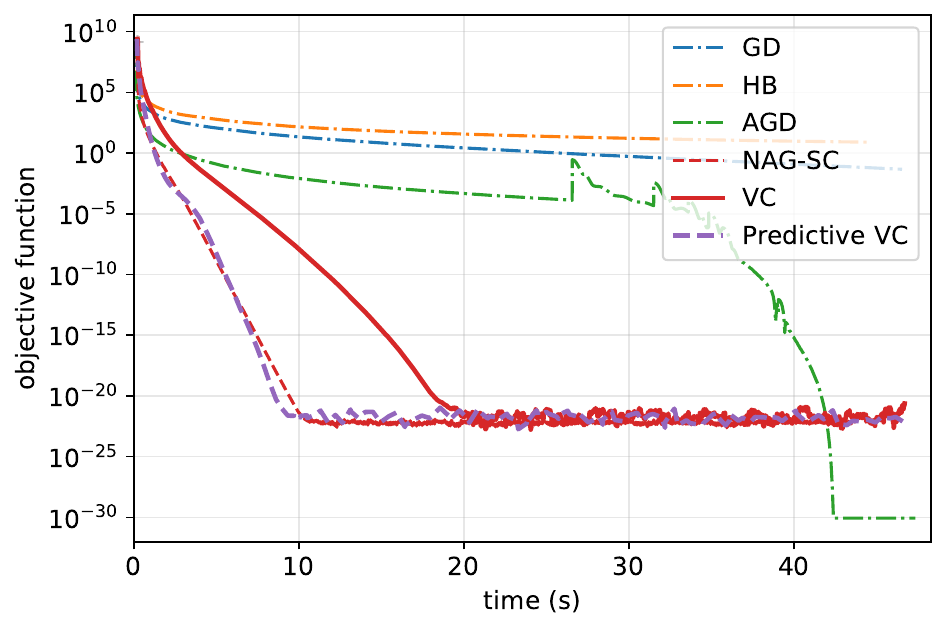}
    \includegraphics[width=0.49\linewidth]{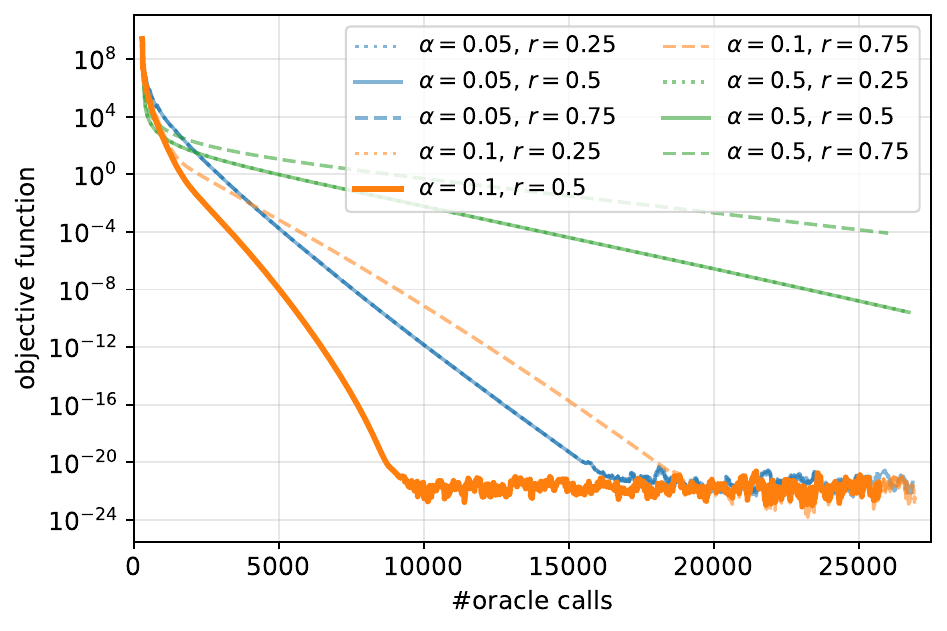}
    \caption{The convergence of $f(x_k) - f^\star$ on the quadratic function.}
    \label{fig:Quad}
\end{figure}

\section{Proof of \cref{lem:upp}}\label{sec:lemc}
The goal of this section is to prove \cref{lem:upp}.
The key ingredient is the following proposition, which is a straightforward generalization of~\cite[Lemma~1]{OMT26} and can be proved in the same way.
\begin{proposition}[cf.~\cite{OMT26}]\label{propac}
    Suppose that $f \colon \RR^d \to \RR$ is twice continuously differentiable and that $\nabla^2 f$ is $M$-Lipschitz continuous.
    Fix $t \in \RR_{> 0}$, and let $w_{t} \colon \RR_{\ge 0} \to \RR_{\ge 0}$ satisfy $\int_{t/2}^t w_{t}(\tau) \dd \tau = 1$.
    Define
    \begin{equation}
        F_{t}(\tau) \coloneqq \int_{t/2}^\tau \dd s_1 \int_\tau^t \dd s_2\,\, w_{t}(s_1) w_{t}(s_2) (s_2-s_1) \quad \text{for} \quad \tau\in[t/2,t]
    \end{equation}
    and $F_t(\tau)\coloneqq0$ for $\tau\notin[t/2,t]$.
    Then, for any continuously differentiable $x \colon \RR_{\ge 0} \to \RR^d$,
    \begin{equation}
        \norm*{\nabla f \paren*{\int_{t/2}^t w_{t}(\tau)x(\tau) \dd \tau}} \le \norm*{\int_{t/2}^t w_{t}(\tau) \nabla f(x(\tau)) \dd \tau} + \frac{M}{2}\int_{t/2}^t F_{t}(\tau) \norm{\dot{x}(\tau)}^2 \dd \tau. \label{hesse}
    \end{equation}
\end{proposition}

\cref{lem:upp} is obtained by evaluating
$\int_{\varepsilon_0}^T (\text{RHS of~\eqref{hesse}}) \dd t$.
More precisely, we substitute the ODE~\eqref{ODE} into the gradient term and the definition of the weight $w_t$ into $F_t$.
The following lemmas provide sufficiently tight estimates for the coefficients of
$\norm{\dot{x}(t)}$ and $\norm{\dot{x}(t)}^2$ appearing in \cref{lem:upp}, which are needed to derive the desired convergence rate.

Let
\[
    \gamma(t) \coloneqq \int_0^t \e^{\alpha s^{6/7}} \dd s,
    \qquad
    \gamma_{\frac12}(t) \coloneqq \int_{t/2}^t \e^{\alpha s^{6/7}} \dd s.
\]
Then the weight can be written as
$w_{t}(\tau) = {\dot{\gamma}(\tau)}/{\gamma_{\frac12}(t)}$.
Note that, since the lower limit $t/2$ depends on $t$,
$w_{t}(\tau) \neq {\dot{\gamma}_{\frac12}(\tau)}/{\gamma_{\frac12}(t)}$.

\begin{lemma}\label{lem:ac1}
    The function
    \[
        \frac{1}{\gamma(t)}\int_0^t \gamma(s) \dd s
    \]
    is nondecreasing in $t$.
\end{lemma}
\begin{proof}
    To show
    \[
        \dv{}{t} \paren*{\frac{1}{\gamma(t)}\int_0^t \gamma(s) \dd s} \ge 0,
    \]
    note that
    \begin{equation}
        \dv{}{t} \paren*{\frac{1}{\gamma(t)}\int_0^t \gamma(s) \dd s} = \frac{\gamma(t)^2 - \dot{\gamma}(t)\int_0^t \gamma(s)\dd s}{\gamma(t)^2}.
    \end{equation}
    It therefore suffices to show
    $G_1(t) \coloneqq \gamma(t)^2 - \dot{\gamma}(t)\int_0^t \gamma(s)\dd s \ge 0$ for $t \ge 0$.
    Since $G_1(0) = 0$, it is enough to show $\dv{G_1}{t}(t) \ge 0$ for $t \ge 0$.
    In view of $\ddot{\gamma}(t) = \frac{6\alpha}{7t^{1/7}} \dot{\gamma}(t)$, we have
    \begin{align}
        \dv{G_1}{t}(t) &= \gamma(t)\dot{\gamma}(t) - \ddot{\gamma}(t)\int_0^t \gamma(s)\dd s
        = \frac{6}{7}\frac{\alpha}{t^{1/7}}\dot{\gamma}(t)\paren*{\frac{7}{6\alpha}t^{1/7}\gamma(t) - \int_0^t \gamma(s)\dd s}.
    \end{align}
    Since $\dot{\gamma}(t) > 0$, it suffices to show
    \[
        G_2(t) \coloneqq \frac{7}{6\alpha}t^{1/7}\gamma(t) - \int_0^t \gamma(s)\dd s \ge 0.
    \]
    This follows from $G_2(0) = 0$ and
    \begin{align}
        \dv{G_2}{t}(t) &= \frac{1}{6\alpha}t^{-6/7}\gamma(t) + \frac{7}{6\alpha}t^{1/7}\dot{\gamma}(t) - \gamma(t)\\
        &= \frac{1}{6\alpha}t^{-6/7}\gamma(t) + \frac{7}{6\alpha}t^{1/7}\e^{\alpha t^{6/7}} - \int_0^t \e^{\alpha s^{6/7}} \dd s\\
        &= \frac{1}{6\alpha}t^{-6/7}\gamma(t) + \frac{1}{6\alpha}\int_0^t s^{-6/7}\e^{\alpha s^{6/7}} \dd s \ge 0,
    \end{align}
    where the last equality follows by integration by parts.
\end{proof}

\begin{lemma}\label{lem:ac2}
    For any $t >0$ and $t/2 \le \tau \le t$, it holds that
    \begin{equation}
        F_{t}(\tau) \le \frac{\gamma(\tau)\gamma(t)}{(\gamma_{\frac12}(t))^2}(t-\tau),
    \end{equation}
    where $F_{t}$ is defined in \cref{propac}.
\end{lemma}
\begin{proof}
    By the definition of $F_{t}(\tau)$,
    \[
        F_{t}(\tau) 
        = \frac{1}{(\gamma_{\frac12}(t))^2}\int_{t/2}^\tau \dd s_1 \int_\tau^t \dd s_2 \dot{\gamma}(s_1) \dot{\gamma}(s_2) (s_2-s_1),
    \]
    and we have
    \begin{align}
        \MoveEqLeft \int_{t/2}^\tau \dd s_1 \int_\tau^t \dd s_2 \dot{\gamma}(s_1) \dot{\gamma}(s_2) (s_2-s_1)\\
        &= \int_{t/2}^\tau \dot{\gamma}(s_1) \paren*{\sqparen*{\gamma(s_2)(s_2-s_1)}_\tau^t - \int_\tau^t \gamma(s_2)\dd s_2} \dd s_1\\
        &= \int_{t/2}^\tau \dot{\gamma}(s_1) \paren*{\gamma(t)(t-s_1) - \gamma(\tau)(\tau-s_1) - \int_\tau^t \gamma(s_2)\dd s_2} \dd s_1\\
        &= \sqparen*{\gamma(s_1) \paren*{\gamma(t)(t-s_1) - \gamma(\tau)(\tau-s_1) - \int_\tau^t \gamma(s_2)\dd s_2}}_{t/2}^\tau\\
        &\quad - \int_{t/2}^\tau \gamma(s_1) \paren*{-\gamma(t) + \gamma(\tau)} \dd s_1\\
        &= \gamma(\tau)\gamma(t)(t-\tau) - \gamma(t/2)\gamma(t)(t-t/2) + \gamma(t/2)\gamma(\tau)(\tau-t/2)\\
        &\quad - (\gamma(\tau)-\gamma(t/2))\int_\tau^t \gamma(s_2)\dd s_2 - (-\gamma(t)+\gamma(\tau))\int_{t/2}^\tau \gamma(s_1) \dd s_1\\
        &= \gamma(\tau)\gamma(t)(t-\tau) - \gamma(t/2)\paren*{\gamma(t)(t-t/2) - \gamma(\tau)(\tau-t/2) - \int_\tau^t \gamma(s)\dd s}\\
        &\quad - \paren*{\gamma(\tau)\int_{t/2}^t \gamma(s) \dd s - \gamma(t)\int_{t/2}^\tau \gamma(s) \dd s}\\
        &= \gamma(\tau)\gamma(t)(t-\tau) - \gamma(t/2)\paren*{\gamma(t)(t-t/2) - \int_{0}^t \gamma(s) \dd s - \gamma(\tau)(\tau-t/2) + \int_{0}^\tau \gamma(s) \dd s}\\
        &\quad - \gamma(\tau)\gamma(t)\paren*{\frac{1}{\gamma(t)}\int_0^t \gamma(s) \dd s - \frac{1}{\gamma(\tau)}\int_0^\tau \gamma(s) \dd s} - (\gamma(t)-\gamma(\tau))\int_0^{t/2} \gamma(s) \dd s.
    \end{align}
    We now show that the second, third, and fourth terms in the above expression are nonpositive, which implies the desired bound.
    In the following, note that $t/2 \le \tau \le t$.
    Since, for $t/2 \le u$,
    \begin{equation}
        \dv{}{u}\paren*{\gamma(u)(u-t/2) - \int_{0}^u \gamma(s) \dd s} = \dot{\gamma}(u)(u-t/2) \ge 0,
    \end{equation}
    we have
    \begin{equation}
        \gamma(t)(t-t/2) - \int_{0}^t \gamma(s) \dd s - \gamma(\tau)(\tau-t/2) + \int_{0}^\tau \gamma(s) \dd s \ge 0.
    \end{equation}
    In addition, from \cref{lem:ac1}, we have
    \begin{equation}
        \frac{1}{\gamma(t)}\int_0^t \gamma(s) \dd s - \frac{1}{\gamma(\tau)}\int_0^\tau \gamma(s) \dd s \ge 0.
    \end{equation}
    Finally, $\gamma(t) - \gamma(\tau) \ge 0$.
    Therefore, we obtain the desired inequality.
\end{proof}

\begin{lemma}\label{lem:ac3}
    It holds that
    \begin{equation}
        \int_{0}^t \e^{\alpha s^{6/7}} \dd s \le \frac{7}{6\alpha}t^{1/7}(\e^{\alpha t^{6/7}} - 1), \quad \text{and} \quad \int_{t/2}^t \e^{\alpha s^{6/7}} \dd s \ge \frac{1}{\alpha}t^{1/7}(\e^{\alpha t^{6/7}} - \e^{\alpha t^{6/7}/2}).
    \end{equation}
\end{lemma}
\begin{proof}
    The first inequality follows from
    \begin{align}
    \MoveEqLeft
        \int_0^t \e^{\alpha s^{6/7}} \dd s = \frac{7}{6\alpha}\int_{0}^t s^{1/7} \paren*{\frac{6}{7}\alpha s^{-1/7} \e^{\alpha s^{6/7}}} \dd s\\
        &\le \frac{7}{6\alpha} \int_{0}^t t^{1/7} \paren*{\frac{6}{7}\alpha s^{-1/7} \e^{\alpha s^{6/7}}} \dd s = \frac{7}{6\alpha}t^{1/7}(\e^{\alpha t^{6/7}} - 1),
    \end{align}
    and the second inequality follows from
    \begin{equation}
        \int_{t/2}^t \e^{\alpha s^{6/7}} \dd s = \int_{t/2}^t \e^{\alpha ss^{-1/7}} \dd s \ge \int_{t/2}^t \e^{\alpha s t^{-1/7}} \dd s = \frac{1}{\alpha}t^{1/7} (\e^{\alpha t^{6/7}} - \e^{\alpha t^{6/7}/2}).
    \end{equation}
\end{proof}

\begin{lemma}\label{lem:ac4}
    Let $0 < r \le 1$. Then, for any $u >0$,
    \begin{equation}
        \frac{1 - \e^{- u^{6/7}}}{1 - \e^{- r u^{6/7}}} \le \frac{1}{r}.
    \end{equation}
\end{lemma}
\begin{proof}
    Since $\frac{1 - \e^{- u^{6/7}}}{1 - \e^{- r u^{6/7}}}$ is decreasing in $u$
    and
    \[
        \lim_{u \to 0} \frac{1 - \e^{- u^{6/7}}}{1 - \e^{- r u^{6/7}}} = \frac{1}{r},
    \]
    the claim follows.
\end{proof}

\begin{lemma}\label{lem:ac5}
    Define $w_{t}(\tau)$ by~\eqref{awt}.
    Then, for any $0 \le \tau \le T$,
    \begin{equation}
        \int_{\tau}^T F_{t}(\tau) \dd t \le \frac{4\cdot 7^4}{6^4\alpha^2}\tau^{1/7}T^{1/7}.
    \end{equation}
\end{lemma}
\begin{proof}
    By \cref{lem:ac2}, we have
    \begin{align}
        \int_{\tau}^T F_{t}(\tau) \dd t \le \paren*{\int_0^\tau \e^{\alpha s^{6/7}} \dd s} \paren*{ \int_{\tau}^T \frac{\int_0^t \e^{\alpha s^{6/7}} \dd s}{(\int_{t/2}^t \e^{\alpha s^{6/7}} \dd s)^2}(t-\tau) \dd t}.
    \end{align}
    By \cref{lem:ac3}, we have
    \begin{equation}
        \int_0^\tau \e^{\alpha s^{6/7}} \dd s \le \frac{7}{6\alpha}\tau^{1/7}(\e^{\alpha \tau^{6/7}} - 1).
    \end{equation}
    Also, by \cref{lem:ac3} and \cref{lem:ac4}, we have
    \begin{align}
        \MoveEqLeft \int_\tau^T \frac{\int_0^t \e^{\alpha s^{6/7}} \dd s}{(\int_{t/2}^t \e^{\alpha s^{6/7}} \dd s)^2}(t-\tau) \dd t
        \le \int_\tau^T \frac{\frac{7}{6\alpha}t^{1/7}(\e^{\alpha t^{6/7}} - 1)}{\frac{1}{\alpha^2}t^{2/7}(\e^{\alpha t^{6/7}} - \e^{\alpha t^{6/7}/2})^2}(t-\tau) \dd t\\
        &= \frac{7}{6}\int_\tau^T \frac{\alpha t^{-1/7}(t-\tau)}{\e^{\alpha t^{6/7}}-1}\frac{(1 - \e^{-\alpha t^{6/7}})^2}{(1 - \e^{-\alpha t^{6/7}/2})^2} \dd t
        \le 4\cdot\frac{7}{6}\int_\tau^T \frac{\alpha t^{-1/7}(t-\tau)}{\e^{\alpha t^{6/7}}-1} \dd t,
    \end{align}
    and
    \begin{align}
        \MoveEqLeft \int_\tau^T \frac{\alpha t^{-1/7}(t-\tau)}{\e^{\alpha t^{6/7}}-1} \dd t
        = \sqparen*{\frac{7}{6}(t-\tau)\log(1-\e^{-\alpha t^{6/7}})}_{\tau}^T - \frac{7}{6}\int_\tau^T \log(1-\e^{-\alpha t^{6/7}}) \dd t\\
        &\le \frac{7}{6} \int_\tau^T \frac{\e^{-\alpha t^{6/7}}}{1-\e^{-\alpha t^{6/7}}} \dd t
        = \frac{7^2}{6^2\alpha} \int_\tau^T \frac{t^{1/7}(\frac{6}{7}\alpha t^{-1/7}\e^{-\alpha t^{6/7}})}{1-\e^{-\alpha t^{6/7}}} \dd t\\
        &\le \frac{7^2}{6^2\alpha} \int_\tau^T \frac{T^{1/7}(\frac{6}{7}\alpha t^{-1/7}\e^{-\alpha t^{6/7}})}{1-\e^{-\alpha t^{6/7}}} \dd t
        = \frac{7^2}{6^2\alpha} T^{1/7} \sqparen*{\log(1-\e^{-\alpha t^{6/7}})}_{\tau}^T\\
        &\le -\frac{7^2}{6^2\alpha} T^{1/7} \log\paren*{1-\e^{-\alpha \tau^{6/7}}}
        \le \frac{7^2}{6^2\alpha} T^{1/7} \frac{1}{\e^{\alpha \tau^{6/7}}-1}.
    \end{align}
    Here, the first and last inequalities follow from
    $\log(1-x) \ge -\frac{x}{1-x}$ for $0 \le x < 1$.
    Therefore, we obtain the desired inequality.
\end{proof}

\begin{proof}[\bf Proof of \cref{lem:upp}.]
    We evaluate the right-hand side of~\eqref{hesse} in \cref{propac}.
    For the first term, by the definition of the ODE and~\eqref{weidampc},
    \begin{align}
        \MoveEqLeft \norm*{\int_{t/2}^t w_{t}(\tau) \nabla f(x(\tau)) \dd \tau}
        = \norm*{\int_{t/2}^t w_{t}(\tau) \paren*{\ddot{x}(\tau) + \frac{1}{w_{t}(\tau)}\dv{w_{t}(\tau)}{\tau} \dot{x}(\tau)}\dd \tau} \\
        &= \norm*{\int_{t/2}^t \dv{}{\tau}\paren*{w_{t}(\tau)\dot{x}(\tau)} \dd \tau}
        \le w_{t}(t) \norm{\dot{x}(t)} + w_{t}(t/2) \norm{\dot{x}(t/2)}.
    \end{align}
    For $t \ge \varepsilon_0$, where $\varepsilon_0$ satisfies
    \[
        \e^{\alpha \varepsilon_0^{6/7}/2}
        \ge
        \e^{(1 - 2^{-6/7})\alpha \varepsilon_0^{6/7}}
        \ge
        1+ 2^{6/7}
        \ge 2
    \]
    by assumption, it follows from \cref{lem:ac3} that
    \begin{align}
        w_{t}(t) &\le \frac{\e^{\alpha t^{6/7}}}{\frac{1}{\alpha}t^{1/7}(\e^{\alpha t^{6/7}} - \e^{\alpha t^{6/7}/2})}\label{wt1}\\
        &\quad = \frac{\alpha t^{-1/7}}{1-\e^{-\alpha t^{6/7}/2}} \le \frac{\alpha t^{-1/7}}{1-\e^{-\alpha \varepsilon_0^{6/7}/2}} \le 2\alpha t^{-1/7},\\
        \text{and} \quad w_{t}(t/2) &\le \frac{\e^{\alpha (t/2)^{6/7}}}{\frac{1}{\alpha}t^{1/7}(\e^{\alpha t^{6/7}} - \e^{\alpha t^{6/7}/2})}
        \le \frac{\e^{\alpha (t/2)^{6/7}}}{\frac{1}{\alpha}t^{1/7}(\e^{\alpha t^{6/7}} - \e^{\alpha (t/2)^{6/7}})}\label{wt0}\\
        &\quad= \frac{\alpha t^{-1/7}}{\e^{(1-2^{-6/7})\alpha t^{6/7}}-1} \le \frac{\alpha t^{-1/7}}{\e^{(1-2^{-6/7})\alpha \varepsilon_0^{6/7}}-1} \le 2^{-6/7}\alpha t^{-1/7}.
    \end{align}
    Therefore, we have
    \begin{align}
        \MoveEqLeft \int_{\varepsilon_0}^T \norm*{\int_{t/2}^t w_{t}(\tau) \nabla f(x(\tau)) \dd \tau} \dd t\\
        &\le \int_{\varepsilon_0}^T 2\alpha t^{-1/7} \norm{\dot{x}(t)} \dd t + \int_{\varepsilon_0}^T 2^{-6/7}\alpha t^{-1/7} \norm{\dot{x}(t/2)} \dd t\\
        &= \int_{\varepsilon_0}^T 2\alpha t^{-1/7} \norm{\dot{x}(t)} \dd t + \int_{\varepsilon_0/2}^{T/2} \alpha t^{-1/7} \norm{\dot{x}(t)} \dd t\\
        &\le \int_{\varepsilon_0/2}^T 3\alpha t^{-1/7} \norm{\dot{x}(t)} \dd t.
    \end{align}

    For the second term on the right-hand side of~\eqref{hesse}, it follows from \cref{lem:ac5} that
    \begin{align}
        \MoveEqLeft \frac{M}{2}\int_{\varepsilon_0}^T \dd t \int_{t/2}^t \dd \tau \, F_{t}(\tau) \norm{\dot{x}(\tau)}^2\\
        &= \frac{M}{2}\int_{\varepsilon_0/2}^T \dd \tau \int_{\varepsilon_0\lor \tau}^{(2\tau) \land T} \dd t \,F_{t}(\tau) \norm{\dot{x}(\tau)}^2\\
        &\le \frac{M}{2}\int_{\varepsilon_0/2}^T \paren*{ \int_\tau^T F_{t}(\tau) \dd t} \norm{\dot{x}(\tau)}^2 \dd \tau\\
        &\le \frac{2\cdot7^4M}{6^4\alpha^2}\int_{\varepsilon_0/2}^T \tau^{1/7}T^{1/7} \norm{\dot{x}(\tau)}^2 \dd \tau.
    \end{align}
    Therefore, we obtain the desired inequality.
\end{proof}

\section{Proof of \cref{lem:uppd}}\label{sec:lemd}
In this section, we prove \cref{lem:uppd} in a manner parallel to the proof in the previous section.
A key proposition is stated below.
It is a straightforward generalization of~\cite[Lemma~3.1]{MT24} and can be proved in the same way using the discrete Hessian-free inequality~\eqref{eq:HFI}.

\begin{proposition}[cf.~\cite{MT24}, corresponding to \cref{propac}]\label{prop:d1}
    Suppose that $f \colon \RR^d \to \RR$ is twice continuously differentiable and that $\nabla^2 f$ is $M$-Lipschitz continuous.
    Fix $t \ge 2$, and let $w_{t} \colon \ZZ_{\ge 0} \to \RR_{\ge 0}$ satisfy $\sum_{\tau=t_0}^{t-1} w_{t}(\tau) = 1$.
    Define
    \begin{equation}
        F_{t}(\tau) \coloneqq \sum_{s_1=t_0}^\tau \sum_{s_2=\tau+1}^{t-1} w_{t}(s_1) w_{t}(s_2) (s_2-s_1) \quad \text{for} \quad t_0 \le \tau \le t-2
    \end{equation}
    and $F_t(\tau) \coloneqq 0$ for integer $\tau$ outside $t_0\le\tau\le t-2$.
    Then, for any $\setE{x_\tau}_{t_0 \le \tau \le t-1} \subset \RR^d$,
    \begin{equation}
        \norm*{\nabla f \paren*{\sum_{\tau=t_0}^{t-1} w_{t}(\tau)x_\tau}} \le \norm*{\sum_{\tau=t_0}^{t-1} w_{t}(\tau) \nabla f(x_\tau)} + \frac{M}{2}\sum_{\tau=t_0}^{t-2} F_t(\tau) \norm{v_{\tau+1}}^2.
    \end{equation}
\end{proposition}

We prove several lemmas for the weights $w_{t}(\tau)$ and the damping factors $a_\tau$ defined by~\eqref{wt} and~\eqref{at}.
Let $S_1 \coloneqq 0$, and for $t \ge 2$, let
\[
    S_{t} \coloneqq \sum_{s=1}^{t-1} \e^{\alpha s^{6/7}},
    \qquad
    S_{t_0,t} \coloneqq \sum_{s=t_0}^{t-1} \e^{\alpha s^{6/7}},
\]
where $t_0$ is defined by~\eqref{t0}.
Using this notation, the weight can be written as
\[
    w_{t}(\tau) = \frac{S_{\tau+1} - S_{\tau}}{S_{t_0,t}}.
\]

\begin{lemma}[Corresponding to \cref{lem:ac1}]\label{lem:ad1}
    For $\tau \ge 1$,
    \[
        \frac{1}{S_{\tau+1}}\sum_{s=2}^\tau S_{s}
    \]
    is nondecreasing in $\tau$, where the empty sum $\sum_{s=2}^1 S_{s}$ is defined as $0$.
\end{lemma}
\begin{proof}
    Since, for $t \ge 2$,
    \begin{equation}
        \frac{1}{S_{t+1}}\sum_{s=2}^t S_{s} - \frac{1}{S_{t}}\sum_{s=2}^{t-1} S_{s} = \frac{S_{t}^2 - (S_{t+1}-S_{t})\sum_{s=2}^{t-1}S_{s}}{S_{t+1}S_{t}},
    \end{equation}
    it suffices to show that
    \[
        G_1(t) \coloneqq S_{t}^2 - (S_{t+1}-S_{t})\sum_{s=2}^{t-1}S_{s} \ge 0
    \]
    for $t\ge 2$.
    Since $G_1(2) = S_2^2 \ge 0$, it is enough to show that $G_1(t)$ is nondecreasing for $t \ge 2$.
    Indeed,
    \begin{align}
        \MoveEqLeft \paren*{S_{t+1}^2 - (S_{t+2}-S_{t+1})\sum_{s=2}^{t}S_{s}} - \paren*{S_{t}^2 - (S_{t+1}-S_{t})\sum_{s=2}^{t-1}S_{s}}\\
        &= (S_{t+1} - S_{t})S_{t+1} - (S_{t+2} - 2S_{t+1} + S_{t})\sum_{s = 2}^t S_{s}\\
        &= \e^{\alpha t^{6/7}} S_{t+1} - (\e^{\alpha (t+1)^{6/7}}-\e^{\alpha t^{6/7}})\sum_{s = 2}^t S_{s}\\
        &= (\e^{\alpha (t+1)^{6/7}}-\e^{\alpha t^{6/7}})\paren*{\frac{\e^{\alpha t^{6/7}}}{\e^{\alpha (t+1)^{6/7}}-\e^{\alpha t^{6/7}}}S_{t+1} - \sum_{s=2}^t S_{s}}.
    \end{align}
    Thus, it suffices to show that
    \[
        G_2(t) \coloneqq \frac{\e^{\alpha t^{6/7}}}{\e^{\alpha (t+1)^{6/7}}-\e^{\alpha t^{6/7}}}S_{t+1} - \sum_{s=2}^t S_{s} \ge 0
    \]
    for $t\ge 1$.
    Since $G_2(1) \ge 0$, the remaining task is to show that the following difference is nonnegative for $t \ge 1$:
    \begin{align}
        \MoveEqLeft \paren*{\frac{\e^{\alpha (t+1)^{6/7}}}{\e^{\alpha (t+2)^{6/7}}-\e^{\alpha (t+1)^{6/7}}}S_{t+2} - \sum_{s=2}^{t+1} S_{s}} - \paren*{\frac{\e^{\alpha t^{6/7}}}{\e^{\alpha (t+1)^{6/7}}-\e^{\alpha t^{6/7}}}S_{t+1} - \sum_{s=2}^t S_{s}}\\
        &= \paren*{\frac{\e^{\alpha (t+1)^{6/7}}}{\e^{\alpha (t+2)^{6/7}}-\e^{\alpha (t+1)^{6/7}}} - \frac{\e^{\alpha t^{6/7}}}{\e^{\alpha (t+1)^{6/7}}-\e^{\alpha t^{6/7}}}}S_{t+2}\\
        &\quad + \frac{\e^{\alpha t^{6/7}}}{\e^{\alpha (t+1)^{6/7}}-\e^{\alpha t^{6/7}}}(S_{t+2}-S_{t+1}) - S_{t+1}.
    \end{align}
    The first term is nonnegative because
    $\frac{\e^{\alpha t^{6/7}}}{\e^{\alpha (t+1)^{6/7}}-\e^{\alpha t^{6/7}}}$
    is increasing in $t$.
    The remaining terms are also nonnegative, since
    \begin{align}
        \MoveEqLeft \frac{\e^{\alpha t^{6/7}}}{\e^{\alpha (t+1)^{6/7}}-\e^{\alpha t^{6/7}}}(S_{t+2}-S_{t+1}) - S_{t+1}\\
        &= \frac{\e^{\alpha t^{6/7}}\e^{\alpha (t+1)^{6/7}}}{\e^{\alpha (t+1)^{6/7}}-\e^{\alpha t^{6/7}}} - \sum_{s=1}^t \frac{\e^{\alpha s^{6/7}}}{\e^{\alpha (s+1)^{6/7}}-\e^{\alpha s^{6/7}}} (\e^{\alpha (s+1)^{6/7}}-\e^{\alpha s^{6/7}})\\
        &= \frac{\e^{2\alpha 1^{6/7}}}{\e^{\alpha 2^{6/7}}-\e^{\alpha 1^{6/7}}} +\sum_{s=2}^t \paren*{\frac{\e^{\alpha s^{6/7}}}{\e^{\alpha (s+1)^{6/7}}-\e^{\alpha s^{6/7}}} - \frac{\e^{\alpha (s-1)^{6/7}}}{\e^{\alpha s^{6/7}}-\e^{\alpha (s-1)^{6/7}}}} \e^{\alpha s^{6/7}} > 0,
     \end{align}
     where the last summation is defined as $0$ when $t = 1$.
\end{proof}

\begin{lemma}[Corresponding to \cref{lem:ac2}]\label{lem:ad2}
    Fix $t \ge 2$.
    For $t_0 \le \tau \le t-2$, it holds that
    \begin{equation}
        F_{t}(\tau) \le \frac{S_{\tau+1}S_{t}}{(S_{t_0,t})^2}(t-\tau-1),
    \end{equation}
    where $F_{t}(\tau)$ is defined in \cref{prop:d1}.
\end{lemma}
\begin{proof}
    By definition,
    \begin{align}
        \MoveEqLeft F_t(\tau) = \sum_{s_1=t_0}^\tau \sum_{s_2=\tau+1}^{t-1} w_{t}(s_1) w_{t}(s_2) (s_2-s_1)\\
        &= \frac{1}{(S_{t_0,t})^2}\sum_{s_1=t_0}^\tau \sum_{s_2=\tau+1}^{t-1} (S_{s_1+1} - S_{s_1}) (S_{s_2+1} - S_{s_2}) (s_2-s_1),
    \end{align}
    and
    \begin{align}
        &
        \sum_{s_1=t_0}^\tau \sum_{s_2=\tau+1}^{t-1} (S_{s_1+1} - S_{s_1}) (S_{s_2+1} - S_{s_2}) (s_2-s_1)\\
        &= \sum_{s_1=t_0}^\tau (S_{s_1+1} - S_{s_1})\paren*{S_{t}(t-1-s_1) - S_{\tau+1}(\tau+1-s_1) - \sum_{s_2=\tau+2}^{t-1} S_{s_2}}\\
        &= S_{\tau+1}S_{t}(t-1-\tau) - S_{t_0}S_{t}(t-1-t_0)
         -S_{\tau+1}S_{\tau+1} + S_{t_0}S_{\tau+1}(\tau+1-t_0)\\
        &\quad + (S_{t} - S_{\tau+1})\sum_{s_1=t_0+1}^\tau S_{s_1} - (S_{\tau+1} - S_{t_0})\sum_{s_2=\tau+2}^{t-1} S_{s_2}\\
        &= S_{\tau+1}S_{t}(t-1-\tau)
        - S_{t_0}\paren*{S_{t}(t-1-t_0) - S_{\tau+1}(\tau+1 -t_0) - \sum_{s=\tau+2}^{t-1} S_{s}}\\
        &\quad - \paren*{S_{\tau+1}\sum_{s=t_0+1}^{t-1} S_{s} - S_{t} \sum_{s=t_0+1}^\tau S_{s}}\\ 
        &= S_{\tau+1}S_{t}(t-1-\tau)
        - S_{t_0}\paren*{S_{t}(t-1-t_0) - \sum_{s=2}^{t-1} S_{s} - S_{\tau+1}(\tau+1 -t_0) + \sum_{s=2}^{\tau+1} S_{s}}\\
        &\quad - S_{\tau+1}S_{t}\paren*{\frac{1}{S_{t}}\sum_{s=2}^{t-1} S_{s} - \frac{1}{S_{\tau+1}} \sum_{s=2}^\tau S_{s}} - (S_{t} - S_{\tau+1})\sum_{s=2}^{t_0}S_{s}.
    \end{align}
    We now show that the second, third, and fourth terms in the above expression are nonpositive, which implies the desired bound.
    In the following, note that $t_0 \le \tau+1 \le t-1$.
    For the second term, observe that
    \[
        G(u) \coloneqq S_{u}(u-t_0) - \sum_{s=2}^{u} S_{s}
    \]
    is nondecreasing because, for $u \ge t_0$,
    \begin{align}
        \MoveEqLeft
        \paren*{S_{u+1}(u+1-t_0) - \sum_{s=2}^{u+1} S_{s}} - \paren*{S_{u}(u-t_0) - \sum_{s=2}^{u} S_{s}}\\
        &= (S_{u+1} -S_{u})(u-t_0) \ge 0.
    \end{align}
    Thus, since $S_t \ge S_{t-1}$, we have
    \begin{align}
        \MoveEqLeft
        S_{t}(t-1-t_0) - \sum_{s=2}^{t-1} S_{s} - S_{\tau+1}(\tau+1 -t_0) + \sum_{s=2}^{\tau+1} S_{s}\\
        &\ge S_{t-1}(t-1-t_0) - \sum_{s=2}^{t-1} S_{s} - S_{\tau+1}(\tau+1 -t_0) + \sum_{s=2}^{\tau+1} S_{s}\\
        &= G(t-1) - G(\tau+1) \ge 0.
    \end{align}
    In addition, by \cref{lem:ad1},
    \begin{equation}
        \frac{1}{S_{t}}\sum_{s=2}^{t-1} S_{s} - \frac{1}{S_{\tau+1}} \sum_{s=2}^\tau S_{s} \ge 0.
    \end{equation}
    Finally, $S_{t} - S_{\tau+1} \ge 0$.
    Therefore, we obtain the desired inequality.
\end{proof}

In the following lemmas, note that $t/4 < t_0 \le t/2$ for $t \ge 2$, which follows from the definition~\eqref{t0}.

\begin{lemma}[Corresponding to \cref{lem:ac3}]\label{lem:ad3}
    For $t \ge 2$, it holds that
    \begin{align}
        S_{t} &\le \paren*{\frac{7}{6\alpha}(t-1)^{1/7} + 1}(\e^{\alpha (t-1)^{6/7}} - 1),\\
        \text{and}\quad S_{t_0,t} &\ge \frac{1}{\alpha}(t-1)^{1/7}(\e^{\alpha (t-1)^{6/7}}-\e^{\alpha (t-1)^{6/7}/2}).
    \end{align}
\end{lemma}
\begin{proof}
    Since
    $\frac{\e^{\alpha s^{6/7}}}{\e^{\alpha s^{6/7}} - \e^{\alpha (s-1)^{6/7}}}$
    is increasing in $s$, the first inequality follows from
    \begin{align}
        \MoveEqLeft S_{t} = \sum_{s=1}^{t-1} \e^{\alpha s^{6/7}} = \sum_{s=1}^{t-1} \frac{\e^{\alpha s^{6/7}}}{\e^{\alpha s^{6/7}} - \e^{\alpha (s-1)^{6/7}}} (\e^{\alpha s^{6/7}} - \e^{\alpha (s-1)^{6/7}})\\
        &\le \frac{\e^{\alpha (t-1)^{6/7}}}{\e^{\alpha (t-1)^{6/7}} - \e^{\alpha (t-2)^{6/7}}} \sum_{s=1}^{t-1} (\e^{\alpha s^{6/7}} - \e^{\alpha (s-1)^{6/7}})\\
        &= \frac{1}{1 - \e^{-\alpha ((t-1)^{6/7} - (t-2)^{6/7})}} (\e^{\alpha (t-1)^{6/7}} - 1)
    \end{align}
    and
    \begin{align}
        \MoveEqLeft \frac{1}{1 - \e^{-\alpha ((t-1)^{6/7} - (t-2)^{6/7})}} \le \frac{1}{\alpha ((t-1)^{6/7} - (t-2)^{6/7})} + 1 \\
        &= \frac{1}{\alpha (t-1)^{6/7}\paren*{1-\paren*{1-\frac{1}{t-1}}^{6/7}}} + 1 \le \frac{7}{6\alpha}(t-1)^{1/7} + 1,
    \end{align}
    where we applied the inequalities $\frac{1}{1-\e^{-x}} \le \frac{1}{x}+1$ for $x > 0$ and $(1-x)^{6/7} \le 1 - \frac{6}{7}x$ for $0 \le x \le 1$.
    
    Similarly to \cref{lem:ac3}, the second inequality follows from
    \begin{align}
        &S_{t_0,t} = \sum_{s=t_0}^{t-1} \e^{\alpha s^{6/7}} \ge \int_{t_0-1}^{t-1} \e^{\alpha s^{6/7}} \dd s \ge \int_{(t-1)/2}^{t-1} \e^{\alpha s^{6/7}} \dd s\\
        &\quad \ge \frac{1}{\alpha}(t-1)^{1/7}(\e^{\alpha (t-1)^{6/7}} - \e^{\alpha (t-1)^{6/7}/2}).
    \end{align}
\end{proof}

\begin{lemma}[Corresponding to \cref{lem:ac5}]\label{lem:ad4}
    Assume that $\tau$ is a sufficiently large integer such that
    \begin{equation} 
        \frac{7^2}{6^2}\paren*{\frac{7}{6}+ \frac{\alpha}{\tau^{1/7}}}^2 \frac{e^{\alpha \tau^{6/7}} - 1}{\e^{\alpha (\tau-1)^{6/7}} - 1} \le 3. \label{asm1}
    \end{equation}
    Then, for any $T$ such that $T \ge \tau + 2$,
    \begin{equation}
        \sum_{t=\tau+1}^{T-1} F_{t}(\tau) \le \frac{12}{\alpha^2} \tau^{1/7} T^{1/7}.
    \end{equation}
\end{lemma}
\begin{proof}
    By \cref{lem:ad2}, we have
    \begin{equation}
        \sum_{t=\tau+1}^{T-1} F_{t}(\tau) \le S_{\tau+1}\sum_{t=\tau+1}^{T-1} \frac{S_{t}}{(S_{t_0,t})^2}(t-\tau-1).
    \end{equation}
    By \cref{lem:ad3}, we have
    \begin{equation}
        S_{\tau+1} \le \paren*{\frac{7}{6\alpha}\tau^{1/7} + 1}(e^{\alpha \tau^{6/7}} - 1) = \frac{\tau^{1/7}}{\alpha}\paren*{\frac{7}{6}+ \frac{\alpha}{\tau^{1/7}}}(e^{\alpha \tau^{6/7}} - 1).
    \end{equation}
    Also, by \cref{lem:ad3} and \cref{lem:ac4},
    \begin{align}
        \MoveEqLeft
        \sum_{t=\tau+1}^{T-1} \frac{S_{t}}{(S_{t_0,t})^2}(t-\tau-1)\\
        &\le \sum_{t=\tau+1}^{T-1} \frac{\paren*{\frac{7}{6\alpha}(t-1)^{1/7} + 1}(e^{\alpha (t-1)^{6/7}} - 1)}{\frac{1}{\alpha^2}(t-1)^{2/7}(e^{\alpha (t-1)^{6/7}}-\e^{\alpha (t-1)^{6/7}/2})^2}(t-\tau-1)\\
        &= \sum_{t=\tau+1}^{T-1} \frac{\alpha (t-\tau-1)}{(t-1)^{1/7}(e^{\alpha (t-1)^{6/7}} - 1)} \frac{\frac{7}{6\alpha}(t-1)^{1/7} + 1}{\frac{1}{\alpha}(t-1)^{1/7}} \frac{(e^{\alpha (t-1)^{6/7}} - 1)^2}{(e^{\alpha (t-1)^{6/7}}-\e^{\alpha (t-1)^{6/7}/2})^2}\\
        &\le 4\paren*{\frac{7}{6} + \frac{\alpha}{\tau^{1/7}}}\sum_{t=\tau+1}^{T-1} \frac{\alpha (t-\tau-1)}{(t-1)^{1/7}(e^{\alpha (t-1)^{6/7}} - 1)}.
    \end{align}
    Since $\log(1-x) \ge -\frac{x}{1-x}$ for $0 < x < 1$,
    \begin{align}
        \MoveEqLeft \sum_{t=\tau+1}^{T-1} \frac{\alpha (t-\tau-1)}{(t-1)^{1/7}(e^{\alpha (t-1)^{6/7}} - 1)}\\
        &= \sum_{t = \tau+1}^{T-1} \paren*{\sum_{s = t}^T \frac{\alpha}{(s-1)^{1/7}(e^{\alpha (s-1)^{6/7}} - 1)} - \sum_{s = t+1}^T \frac{\alpha}{(s-1)^{1/7}(e^{\alpha (s-1)^{6/7}} - 1)}}(t-\tau-1)\\
        &\le \sum_{t=\tau+2}^{T-1} \sum_{s = t}^T \frac{\alpha}{(s-1)^{1/7}(e^{\alpha (s-1)^{6/7}} - 1)} - \frac{\alpha}{(T-1)^{1/7}(e^{\alpha (T-1)^{6/7}} - 1)}(T-\tau-2) \\
        &\le \sum_{t=\tau+2}^{T-1} \int_{t-2}^{T-1} \frac{\alpha}{s^{1/7}(e^{\alpha s^{6/7}} - 1)} \dd s\\
        &\le -\sum_{t=\tau+2}^{T-1} \frac{7}{6}\log(1-\e^{-\alpha (t-2)^{6/7}})
        \le \sum_{t=\tau+2}^{T-1} \frac{7}{6}\frac{\e^{-\alpha(t-2)^{6/7}}}{1 - \e^{-\alpha(t-2)^{6/7}}}\\
        &\le \int_{\tau-1}^{T-3} \frac{7}{6} \frac{\e^{-\alpha t^{6/7}}}{1 - \e^{-\alpha t^{6/7}}} \dd t
        \le -\frac{7^2}{6^2\alpha}(T-3)^{1/7}\log(1-\e^{-\alpha (\tau-1)^{6/7}})\\
        &\le \frac{7^2}{6^2\alpha}(T-3)^{1/7}\frac{1}{\e^{\alpha (\tau-1)^{6/7}} - 1}.
    \end{align}
    Therefore, by the assumption,
    \begin{align}
        \MoveEqLeft \sum_{t=\tau+1}^{T-1} F_{t}(\tau) \le 4\frac{7^2}{6^2\alpha^2}\paren*{\frac{7}{6}+ \frac{\alpha}{\tau^{1/7}}}^2 \frac{e^{\alpha \tau^{6/7}} - 1}{\e^{\alpha (\tau-1)^{6/7}} - 1} \tau^{1/7} (T-3)^{1/7}
        \le \frac{12}{\alpha^2} \tau^{1/7} T^{1/7}.
    \end{align}
\end{proof}

\begin{lemma}[Corresponding to~\eqref{wt1} and~\eqref{wt0}]\label{lem:ad5}
    Let $\varepsilon_0$ be a sufficiently large integer satisfying, for all $t \ge \varepsilon_0$,
    \begin{equation}
        \e^{-\alpha (t - 1)^{6/7}/2} \le \frac12, \quad  \text{and} \quad \frac{\e^{-\alpha \paren*{(t-1)^{6/7} - (t/2)^{6/7}}}}{1 - \e^{-\alpha (t-1)^{6/7}/2}} \le \frac{23\cdot2^{1/7}}{t}. \label{asm2}
    \end{equation}
    Then, for any $t \ge \varepsilon_0$,
    \begin{equation}
        w_{t}(t-1) \le \frac{2\alpha}{(t-1)^{1/7}}, \qquad w_{t}(t_0) \le \frac{23\cdot2^{1/7}\alpha}{t(t-1)^{1/7}}.
    \end{equation}
\end{lemma}
\begin{remark}
    In the continuous-time analysis, we used $w_t(t/2) = \Order{t^{-1/7}}$.
    In the discrete-time analysis, we instead use $w_{t}(t_0) = \Order{t^{-8/7}}$.
    This sharper estimate is necessary to bound the error arising from the fact that $t/4 < t_0 \le t/2$ is used as a discrete counterpart of $t/2$.
    
    When $\alpha = 0.1$, the assumption~\eqref{asm2} is satisfied by $\varepsilon_0 \ge 28$.
\end{remark}

\begin{proof}
    By \cref{lem:ad3} and the assumption, we have
    \begin{align}
        \MoveEqLeft w_{t}(t-1) = \frac{\e^{\alpha (t-1)^{6/7}}}{S_{t_0,t}} \le \frac{\alpha\e^{\alpha (t-1)^{6/7}}}{(t-1)^{1/7}(e^{\alpha (t-1)^{6/7}}-\e^{\alpha (t-1)^{6/7}/2})}\\
        &= \frac{\alpha}{(t-1)^{1/7}(1 - \e^{-\alpha (t-1)^{6/7}/2})} \le \frac{2\alpha}{(t-1)^{1/7}},
    \end{align}
    and
    \begin{align}
        \MoveEqLeft w_{t}(t_0) = \frac{\e^{\alpha t_0^{6/7}}}{S_{t_0,t}} \le \frac{\alpha\e^{\alpha (t/2)^{6/7}}}{(t-1)^{1/7}(e^{\alpha (t-1)^{6/7}}-\e^{\alpha (t-1)^{6/7}/2})}\\
        &= \frac{\alpha\e^{-\alpha \paren*{(t-1)^{6/7} - (t/2)^{6/7}}}}{(t-1)^{1/7}(1 - \e^{-\alpha (t-1)^{6/7}/2})} \le \frac{23\cdot2^{1/7}\alpha}{t(t-1)^{1/7}}.
    \end{align}
\end{proof}

The following two lemmas are specific to the discrete-time analysis.
For the next lemma, note that $-x \log(1-x^{-1})$ is decreasing and larger than $1$ for $x > 1$.

\begin{lemma}\label{lem:dF6}
    Let $\varepsilon_0$ be a sufficiently large integer satisfying
    \begin{equation}
        -\frac{7}{3}\e^{\alpha (\varepsilon_0-1)^{6/7}} \log(1-\e^{-\alpha (\varepsilon_0-2)^{6/7}})\le \frac{13}{5}.  \label{asm3}
    \end{equation}
    Then, for any $T$ and $\tau$ such that $T-1 \ge \tau+1 \ge 2$,
    \begin{equation}
        \sum_{t=(\tau+1)\lor\varepsilon_0}^{T-1} \frac{\e^{\alpha \tau^{6/7}}}{S_{t_0,t}} \le \frac{13}{5}.
    \end{equation}
\end{lemma}
\begin{proof}
    By \cref{lem:ad3} and \cref{lem:ac4}, we have
    \begin{align}
        \MoveEqLeft \sum_{t=(\tau+1)\lor\varepsilon_0}^{T-1} \frac{\e^{\alpha \tau^{6/7}}}{S_{t_0,t}}
        \le \sum_{t=(\tau+1)\lor\varepsilon_0}^{T-1} \frac{\alpha\e^{\alpha \tau^{6/7}}}{(t-1)^{1/7}(e^{\alpha (t-1)^{6/7}}-\e^{\alpha (t-1)^{6/7}/2})}\\
        &\le \int_{(\tau-1)\lor(\varepsilon_0-2)}^{T-2} \frac{\alpha\e^{\alpha \tau^{6/7}}}{t^{1/7}(\e^{\alpha t^{6/7}}-\e^{\alpha t^{6/7}/2})} \dd t 
        \le 2\int_{(\tau-1)\lor(\varepsilon_0-2)}^{T-2}\frac{\alpha\e^{\alpha \tau^{6/7}}}{t^{1/7}(\e^{\alpha t^{6/7}}-1)}\dd t\\
        &\le -\frac{7}{3}\e^{\alpha \tau^{6/7}} \log(1-\e^{-\alpha ((\tau-1)\lor(\varepsilon_0-2))^{6/7}})\\
        &= \begin{cases}
            -\frac{7}{3}\e^{\alpha \tau^{6/7}} \log(1-\e^{-\alpha (\varepsilon_0-2)^{6/7}}) & \text{if} \quad \tau+1 \le \varepsilon_0, \\
            -\frac{7}{3}\e^{\alpha \tau^{6/7}} \log(1-\e^{-\alpha (\tau-1)^{6/7}}) & \text{otherwise.}
        \end{cases}
    \end{align}
    Here, the first case is increasing in $\tau$, whereas the second case is decreasing in $\tau$.
    Thus, by the assumption,
    \begin{equation}
        \sum_{t=(\tau+1)\lor\varepsilon_0}^{T-1} \frac{\e^{\alpha \tau^{6/7}}}{S_{t_0,t}}
        \le -\frac{7}{3}\e^{\alpha (\varepsilon_0-1)^{6/7}} \log(1-\e^{-\alpha (\varepsilon_0-2)^{6/7}})
        \le \frac{13}{5}.
    \end{equation}
\end{proof}

\begin{lemma}\label{lem:dF7}
    For any $\tau \ge 1$, it holds that
    \begin{equation}
        \frac{a_{\tau}-a_{\tau+1}}{2 + a_\tau} \le \frac{10\alpha}{49\tau^{8/7}}.
    \end{equation}
\end{lemma}
\begin{proof}
    Since $1 - \e^x \le -x$ for any $x \in \RR$, we have
    \begin{align}
        \MoveEqLeft \frac{a_{\tau}-a_{\tau+1}}{2 + a_{\tau}} = \frac{\e^{\alpha(\tau^{6/7} - (\tau-1)^{6/7})}- \e^{\alpha((\tau+1)^{6/7} - \tau^{6/7})}}{1 + \e^{\alpha (\tau^{6/7} - (\tau-1)^{6/7})}} \\
        &= \frac{1- \e^{\alpha((\tau+1)^{6/7} - 2\tau^{6/7} + (\tau-1)^{6/7})}}{\e^{-\alpha(\tau^{6/7} - (\tau-1)^{6/7})} + 1} \le \frac{\alpha(-(\tau+1)^{6/7} + 2\tau^{6/7} - (\tau-1)^{6/7})}{\e^{-\alpha(\tau^{6/7} - (\tau-1)^{6/7})} + 1}.
    \end{align}
    Since $(1 + x)^{6/7} \ge 1 + \frac{6}{7}x - \frac{3}{49}x^2$ for $x \ge 0$ and $(1 - x)^{6/7} \ge 1 - \frac{6}{7}x - \frac{7}{49}x^2$ for $0 \le x \le 1$, it holds that
    \begin{align}
        \MoveEqLeft -(\tau+1)^{6/7} + 2\tau^{6/7} - (\tau-1)^{6/7} = \tau^{6/7}\paren*{2 - \paren*{1 + \frac{1}{\tau}}^{6/7} - \paren*{1 - \frac{1}{\tau}}^{6/7}} \le \frac{10}{49\tau^{8/7}}.
    \end{align}
    Therefore, since
    \[
        \frac{1}{\e^{-\alpha(\tau^{6/7} - (\tau-1)^{6/7})} + 1} \le 1,
    \]
    the claim follows.
\end{proof}

\begin{proof}[Proof of \cref{lem:uppd}]
    From \cref{prop:d1}, we have
    \begin{equation}
        \begin{split}
            \MoveEqLeft
            \sum_{t=\varepsilon_0}^{T-1} \norm*{\nabla f \paren*{\bar{x}_t}}
            \le \sum_{t=\varepsilon_0}^{T-1}\paren*{\norm*{\sum_{\tau=t_0}^{t-1} w_{t}(\tau) \nabla f(x_\tau)} + \frac{M}{2}\sum_{\tau=t_0}^{t-2} F_t(\tau) \norm{v_{\tau+1}}^2}.\label{pr4init}
        \end{split}
    \end{equation}
    Let $q_t \coloneqq {h_{t}^2}/{h_{t+1}^2}$.
    By~\eqref{ifr0}, we have $q_t\hat{r}_t \le 1$ for all $t$.
    For the first term on the right-hand side of~\eqref{pr4init}, substituting the update~\eqref{scheme} and using $(1 + a_{\tau})w_{t}(\tau-1) = w_{t}(\tau)$, which follows from the definition of $a_\tau$ in~\eqref{at}, we obtain
    \begin{align}
        \MoveEqLeft \norm*{\sum_{\tau=t_0}^{t-1} w_{t}(\tau)\nabla f(x_\tau)}\\
        &= \norm*{\sum_{\tau=t_0}^{t-1} w_{t}(\tau)\paren*{(2-\hat{r}_{\tau+1} + a_{\tau+1}) \frac{v_{\tau+1}}{h_{\tau+1}^2} - q_\tau\hat{r}_{\tau}\frac{2+a_{\tau+1}}{2+a_{\tau}}\frac{v_{\tau}}{h_{\tau}^2}}}\\
        &= \bigg \| w_{t}(t-1)(2-\hat{r}_{t} + a_{t}) \frac{v_{t}}{h_{t}^2} - w_{t}(t_0)q_{t_0}\hat{r}_{t_0}\frac{2+a_{t_0+1}}{2+a_{t_0}}\frac{v_{t_0}}{h_{t_0}^2}\\
        &\qquad + \sum_{\tau=t_0+1}^{t-1}  \paren*{ w_{t}(\tau-1)(2 - \hat{r}_{\tau} + a_{\tau}) - w_{t}(\tau)q_\tau\hat{r}_{\tau}\frac{2+a_{\tau+1}}{2+a_{\tau}} }\frac{v_{\tau}}{h_{\tau}^2} \bigg\|\\
        &= \bigg\| w_{t}(t-1)(2-\hat{r}_{t} + a_{t}) \frac{v_{t}}{h_{t}^2} - w_{t}(t_0)q_{t_0}\hat{r}_{t_0}\frac{2+a_{t_0+1}}{2+a_{t_0}}\frac{v_{t_0}}{h_{t_0}^2} \\
        &\qquad + \sum_{\tau=t_0+1}^{t-1} \paren*{ w_{t}(\tau-1)(1-\hat{r}_{\tau}) + w_{t}(\tau)\paren*{1-q_\tau\hat{r}_{\tau}\frac{2+a_{\tau+1}}{2+a_{\tau}}}}\frac{v_{\tau}}{h_{\tau}^2} \bigg\| \\
        &\le w_{t}(t-1)(2-\hat{r}_{t} + a_{t}) \frac{\norm{v_t}}{h_t^2} + w_{t}(t_0)\frac{\norm{v_{t_0}}}{h_{t_0}^2}\\
        &\qquad + \sum_{\tau=t_0+1}^{t-1} \paren*{ w_{t}(\tau-1)(1 - \hat{r}_{\tau}) + w_{t}(\tau)\paren*{1-q_\tau\hat{r}_{\tau}\frac{2+a_{\tau+1}}{2+a_{\tau}}}}\frac{\norm{v_{\tau}}}{h_\tau^2}.
    \end{align}
    Therefore, by \cref{lem:ad5}, whose assumption~\eqref{asm2} is satisfied by~\eqref{asm92},
    \begin{align}
        \MoveEqLeft \sum_{t=\varepsilon_0}^{T-1} \norm*{\sum_{\tau=t_0}^{t-1} w_t(\tau)\nabla f(x_\tau)}\\
        &\le \sum_{t=\varepsilon_0}^{T-1} \frac{2\alpha}{h_t^2(t-1)^{1/7}}(2-\hat{r}_{t} + a_{t}) \norm{v_t} + \sum_{t=\varepsilon_0}^{T-1} \frac{23\cdot 2^{1/7}\alpha}{h_{t_0}^2t(t-1)^{1/7}} \norm{v_{t_0}}\\
        &\qquad + \sum_{t=\varepsilon_0}^{T-1} \sum_{\tau=t_0+1}^{t-1} \frac{1}{h_\tau^2}\paren*{ \frac{\e^{\alpha (\tau-1)^{6/7}}}{S_{t_0,t}}(1 - \hat{r}_{\tau}) + \frac{\e^{\alpha \tau^{6/7}}}{S_{t_0,t}}\paren*{1-q_\tau\hat{r}_{\tau}\frac{2+a_{\tau+1}}{2+a_{\tau}}}}\norm{v_{\tau}}. \label{daG11}
    \end{align}
    For the second term on the right-hand side above, since $t/4 < t_0 \le t/2$, that is, $2t_0 \le t < 4t_0$, we have
    \begin{align}
        \MoveEqLeft
        \sum_{t=\varepsilon_0}^{T-1} \frac{23\cdot2^{1/7}\alpha}{h_{t_0}^2t(t-1)^{1/7}}\norm{v_{t_0}}
        \le \sum_{t=\varepsilon_0}^{T-1} \frac{23\cdot2^{1/7}\alpha}{h_{t_0}^2 2t_0(2t_0-1)^{1/7}}\norm{v_{t_0}}\\
        &\le \sum_{t_0=\floor{\varepsilon_0/4}}^{\floor{(T-1)/2}} \paren*{4t_0 - 2t_0}\frac{23\alpha}{h_{t_0}^22t_0(t_0-\frac12)^{1/7}}\norm{v_{t_0}}
        \le  \sum_{t=\floor{\varepsilon_0/4}}^{T-1} \frac{23\alpha}{h_{t}^2(t-1)^{1/7}} \norm{v_{t}},
    \end{align}
    where, in the last inequality, $t_0$ is relabeled as $t$.
    For the third term in~\eqref{daG11}, by \cref{lem:dF6}, whose assumption~\eqref{asm3} is satisfied by~\eqref{asm93}, and by \cref{lem:dF7}, we have
    \begin{align}
        \MoveEqLeft \sum_{t=\varepsilon_0}^{T-1} \sum_{\tau=t_0+1}^{t-1} \frac{1}{h_\tau^2}\paren*{ \frac{\e^{\alpha (\tau-1)^{6/7}}}{S_{t_0,t}}(1 - \hat{r}_{\tau}) + \frac{\e^{\alpha \tau^{6/7}}}{S_{t_0,t}}\paren*{1-q_\tau\hat{r}_{\tau}\frac{2+a_{\tau+1}}{2+a_{\tau}}}}\norm{v_{\tau}}\\
        &\le  \sum_{t=\varepsilon_0}^{T-1} \sum_{\tau=\floor{\varepsilon_0/4}}^{t-1} \frac{1}{h_\tau^2}\left( \frac{\e^{\alpha (\tau-1)^{6/7}}}{S_{t_0,t}}(1 - \hat{r}_\tau) \right.\\
        &\hspace{90pt} \left. + \frac{\e^{\alpha \tau^{6/7}}}{S_{t_0,t}}\frac{2 + a_{\tau+1}}{2 + a_\tau}(1 - q_\tau\hat{r}_\tau) + \frac{\e^{\alpha \tau^{6/7}}}{S_{t_0,t}}\frac{a_\tau - a_{\tau+1}}{2 + a_\tau}\right)\norm{v_{\tau}}\\
        &= \sum_{\tau=\floor{\varepsilon_0/4}}^{T-2} \sum_{t=(\tau+1)\lor\varepsilon_0}^{T-1} \frac{1}{h_\tau^2}\left(\frac{\e^{\alpha (\tau-1)^{6/7}}}{S_{t_0,t}}(1 - \hat{r}_\tau)\right. \\
        &\hspace{100pt} \left. + \frac{\e^{\alpha \tau^{6/7}}}{S_{t_0,t}}\frac{2 + a_{\tau+1}}{2 + a_\tau}(1 - q_\tau\hat{r}_\tau) + \frac{\e^{\alpha \tau^{6/7}}}{S_{t_0,t}}\frac{a_\tau - a_{\tau+1}}{2 + a_\tau}\right)\norm{v_{\tau}}\\
        &\le \sum_{\tau=\floor{\varepsilon_0/4}}^{T-2} \frac{1}{h_\tau^2}\paren*{\frac{13}{5}(1 - \hat{r}_\tau) + \frac{13}{5}(1 - q_\tau\hat{r}_\tau) + \frac{13}{5}\frac{10\alpha}{49\tau^{8/7}}}\norm{v_{\tau}}.
    \end{align}
    Thus, since
    $2a_{\floor{\varepsilon_0/4}} + \frac{26}{49(\floor{\varepsilon_0/4}-1)} \le 1$
    by the assumption~\eqref{asm94},
    \begin{align}
        \MoveEqLeft \sum_{t=\varepsilon_0}^{T-1} \norm*{\sum_{\tau=t_0}^{t-1} w_t(\tau)\nabla f(x_\tau)}\\
        &\le \sum_{t=\floor{\varepsilon_0/4}}^{T-1} \frac{1}{h_t^2}\bigg( \frac{2\alpha}{(t-1)^{1/7}}(2-\hat{r}_{t} + a_{t}) + \frac{23\alpha}{(t-1)^{1/7}} \\
        &\hspace{80pt} + \frac{13}{5}(1 - \hat{r}_t) + \frac{13}{5}(1 - q_t\hat{r}_t) + \frac{26\alpha}{49t^{8/7}}\bigg) \norm{v_t}\\
        &\le \sum_{t=\floor{\varepsilon_0/4}}^{T-1} \frac{1}{h_t^2}\bigg( \frac{\alpha}{(t-1)^{1/7}}\paren*{2(2-\hat{r}_{t}) + 23 + 2a_t + \frac{26}{49(t-1)}}\\
        &\hspace{80pt} + \frac{13}{5}(1 - \hat{r}_t) + \frac{13}{5}(1 - q_t\hat{r}_t)\bigg) \norm{v_t}\\
        &\le \sum_{t=\floor{\varepsilon_0/4}}^{T-1} \frac{1}{h_t^2}\paren*{ \frac{2\alpha}{(t-1)^{1/7}}\paren*{14-\hat{r}_{t}}+ \frac{13}{5}(1 - \hat{r}_t) + \frac{13}{5}(1 - q_t\hat{r}_t)} \norm{v_t}.
    \end{align}
    For the second term in~\eqref{pr4init}, by \cref{lem:ad4}, whose assumption~\eqref{asm1} is satisfied by~\eqref{asm91},
    \begin{align}
        \MoveEqLeft 
        \sum_{t=\varepsilon_0}^{T-1} \sum_{\tau=t_0}^{t-2} F_t(\tau) \norm{v_{\tau+1}}^2
        \le \sum_{t=\varepsilon_0}^{T-1} \sum_{\tau=\floor{\varepsilon_0/4}}^{t-2} F_t(\tau) \norm{v_{\tau+1}}^2\\
        &\le \sum_{\tau=\floor{\varepsilon_0/4}}^{T-3}\sum_{t=(\tau+2)\vee \varepsilon_0}^{T-1} F_t(\tau) \norm{v_{\tau+1}}^2 \\
        &\le \sum_{t=\floor{\varepsilon_0/4}+1}^{T-2} \frac{12}{\alpha^2} t^{1/7} T^{1/7} \norm{v_{t}}^2.
    \end{align}
    Therefore, we obtain the desired inequality.
\end{proof}

\backmatter





\bmhead{Acknowledgements}
The authors are grateful to Takayasu Matsuo for his valuable comments.
This work was partially supported by JSPS KAKENHI (24KJ0595, 24K22290, 24K00540). 



\bmhead{Declarations}
The authors have no competing interests to declare that are relevant to the content of this article.

\bmhead{Code availability}
The source code and experiment configurations used in this paper are available at
\url{https://github.com/kanseiushiyama/velocity_control}.

\bmhead{Use of LLMs}
The authors used a large language model (LLM) to assist with writing code for the numerical experiments and proofreading the English text of the manuscript.









\begin{appendices}

\section{Computing the output}\label{sec:output}
The proposed algorithms can be implemented directly from the definition of $\bar{x}_t$ in~\eqref{wt}, but their implementation can be made more memory-efficient by computing $\bar{x}_t$ recursively.
For $t\ge2$, the algorithms maintain
\begin{equation}
X_t = \alpha t^{-1/7}\e^{-\alpha t^{6/7}}\sum_{\tau=t_0}^{t-1} \e^{\alpha \tau^{6/7}}x_\tau,\quad\text{and}\quad A_t = \alpha t^{-1/7}\e^{-\alpha t^{6/7}}\sum_{\tau=t_0}^{t-1} \e^{\alpha \tau^{6/7}},\label{defXt}
\end{equation}
so that the output can be computed as
$\bar{x}_t = X_t  / A_t$.

Note that $\bar{x}_t$ does not involve $x_t$, which may be rewritten at iteration $t+1$ and is therefore finalized only then in \cref{alg:velocity}.
The factor $\alpha t^{-1/7}\e^{-\alpha t^{6/7}}$ is introduced to counteract the growth in $X_t$ and $A_t$ caused by the exponential weights, which is important for practical implementation.
This factor is motivated by \cref{lem:ad3}, which states that $\sum_{\tau=t_0}^{t-1} \e^{\alpha \tau^{6/7}}$ is bounded above, up to lower-order terms, by $\frac76 \alpha^{-1} t^{1/7}\e^{\alpha t^{6/7}}$.

\begin{algorithm}[htbp]
\caption{Computing the output for $t \ge 2$}
\label{alg:output}
\begin{algorithmic}[1]

\State \textbf{Input:} $\alpha, x_{t-1}, t \ge 2, X_{t-1}, A_{t-1}, \tcut, X_\mathrm{cut}, A_\mathrm{cut}$

\State \Comment{$\tcut, X_\mathrm{cut}, A_\mathrm{cut}$ are maintained by the parent algorithm with initial values $1,0,0$}

\State Compute $X_t$ and $A_t$ by~\eqref{XA}

\If{$t = 2\tcut$}\label{alg3if}

    \State $X_t \gets X_t - \frac{\tcut^{1/7} \e^{\alpha \tcut^{6/7}}}{t^{1/7} \e^{\alpha t^{6/7}}} X_\mathrm{cut}$, \quad  $A_t \gets A_t - \frac{\tcut^{1/7} \e^{\alpha \tcut^{6/7}}}{t^{1/7} \e^{\alpha t^{6/7}}} A_\mathrm{cut}$

    \State $X_\mathrm{cut} \gets X_t$, \quad $A_\mathrm{cut} \gets A_t$, \quad $\tcut \gets t$

\EndIf

\State \textbf{Output:} $X_t / A_t$

\end{algorithmic}
\end{algorithm}

\begin{figure}[htbp]
\centering
\begin{tikzpicture}[scale=1.6]

    \foreach \x/\lab in {
        2/$2^{i-1}$,
        3/$2^i$,
        4/$2^{i+1}$
    }{
        \draw (\x,-0.05) -- (\x,-0.15);
        \node[below] at (\x,-0.15) {\lab};
    }

    \node[below] at (4.5,-0.2) {$\cdots$};
    \node[below] at (1.5,-0.2) {$\cdots$};

    \draw[->] (1.25,-0.1) -- (4.75,-0.1);
    \node[right] at (4.75,-0.1) {$\tau$};

    \draw[densely dotted] (2,-0.1) -- (2,0.4);
    \draw[densely dotted] (3,-0.1) -- (3,0.2);
    \draw[densely dotted] (3.5,-0.1) -- (3.5,0.4);

    \draw[thick] (2,0.45) -- (3.5,0.45);
    \filldraw (2,0.45) circle (1.5pt);
    \draw[thick, fill=white] (3.5,0.45) circle (1.5pt);
    \node[above] at (2.75,0.45) {$X_t$};
    
    \draw[thick] (2,0.2) -- (3,0.2);
    \filldraw (2,0.2) circle (1.5pt);
    \draw[thick, fill=white] (3,0.2) circle (1.5pt);
    \node[below] at (2.5,0.2) {$X_\mathrm{cut}$};

    \node[above] at (3.5,0.5) {$t$};
    \node[above] at (2,0.5) {$t_0$};
    \node[draw = black,rectangle, align = center] at (3,1.2) {Summation range\\ when $2^i < t < 2^{i+1}$};

    \begin{scope}[xshift=4cm]

        \foreach \x/\lab in {
            2/$2^{i-1}$,
            3/$2^i$,
            4/$2^{i+1}$
        }{
            \draw (\x,-0.05) -- (\x,-0.15);
            \node[below] at (\x,-0.15) {\lab};
        }

        \node[below] at (4.5,-0.2) {$\cdots$};
        \node[below] at (1.5,-0.2) {$\cdots$};

        \draw[->] (1.25,-0.1) -- (4.75,-0.1);
        \node[right] at (4.75,-0.1) {$\tau$};

        \draw[densely dotted] (2,-0.1) -- (2,0.2);
        \draw[densely dotted] (3,-0.1) -- (3,0.4);
        \draw[densely dotted] (4,0) -- (4,0.4);

        \draw[dashed, thick] (2,0.35) -- (3,0.35);
        \filldraw (2,0.35) circle (1.5pt);
        \draw[fill=white] (3,0.35) circle (1.5pt);
        \node[above] at (2.5,0.35) {discarded};

        \draw[dashed, thick] (2,0.1) -- (3,0.1);
        \filldraw (2,0.1) circle (1.5pt);
        \draw[fill=white] (3,0.1) circle (1.5pt);
        
        \draw[thick] (3,0.45) -- (4,0.45);
        \filldraw (3,0.45) circle (1.5pt);
        \draw[thick, fill=white] (4,0.45) circle (1.5pt);
        \node[above] at (3.5,0.45) {$X_t$};

        \draw[thick] (3,0.2) -- (4,0.2);
        \filldraw (3,0.2) circle (1.5pt);
        \draw[thick, fill=white] (4,0.2) circle (1.5pt);
        \node[below] at (3.5,0.2) {$X_\mathrm{cut}$};

        \node[above] at (4,0.5) {$t$};
        \node[above] at (3,0.5) {$t_0$};
        \node[draw = black,rectangle, align = center] at (3,1.2) {Summation range\\ when $t = 2^{i+1}$};

    \end{scope}

\end{tikzpicture}
\caption{Visualization of \cref{alg:output}. Filled and open circles indicate iterations that are included and excluded from the summation, respectively.}
\label{fig:output}
\end{figure}

\cref{alg:output} computes $X_t$ and $A_t$ without storing the history of $x_t$.
The algorithm is based on the observation that both quantities can be computed recursively as
\begin{equation}
X_{t} = \frac{\e^{\alpha (t-1)^{6/7}}}{t^{1/7}\e^{\alpha t^{6/7}}}\paren*{(t-1)^{1/7}X_{t-1} + \alpha x_{t-1}},\text{ and }A_{t} = \frac{\e^{\alpha (t-1)^{6/7}}}{t^{1/7}\e^{\alpha t^{6/7}}}\paren*{(t-1)^{1/7}A_{t-1} + \alpha }, \label{XA}
\end{equation}
with $X_1 = 0$ and $A_1 = 0$.
This recursion is valid except at times $t = 2^i$ $(i = 2,3,4,\ldots)$, where $t_0$ changes between iterations $t-1$ and $t$.
At such times, the algorithm first computes tentative values of $X_t$ and $A_t$ using the above recursion, and then corrects them using $X_{2^{i-1}}$ and $A_{2^{i-1}}$, denoted by $X_\mathrm{cut}$ and $A_\mathrm{cut}$ in \cref{alg:output}.
This procedure is illustrated in \cref{fig:output}.

Since the weight $w_t(\tau)$ increases exponentially with $\tau$, early iterates have little influence on the weighted average.
Therefore, changing the constant $2$ in Line~\ref{alg3if} has little effect on the actual behavior of the algorithm, although the constant appearing in the convergence rate changes accordingly, as in the continuous-time case described in \cref{rmk:const}.

\section{Proof of \cref{lem:lemcases}}\label{sec:lemcases}

\bmhead{Case 1-1a or Case 2-1}
We evaluate $F^{t}_1(u)$ for $0\le u \le m_t$.
By~\eqref{q1} and~\eqref{F1q}, we have
\begin{align}
    \MoveEqLeft F^{t}_1(u) \label{F1u} \\
    &\le \frac{28\hmax MT^{\svnth}}{\rmax \alpha^3}\frac{M_t}{12}u^3 - \paren*{\frac{24MT^{\svnth}}{\alpha^2t^{\svnon}} - 6\frac{M}{\alpha^2} t^{\svnon} T^{\svnon}}u^2  + \frac{26\alpha}{h_t^2(t-1)^{\svnon}} u\\
    &= \frac{7\hmax M_tMT^{\svnth}}{3\rmax \alpha^3}u^3 - \frac{6MT^{\svnon}}{\alpha^2t^{\svnon}}\paren*{4T^{\svntw} - t^{\svntw}}u^2 + \frac{26\alpha}{h_t^2(t-1)^{\svnon}} u\\
    &\le \frac{7\hmax M_tMT^{\svnth}}{3\rmax \alpha^3}u^3 - \frac{18MT^{\svnth}}{\alpha^2t^{\svnon}}u^2 + \frac{26\alpha}{h_t^2(t-1)^{\svnon}} u.
\end{align}

If $M_t = 0$, then the cubic term vanishes and $F^{t}_1(u)$ is bounded above by the maximum of a quadratic function:
\begin{align}
    F^{t}_1(u) &\le - \frac{18MT^{\svnth}}{\alpha^2t^{\svnon}}\paren*{u - \frac{13\alpha^3 t^{\svnon}}{18h_t^2M(t-1)^{\svnon} T^{\svnth}}}^2 + \frac{169\alpha^4 \paren*{\frac{t}{t-1}}^{\svntw}}{18 h_t^4 M t^{\svnon} T^{\svnth}} \label{Mt0} \\
    &\le \frac{169\alpha^4}{18 h_t^4 M t^{\svnon} T^{\svnth}} + \Order{t^{-\svnei}T^{-\svnth}}.
\end{align}

Otherwise, it remains to evaluate this cubic function for $u \le m_t < +\infty$.
We use the following lemma, which applies to general cubic functions.
\begin{lemma}\label{lem:cube}
    Let $c_3,c_2,c_1>0$ satisfy $c_2^2 \ge 3c_1c_3$, and consider the cubic function
    \[
        g(u) = c_3 u^3 - c_2u^2 + c_1u.
    \]
    Then, for $0 \le u \le \frac{c_2}{c_3} - \frac{2c_1}{c_2}$, it holds that
    \begin{equation}
        g(u) \le \frac{c_1^2}{3c_2}\paren*{1+\frac{3c_1c_3}{c_2^2}}.
    \end{equation}
\end{lemma}

Although the proof is deferred to \cref{app:cube}, \cref{fig:cube} provides an intuitive visualization.
\begin{figure}
    \centering
    \begin{tikzpicture}[scale=0.8]
    \begin{axis}[
        axis lines = middle,
        xmin = -0.2, xmax = 2.4,
        ymin = -0.15, ymax = 0.8,
        samples = 200,
        domain = -0.05:2.3,
        xtick = {1.95},
        xticklabels = {$\frac{c_2}{c_3}-\frac{2c_1}{c_2}$},
        ytick = {0.66},
        yticklabels = {$\frac{c_1^2}{3c_2}\paren*{1+\frac{3c_1c_3}{c_2^2}}$},
        tick style={draw=none},
        tick label style = {font=\small},
        xlabel style = {at={(axis description cs:1,0.46)}, anchor=west},
        ylabel style = {at={(axis description cs:0.48,1)}, anchor=south},
        clip = false,
    ]
    \addplot[thick, blue]
        {5*x^3/8 - 35*x^2/16 + 2*x}node[pos=0.73, above] {$g(u)$};
   
    \pgfmathsetmacro{\xmaxpt}{7/6 - sqrt(265)/30}
    \pgfmathsetmacro{\ymaxpt}{301/864 + 53*sqrt(265)/4320}
    
    \pgfmathsetmacro{\xother}{7/6 + sqrt(265)/15}
    
    \addplot[dashed] coordinates {(-0.2,\ymaxpt) (\xother,\ymaxpt)};
    
    \addplot[dashed] coordinates {(\xother,-0.1) (\xother,\ymaxpt)};
    
    \addplot[only marks, mark=*, mark size=1.5pt] coordinates {(\xmaxpt,\ymaxpt)(\xother,\ymaxpt)};

    \addplot[only marks, mark=*, mark size=2pt]
    coordinates {(1.95,0)};
    
    \addplot[only marks, mark=*, mark size=2pt]
    coordinates {(0,0.66)};

    \node at (axis cs:0,0) [below right] {$0$};

    \node at (axis cs:2.38,0) [below] {$u$};
    \end{axis}
    \end{tikzpicture}
    \caption{Visualization of \cref{lem:cube}.}
    \label{fig:cube}
\end{figure}

We now apply this lemma to the right-hand side of~\eqref{F1u}.
The condition $c_2^2 \ge 3c_1c_3$ is satisfied for sufficiently large $T$.
To ensure that $m_t \le \frac{c_2}{c_3} - \frac{2c_1}{c_2}$, we take $T$ sufficiently large so that, for all $\floor{\varepsilon_0/4} \le t \le T$, the following upper bound is at most $8/9$\footnote{When $\alpha = 0.1$, $\rmax = 0.5$, and $\varepsilon_0 = 40$, this holds unconditionally for \cref{alg:velocityC}.}:
\begin{equation}
    \frac{91\hmax \alpha^2(\frac{t}{t-1})^{\svnon}t^{\svnon}}{243\rmax h_t^2T^{\svnth}} \le 
    \begin{cases}   
    \frac{91 \alpha^2(\frac{t}{t-1})^{\svnon}t^{\svnon}}{243\rmax T^{\svnth}} & \text{(\cref{alg:velocityC})}\\ 
    \frac{91\alpha^2(\frac{t}{t-1})^{\svnon}t^{\svnon}}{243\rmax T^{\svnth}}\paren*{\frac{\hmax \Lmax}{4(1-\rmax)}\vee1} & \text{(\cref{alg:velocity})}
    \end{cases}
    \le \frac{8}{9},
\end{equation}
where the first inequality follows from $h_t^2 = \hmax$ in~\cref{alg:velocityC} and $h_t^{-2} \le \frac{\Lmax}{4(1-\rmax)}\vee \frac{1}{\hmax}$ in \cref{alg:velocity}.
Then, since $M_t \le M$ by~\eqref{ests}, it holds that
\begin{align}
    \MoveEqLeft \frac{c_2}{c_3}-\frac{2c_1}{c_2} = \frac{\frac{18MT^{\svnth}}{\alpha^2t^{\svnon}}}{\frac{7\hmax M_tMT^{\svnth}}{3\rmax\alpha^3}} - \frac{2\frac{26\alpha}{h_t^2(t-1)^{\svnon}}}{\frac{18MT^{\svnth}}{\alpha^2t^{\svnon}}}
    = \frac{54\rmax \alpha}{7\hmax M_t t^{\svnon}} -  \frac{26\alpha^3(\frac{t}{t-1})^{\svnon}}{9h_t^2MT^{\svnth}}\\
    &= \frac{54\rmax\alpha}{7\hmax M_tt^{\svnon}}\paren*{1 - \frac{M_t}{M} \frac{91\hmax \alpha^2(\frac{t}{t-1})^{\svnon}t^{\svnon}}{243\rmax h_t^2T^{\svnth}}}\\
    &\ge \frac{54\rmax \alpha }{7\hmax M_tt^{\svnon}}\paren*{1 - \frac{8}{9}} = \frac{6\rmax \alpha}{7\hmax M_tt^{\svnon}} = m_t.
\end{align}
Therefore, by~\cref{lem:cube}, we have
\[
    F_1^t(u) \le \frac{c_1^2}{3c_2}\paren*{1+\frac{3c_1c_3}{c_2^2}}
\]
for $u \le m_t$.
Since
\begin{equation}
    \frac{c_1^2}{3c_2} = \frac{\paren*{\frac{26\alpha}{h_t^2(t-1)^{\svnon}}}^2}{3\frac{18MT^{\svnth}}{\alpha^2t^{\svnon}}} = \frac{338\alpha^4 (\frac{t}{t-1})^{\svntw}}{27h_t^4 M t^{\svnon}T^{\svnth}}
    = \frac{338\alpha^4}{27h_t^4 M t^{\svnon}T^{\svnth}}
 + \Order{t^{-\svnei}T^{-\svnth}}
\end{equation}
and
\begin{equation}
    \frac{3c_1c_3}{c_2^2} = \frac{3\frac{26\alpha}{h_t^2(t-1)^{\svnon}}\frac{7\hmax M_tMT^{\svnth}}{3\rmax \alpha^3}}{\paren*{\frac{18MT^{\svnth}}{\alpha^2t^{\svnon}}}^2} = \frac{91\hmax\alpha^2M_t(\frac{t}{t-1})^{\svnon}t^{\svnon}}{162\rmax h_t^2MT^{\svnth}}
    = \Order{t^{\svnon}T^{-\svnth}},
\end{equation}
we obtain
\begin{equation}
    \max_{\norm{v_t} \le m_t} F^t_1(\norm{v_t}) \le \frac{338\alpha^4}{27h_t^4M t^{\svnon}T^{\svnth}} + \Order{T^{-\svnsi}\vee t^{-\svnei}T^{-\svnth}}.
\end{equation}
The case $M_t=0$ in~\eqref{Mt0} can be incorporated into this bound because $\frac{338}{27} > \frac{169}{18}$.

\bmhead{Case 1-1b}
We evaluate $F^t_1(u)$ for $m_t/2 \le u \le 2m_t$.
Since this case occurs only in \cref{alg:velocityC}, we set $M_t = M$, $\hmax = h_t^2$, and $m_t = \frac{6\rmax \alpha}{7h_t^2 Mt^{\svnon}} < \infty$.
By a calculation similar to~\eqref{F1u}, with the only difference being that we use~\eqref{q2} instead of~\eqref{q1}, we have
\begin{align}
    & F^{t}_1(u) / u\\
    &\le \frac{7\hmax M_tMT^{\svnth}}{3\rmax \alpha^3}(2m_t)^2 - \frac{18MT^{\svnth}}{\alpha^2t^{\svnon}}\frac{m_t}{2} + \frac{1}{h_t^2}\paren*{ \frac{28\alpha}{(t-1)^{\svnon}}+ \frac{26}{5}}\\
    &= \frac{7h_t^2 M^2T^{\svnth}}{3\rmax \alpha^3}\paren*{\frac{12\rmax \alpha}{7h_t^2 Mt^{\svnon}}}^2 - \frac{18MT^{\svnth}}{\alpha^2t^{\svnon}}\frac{3\rmax \alpha}{7h_t^2 Mt^{\svnon}} + \frac{1}{h_t^2}\paren*{ \frac{28\alpha}{(t-1)^{\svnon}}+ \frac{26}{5}}\\
    &= \frac{48\rmax T^{\svnth}}{7h_t^2 \alpha t^{\svntw}} - \frac{54\rmax T^{\svnth}}{7h_t^2 \alpha t^{\svntw}}+ \frac{1}{h_t^2}\paren*{ \frac{28\alpha}{(t-1)^{\svnon}}+ \frac{26}{5}}\\
    &\le \frac{1}{h_t^2}\paren*{-\frac{6\rmax T^{\svnon}}{7 \alpha} +  \frac{28\alpha}{(\floor{\varepsilon_0/4}-1)^{\svnon}}+\frac{26}{5}},
\end{align}
where the last inequality uses $\floor{\varepsilon_0/4} \le t \le T$.
Thus, for sufficiently large $T$, we have $F_1^{t}(u) < 0$ for $m_t/2 \le u \le 2m_t$\footnote{When $\alpha = 0.1$, $\rmax = 0.5$, and $\varepsilon_0 = 40$, this is satisfied by $T \ge 40$.}.

\bmhead{Case 2-2a}
We evaluate $F_2^{t}(u)$ for $u \le m_t$.
This case occurs only if $\hat{r}_t = 0$.
Thus, by~\eqref{hL4} and \cref{lem:ad0}, the coefficient of $u^2$ in $F^{t-}_2$ divided by $-\lambda$ is 
\[
    \frac{2(1-\hat{r}_t)+a_t}{h_t^2} - \frac{L_t}{2} = \frac{2}{h_t^2} - \frac{L_t}{2} + \frac{a_t}{h_t^2} \ge \frac{2\rmax}{\hmax} + \frac{6\alpha}{7h_t^2t^{\svnon}}.
\]
Hence, since $\rmax \le 1$ and $h_t^2 \le \hmax$, the coefficient of $u^2$ in $F^t_2$ is
\begin{align}
    \MoveEqLeft
    -\frac{28\hmax MT^{\svnth}}{\rmax \alpha^3}\paren*{\frac{2\rmax}{\hmax} + \frac{6\alpha}{7h_t^2t^{\svnon}}} + 6\frac{M}{\alpha^2} t^{\svnon} T^{\svnon}\\
    &\le -\frac{56MT^{\svnth}}{\alpha^3} - \frac{6MT^{\svnon}}{\alpha^2t^{\svnon}}\paren*{\frac{4\hmax}{\rmax h_t^2}T^{\svntw} - t^{\svntw} } \le -\frac{56MT^{\svnth}}{\alpha^3}.
\end{align}
Thus,
\begin{align}
    \MoveEqLeft
    F_2^{t}(u) \le -\frac{56MT^{\svnth}}{\alpha^3}u^2 + \frac{1}{h^2_t} \paren*{\frac{28\alpha}{(t-1)^{\svnon}}+ \frac{26}{5}}u\\
    &\le \frac{\frac{1}{h^4_t} \paren*{\frac{28\alpha}{(t-1)^{\svnon}}+ \frac{26}{5}}^2}{4 \frac{56MT^{\svnth}}{\alpha^3}} \le \frac{\alpha^3(14\alpha + \frac{13}{5})^2 }{56h_t^4 M T^{\svnth}} = \Order{T^{-\svnth}}.
\end{align}

\bmhead{Case 1-2 or Case 2-2b}
We evaluate $F_2^{t}(u)$ for $u \ge m_t/2$.
By~\eqref{q2} and~\eqref{F2q}, and since $\floor{\varepsilon_0/4} \le t \le T$, $\rmax \le 1$, $h_t^2 \le \hmax$, and $M_t \le M$, we have
\begin{align}
    F^{t}_2(u) &\le -\paren*{ \frac{28\hmax MT^{\svnth}}{\rmax \alpha^3} \frac{6\alpha}{7h_t^2t^{\svnon}} - 6\frac{M}{\alpha^2} t^{\svnon} T^{\svnon}}u^2 + \frac{1}{h_t^2}\paren*{ \frac{28\alpha}{(t-1)^{\svnon}}+ \frac{26}{5}}u \\
    &= \frac{u}{h_t^2}\sqparen*{ -\paren*{ \frac{2 T^{\svntw}}{\rmax t^{\svntw}} - \frac{h_t^2}{2\hmax}}\frac{12\hmax Mt^{\svnon}T^{\svnon}}{\alpha^2}u + \frac{28\alpha}{(t-1)^{\svnon}}+ \frac{26}{5}}\\
    &\le \frac{u}{h_t^2}\sqparen*{ -\paren*{ 2 - \frac12}\frac{12\hmax Mt^{\svnon}T^{\svnon}}{\alpha^2}\frac{3\rmax \alpha}{7\hmax M_tt^{\svnon}} + \frac{28\alpha}{(t-1)^{\svnon}}+ \frac{26}{5}}\\
    &\le \frac{T^{\svnon}u}{h_t^2}\sqparen*{ -\frac{54\rmax}{7\alpha}+  \frac{1}{T^{\svnon}}\paren*{\frac{28\alpha}{(\floor{\varepsilon_0/4}-1)^{\svnon}}+ \frac{26}{5}}}.
\end{align}
We take $T$ sufficiently large so that the right-hand side is bounded above by $-{cT^{\svnon}u}/{h_t^2}$ for some constant $c > 0$\footnote{When $\alpha = 0.1$, $\rmax = 0.5$, and $\varepsilon_0 = 40$, this holds unconditionally with $c = 31$.}.
Then, for any $u \ge m_t/2$,
\begin{equation}
    F^{t}_2(u) 
    \le -c\frac{T^{\svnon}}{h_t^2}\frac{3\rmax \alpha}{7\hmax M_tt^{\svnon}}
    \le -c\frac{3\rmax \alpha}{7h_{\mathrm{max}}^4 M } = -\Omega(1).
\end{equation}

\section{Proof of \cref{lem:UT}}\label{sec:lemUT}
\begin{proof}[Proof of \cref{lem:UT}.]
    Let $N_\mathrm{a}$ and $N_\mathrm{b}$ denote the numbers of iterations $t\in\{\floor{\varepsilon_0/4},\ldots,T-1\}$ in which Case 2-2a and Case 2-2b occur, respectively.
    Let $N_+$ and $N_-$ denote the numbers of indices $t$ in the same range for which $L_t < L_{t+1}$ and $L_t > L_{t+1}$, respectively.
    We first observe the following relations among these quantities.
    \begin{itemize}
        \item Since the iteration $t$ falls under Case 2-2a only when $L_t < L_{t+1}$, we have $N_\mathrm{a} \le N_+$. 
        \item Since $L_t > L_{t+1}$ holds only when the iteration $t$ falls under Case 2-2b, we have $N_- \le N_\mathrm{b}$. 
        \item Let $\delta > 0$ be the minimum value of
        $
            \setI{\beta_{\mathrm{dec}}\beta_{\mathrm{inc}}^j > 1}{j = 1,2,\ldots}.
        $
        Since $\delta L_{t} \le L_{t+1}$ whenever $L_t < L_{t+1}$, and since $\beta_{\mathrm{dec}}L_{t} \le L_{t+1}$ whenever $L_t > L_{t+1}$, we have
        $
            \delta^{N_+}\beta_{\mathrm{dec}}^{N_-}L_{\floor{\varepsilon_0/4}} \le L_T
        $.
        \item By~\eqref{ests}, we have $L_t \le \Lmax$ for all $t$.
    \end{itemize}
    Combining these relations gives
    \[
        \delta^{N_\mathrm{a}}\beta_{\mathrm{dec}}^{N_\mathrm{b}}L_{\floor{\varepsilon_0/4}} \le \Lmax.
    \]
    Hence,
    \[
        N_\mathrm{a}
        \le
        N_\mathrm{b} \log_{\delta} \paren*{\frac{1}{\beta_\mathrm{dec}}}
        +
        \log_{\delta}\paren*{\frac{\Lmax}{L_{\floor{\varepsilon_0/4}}}}
        =
        (\mathrm{const.}) N_\mathrm{b} + \mathrm{const.}
    \]
    Therefore, for sufficiently large $T$,
    \[
        U_T
        =
        N_\mathrm{a}\Order{T^{-3/7}} - N_\mathrm{b}\Omega(1)
        \le
        \Order{T^{-3/7}}.
    \]
\end{proof}

\section{Proof of \cref{lem:cube}}\label{app:cube}

\begin{proof}[Proof of \cref{lem:cube}.]
    The derivative of $g$ is $g'(u) = 3c_3u^2 - 2c_2u + c_1$, and the critical points of $g$ are
    \[
    u_- = \frac{c_2 - \Delta}{3c_3}, \quad u_+ = \frac{c_2 + \Delta}{3c_3},
    \]
    where $\Delta = \sqrt{c_2^2 - 3 c_3c_1}$.
    Since $g$ increases in $[0, u_-]$, decreases in $[u_-,u_+]$, and increases again in $[u_+,\infty)$, and since $g\paren*{\frac{c_2 + 2\Delta}{3c_3}} = g(u_-)$, $g(u) \le g(u_-)$ holds for $0 \le u \le \frac{c_2 + 2\Delta}{3c_3}$.
    Therefore, the claim follows from
    \begin{align}
        g(u_-) &= \paren*{\frac{u_-}{3} - \frac{c_2}{9c_3}}g'(u_-) - \frac{2(c_2^2 - 3c_1c_3)}{9c_3}u_- + \frac{c_1c_2}{9c_3}\\
        &= \frac{2c_1^2}{3(c_2 + \sqrt{c_2^2-3c_3c_1})} -\frac{c_2c_1}{9c_3}\paren*{\frac{2c_2}{c_2 + \sqrt{c_2^2-3c_3c_1}}-1}\\
        &\le \frac{2c_1^2}{3(c_2 + \sqrt{c_2^2-3c_3c_1})} \le \frac{c_1^2}{3c_2}\paren*{1+\frac{3c_3c_1}{c_2^2}},
    \end{align}
    and
    \begin{equation}
        \frac{c_2}{c_3} - \frac{2c_1}{c_2} \le \frac{c_2 + 2\Delta}{3c_3}.
    \end{equation}
\end{proof}

\section{Modified algorithm}\label{app:predictive}
$\mathtt{Predictive VC}$ used in \cref{sec:exp} incorporates the following two modifications.

\paragraph{Deferred evaluation of $x^r_t$.}
The function value $f(x_t^r)$ is evaluated only when the choice $\hat r_t=1$ is rejected.
This reduces the number of oracle calls when $\hat r_t=1$ is frequently selected.
The procedure is presented in~\cref{alg:defR}.

\begin{algorithm}[htbp]
\caption{VC iteration with deferred evaluation of the candidate $x_t^r$}
\begin{algorithmic}[1]
\State Perform the original backtracking and previous-state rewrite,
       checking $E^-_{L_t}\le0$ only at $x_t^1$ and $x_t^0$.
\State Compute $M_t$ and $m_t$ using $x_t^1$.
\If{$\|v_t^1\|\le m_t$}
    \State $\hat r_t\gets1$, \quad $L_{t+1}\gets L_t$
    \Comment{$x^r_t$ is not used}
\Else
    \If{$E^-_{L_t}(x_t^r,x_{t-1})\le0$ and $E^+_{L_t}(x_t^r,x_{t-1})\le0$}
    \Comment{Evaluate $f(x_t^r)$ and $\nabla f(x_t^r)$}
        \State $\hat r_t\gets r$
    \Else
    \State $\hat r_t\gets0$
    \EndIf
    \State $L_{t+1}\gets\beta_{\mathrm{dec}}L_t$
\EndIf
\State $(v_t,x_t)\gets(v_t^{\hat r_t},x_t^{\hat r_t})$.
\end{algorithmic}
\label{alg:defR}
\end{algorithm}

\paragraph{Prediction with checkpoints.}
After $T$ consecutive selections of $\hat r_t=1$, we enter a prediction phase in which we predict that this choice will continue.
During this phase, we check $E^-_{L_t}(x_t^1,x_{t-1})\le0$ and $\norm{v_t^1}\le m_t$ at every iteration, while skipping evaluations of the other candidates except at designated checkpoints.
At each checkpoint, we additionally check the candidate $x_t^0$ and save the updated state if all required checks pass.
If any check fails, we restore the most recent checkpoint, discard the subsequent trajectory, and recompute the steps with prediction disabled.
This procedure can reduce the number of oracle calls when $\hat r_t=1$ continues to be selected.
In the numerical experiments, we use $T=10$.
The procedure is presented in~\cref{alg:check}.

\begin{algorithm}[htbp]
\caption{A prediction phase after $\hat r_t = 1$ is consecutively selected $T$ times}
\begin{algorithmic}[1]
\State Save the current state as a checkpoint.
\While{the stopping criterion is not satisfied}
    \State Check $E^-_{L_t}(x_t^1,x_{t-1})\le0$ and $\|v_t^1\|\le m_t$.
    \Comment{Evaluate $f(x_t^1)$ and $\nabla f(x_t^1)$.}
    \If{the consecutive-selection count reaches $T+j^2$, $j=1,2,\ldots$}
    \Comment{Checkpoint}
        \State Also check $E^-_{L_t}(x_t^0,x_{t-1})\le0$.
    \EndIf
    \If{any required check fails}
        \State Restore the latest checkpoint and discard the subsequent trajectory.
        \State Replay using ordinary steps, with prediction disabled.
        \State \Return to ordinary execution; prediction may resume when eligible.
    \EndIf
    \State $(\hat r_t,v_t,x_t)\gets(1,v_t^1,x_t^1)$, \quad $L_{t+1}\gets L_t$.
    \State If this is a checkpoint step, save the complete updated state.
    \State $t\gets t+1$.
\EndWhile
\end{algorithmic}
\label{alg:check}
\end{algorithm}






\end{appendices}


\bibliography{references}

@article {P64,
    AUTHOR = {Polyak, B. T.},
     TITLE = {Some methods of speeding up the convergence of iteration methods},
     journal = {USSR Comput. Math. Math. Phys.},
fjournal = {USSR Computational Mathematics and Mathematical Physics},
volume = {4},
number = {5},
pages = {1--17},
year = {1964},
issn = {0041-5553},
doi = {10.1016/0041-5553(64)90137-5},
url = {https://www.sciencedirect.com/science/article/pii/0041555364901375},
}

@book {N18b,
    AUTHOR = {Nesterov, Yurii},
     TITLE = {Lectures on Convex Optimization},
    SERIES = {Springer Optimization and Its Applications},
    VOLUME = {137},
 PUBLISHER = {Springer},
   ADDRESS = {Cham},
      YEAR = {2018},
      ISBN = {978-3-319-91577-7; 978-3-319-91578-4},
   MRCLASS = {90-01 (90C25)},
  MRNUMBER = {3839649},
MRREVIEWER = {Giorgio Giorgi},
       DOI = {10.1007/978-3-319-91578-4},
  EDITION = {2},
}

@article {SBC16,
    AUTHOR = {Su, Weijie and Boyd, Stephen and Cand\`es, Emmanuel J.},
     TITLE = {A differential equation for modeling {N}esterov's accelerated
              gradient method: theory and insights},
   JOURNAL = {J. Mach. Learn. Res.},
  FJOURNAL = {Journal of Machine Learning Research (JMLR)},
    VOLUME = {17},
      YEAR = {2016},
      number = {153},
     PAGES = {1--43},
     url     = {http://jmlr.org/papers/v17/15-084.html}
}

@article {N83,
    AUTHOR = {Nesterov, Yu. E.},
     TITLE = {A method for solving the convex programming problem with convergence rate {$O(1/k^{2})$}},
   JOURNAL = {Dokl. Akad. Nauk SSSR},
  FJOURNAL = {Doklady Akademii Nauk SSSR},
    VOLUME = {269},
      YEAR = {1983},
    NUMBER = {3},
     PAGES = {543--547},
      ISSN = {0002-3264},
   MRCLASS = {90C25},
  MRNUMBER = {701288},
MRREVIEWER = {R. \c{S}erban},
}

@article {WRJ21,
    AUTHOR = {Wilson, Ashia C. and Recht, Ben and Jordan, Michael I.},
     TITLE = {A {L}yapunov analysis of accelerated methods in optimization},
   JOURNAL = {J. Mach. Learn. Res.},
  FJOURNAL = {Journal of Machine Learning Research (JMLR)},
    VOLUME = {22},
      YEAR = {2021},
      number = {113},
     PAGES = {1--34},
url     = {http://jmlr.org/papers/v22/20-195.html}
}

@InProceedings{SRR22,
  title = 	 {Continuous-time analysis of accelerated gradient methods via conservation laws in dilated coordinate systems},
  author =       {Suh, Jaewook J and Roh, Gyumin and Ryu, Ernest K},
  booktitle = 	 {Proceedings of the 39th International Conference on Machine Learning},
  pages = 	 {20640--20667},
  year = 	 {2022},
  url = 	 {https://proceedings.mlr.press/v162/suh22a.html},
  volume = {162},
}

@inproceedings{KBB15,
 author = {Krichene, Walid and Bayen, Alexandre and Bartlett, Peter L},
 booktitle = {Advances in Neural Information Processing Systems},
 pages = {},
 title = {Accelerated Mirror Descent in Continuous and Discrete Time},
 volume = {28},
 year = {2015}
}

@article{ADA21,
url = {http://dx.doi.org/10.1561/2400000036},
year = {2021},
volume = {5},
journal = {Foundations and Trends in Optimization},
title = {Acceleration Methods},
doi = {10.1561/2400000036},
issn = {2167-3888},
number = {1-2},
pages = {1-245},
author = {Alexandre d'Aspremont and Damien Scieur and Adrien Taylor}
}

@article {DT14,
    AUTHOR = {Drori, Yoel and Teboulle, Marc},
     TITLE = {Performance of first-order methods for smooth convex
              minimization: a novel approach},
   JOURNAL = {Math. Program.},
  FJOURNAL = {Mathematical Programming},
    VOLUME = {145},
      YEAR = {2014},
    NUMBER = {1-2, Ser. A},
     PAGES = {451--482},
      ISSN = {0025-5610},
   MRCLASS = {90C60 (68Q25 90C22 90C25)},
  MRNUMBER = {3207695},
MRREVIEWER = {Klaus Meer},
       DOI = {10.1007/s10107-013-0653-0},
       URL = {https://doi.org/10.1007/s10107-013-0653-0},
}

@article {DT221,
    AUTHOR = {Taylor, Adrien and Drori, Yoel},
     TITLE = {An optimal gradient method for smooth strongly convex minimization},
   JOURNAL = {Math. Program.},
  FJOURNAL = {Mathematical Programming},
      YEAR = {2023},
     PAGES = {557--594},
      ISSN = {1436-4646},
       DOI = {10.1007/s10107-022-01839-y},
       URL = {https://doi.org/10.1007/s10107-022-01839-y},
  VOLUME = {199},
}

@article {KF16,
    AUTHOR = {Kim, Donghwan and Fessler, Jeffrey A.},
     TITLE = {Optimized first-order methods for smooth convex minimization},
   JOURNAL = {Math. Program.},
  FJOURNAL = {Mathematical Programming},
    VOLUME = {159},
      YEAR = {2016},
    NUMBER = {1-2, Ser. A},
     PAGES = {81--107},
      ISSN = {0025-5610},
   MRCLASS = {90C60 (49M25 68Q25 90C20 90C22 90C25)},
  MRNUMBER = {3535919},
       DOI = {10.1007/s10107-015-0949-3},
       URL = {https://doi-org.utokyo.idm.oclc.org/10.1007/s10107-015-0949-3},
}

@article {MTB23,
    AUTHOR = {Moucer, C\'{e}line and Taylor, Adrien and Bach, Francis},
     TITLE = {A Systematic Approach to {L}yapunov Analyses of
              Continuous-Time Models in Convex Optimization},
   JOURNAL = {SIAM J. Optim.},
  FJOURNAL = {SIAM Journal on Optimization},
    VOLUME = {33},
      YEAR = {2023},
    NUMBER = {3},
     PAGES = {1558--1586},
      ISSN = {1052-6234,1095-7189},
   MRCLASS = {90C25 (34D20 90C22 90C30)},
  MRNUMBER = {4619884},
       DOI = {10.1137/22M1498486},
       URL = {https://doi.org/10.1137/22M1498486},
}

@inproceedings{USM23,
  title={A unified discretization framework for differential equation approach with {Lyapunov} arguments for convex optimization},
  author={Ushiyama, Kansei and Sato, Shun and Matsuo, Takayasu},
  booktitle={Advances in Neural Information Processing Systems},
  volume={36},
  year={2023}
}

@inproceedings{KY23b,
  title={Convergence analysis of {ODE} models for accelerated first-order methods via positive semidefinite kernels},
  author={Kim, Jungbin and Yang, Insoon},
  booktitle={Advances in Neural Information Processing Systems},
  volume={36},
  year={2023},
  pages = {63023--63035},
}

@article {KF21,
    AUTHOR = {Kim, Donghwan and Fessler, Jeffrey A.},
     TITLE = {Optimizing the efficiency of first-order methods for
              decreasing the gradient of smooth convex functions},
   JOURNAL = {J. Optim. Theory Appl.},
  FJOURNAL = {Journal of Optimization Theory and Applications},
    VOLUME = {188},
      YEAR = {2021},
    NUMBER = {1},
     PAGES = {192--219},
      ISSN = {0022-3239,1573-2878},
   MRCLASS = {90C25 (68Q25 90C22 90C30 90C60)},
  MRNUMBER = {4200948},
       DOI = {10.1007/s10957-020-01770-2},
       URL = {https://doi.org/10.1007/s10957-020-01770-2},
}

@article {THG17,
    AUTHOR = {Taylor, Adrien B. and Hendrickx, Julien M. and Glineur,
              Fran\c{c}ois},
     TITLE = {Smooth strongly convex interpolation and exact worst-case
              performance of first-order methods},
   JOURNAL = {Math. Program.},
  FJOURNAL = {Mathematical Programming},
    VOLUME = {161},
      YEAR = {2017},
    NUMBER = {1-2},
     PAGES = {307--345},
      ISSN = {0025-5610,1436-4646},
   MRCLASS = {90C25 (68Q25 90C22 90C30 90C60)},
  MRNUMBER = {3592780},
       DOI = {10.1007/s10107-016-1009-3},
       URL = {https://doi.org/10.1007/s10107-016-1009-3},
}

@article{GMGH22,
  title={{PEPit}: computer-assisted worst-case analyses of first-order optimization methods in {Python}},
  author={Goujaud, Baptiste and Moucer, C{\'e}line and Glineur, Fran{\c c}ois and Hendrickx, Julien M. and Taylor, Adrien B. and Dieuleveut, Aymeric},
  journal={Mathematical Programming Computation},
  year={2024},
  volume = {16},
  pages = {337--367},
  doi = {10.1007/s12532-024-00259-7},
}

@InProceedings{TVL18b,
  title = 	 {{Lyapunov} functions for first-order methods: Tight automated convergence guarantees},
  author =       {Taylor, Adrien and Van Scoy, Bryan and Lessard, Laurent},
  booktitle = 	 {Proceedings of the 35th International Conference on Machine Learning},
  year = 	 {2018},
  url = 	 {https://proceedings.mlr.press/v80/taylor18a.html},
  series = {Proceedings of Machine Learning Research},
  volume = {80},
  pages = {4897--4906},
}

@INPROCEEDINGS{THG17c,
  author={Taylor, Adrien B. and Hendrickx, Julien M. and Glineur, François},
  booktitle={2017 IEEE 56th Annual Conference on Decision and Control (CDC)}, 
  title={Performance estimation toolbox ({PESTO}): Automated worst-case analysis of first-order optimization methods}, 
  year={2017},
  volume={},
  number={},
  pages={1278-1283},
  doi={10.1109/CDC.2017.8263832}}

@InProceedings{TB19,
  title = 	 {Stochastic first-order methods: non-asymptotic and computer-aided analyses via potential functions},
  author =       {Taylor, Adrien and Bach, Francis},
  booktitle = 	 {Proceedings of the Thirty-Second Conference on Learning Theory},
  pages = 	 {2934--2992},
  year = 	 {2019},
  volume = 	 {99},
  month = 	 {25--28 Jun},
  url = 	 {https://proceedings.mlr.press/v99/taylor19a.html},
}

@article {SDJS22,
    AUTHOR = {Shi, Bin and Du, Simon S. and Jordan, Michael I. and Su,
              Weijie J.},
     TITLE = {Understanding the acceleration phenomenon via high-resolution
              differential equations},
   JOURNAL = {Math. Program.},
  FJOURNAL = {Mathematical Programming},
    VOLUME = {195},
      YEAR = {2022},
    NUMBER = {1-2},
     PAGES = {79--148},
      ISSN = {0025-5610,1436-4646},
   MRCLASS = {90C25 (34E10 65K10 65L20 90C35)},
  MRNUMBER = {4499055},
       DOI = {10.1007/s10107-021-01681-8},
       URL = {https://doi.org/10.1007/s10107-021-01681-8},
}

@techreport{USM24,
  title={Deriving optimal rates of continuous-time accelerated first-order methods via performance estimation problems},
  author={Ushiyama, Kansei and Sato, Shun and Matsuo, Takayasu},
  year={2024},
  institution = {The University of Tokyo},
  number = {METR 2024-02},
  url = {https://www.keisu.t.u-tokyo.ac.jp/data/2024/METR24-02.pdf},
}

@article{OMT26,
  author  = {Okamura, Kaito and Marumo, Naoki and Takeda, Akiko},
  title   = {Heavy-Ball Differential Equation Achieves {$O(\varepsilon^{-7/4})$} Convergence for Nonconvex Functions},
  journal = {Optimization Letters},
  volume  = {20},
  number  = {5},
  pages   = {993--1010},
  year    = {2026},
  doi     = {10.1007/s11590-025-02270-w},
  url     = {https://doi.org/10.1007/s11590-025-02270-w}
}

@article {MT24,
    AUTHOR = {Marumo, Naoki and Takeda, Akiko},
     TITLE = {Parameter-free accelerated gradient descent for nonconvex
              minimization},
   JOURNAL = {SIAM J. Optim.},
  FJOURNAL = {SIAM Journal on Optimization},
    VOLUME = {34},
      YEAR = {2024},
    NUMBER = {2},
     PAGES = {2093--2120},
      ISSN = {1052-6234,1095-7189},
   MRCLASS = {90C26 (65K05 90C52)},
  MRNUMBER = {4759615},
       DOI = {10.1137/22M1540934},
       URL = {https://doi.org/10.1137/22M1540934},
}

@article {MT24b,
    AUTHOR = {Marumo, Naoki and Takeda, Akiko},
     TITLE = {Universal heavy-ball method for nonconvex optimization under
              {H}\"older continuous {H}essians},
   JOURNAL = {Math. Program.},
  FJOURNAL = {Mathematical Programming},
    VOLUME = {212},
      YEAR = {2025},
    NUMBER = {1-2},
     PAGES = {147--175},
      ISSN = {0025-5610,1436-4646},
   MRCLASS = {90C26 (65K05 90C06 90C60)},
  MRNUMBER = {4921734},
       DOI = {10.1007/s10107-024-02100-4},
       URL = {https://doi.org/10.1007/s10107-024-02100-4},
}

@InProceedings{CDHS17,
  title = 	 {``{C}onvex Until Proven Guilty'': Dimension-Free Acceleration of Gradient Descent on Non-Convex Functions},
  author =       {Yair Carmon and John C. Duchi and Oliver Hinder and Aaron Sidford},
  booktitle = 	 {Proceedings of the 34th International Conference on Machine Learning},
  pages = 	 {654--663},
  year = 	 {2017},
  volume = 	 {70},
  month = 	 {06--11 Aug},
  url = 	 {https://proceedings.mlr.press/v70/carmon17a.html},
}

@article {CDHS18,
    AUTHOR = {Carmon, Yair and Duchi, John C. and Hinder, Oliver and
              Sidford, Aaron},
     TITLE = {Accelerated methods for nonconvex optimization},
   JOURNAL = {SIAM J. Optim.},
  FJOURNAL = {SIAM Journal on Optimization},
    VOLUME = {28},
      YEAR = {2018},
    NUMBER = {2},
     PAGES = {1751--1772},
      ISSN = {1052-6234,1095-7189},
   MRCLASS = {90C06 (65K05 90C26 90C30)},
  MRNUMBER = {3814027},
MRREVIEWER = {Georgina\ Hall},
       DOI = {10.1137/17M1114296},
       URL = {https://doi.org/10.1137/17M1114296},
}

@article {RW18,
    AUTHOR = {Royer, Cl\'ement W. and Wright, Stephen J.},
     TITLE = {Complexity analysis of second-order line-search algorithms for
              smooth nonconvex optimization},
   JOURNAL = {SIAM J. Optim.},
  FJOURNAL = {SIAM Journal on Optimization},
    VOLUME = {28},
      YEAR = {2018},
    NUMBER = {2},
     PAGES = {1448--1477},
      ISSN = {1052-6234,1095-7189},
   MRCLASS = {90C26 (49M15 90C06 90C60)},
  MRNUMBER = {3799071},
       DOI = {10.1137/17M1134329},
       URL = {https://doi.org/10.1137/17M1134329},
}

@article {ROW20,
    AUTHOR = {Royer, Cl\'ement W. and O'Neill, Michael and Wright, Stephen
              J.},
     TITLE = {A {N}ewton-{CG} algorithm with complexity guarantees for
              smooth unconstrained optimization},
   JOURNAL = {Math. Program.},
  FJOURNAL = {Mathematical Programming},
    VOLUME = {180},
      YEAR = {2020},
    NUMBER = {1-2},
     PAGES = {451--488},
      ISSN = {0025-5610,1436-4646},
   MRCLASS = {90C53 (49M15 65F10 65F15 90C26 90C60)},
  MRNUMBER = {4062843},
       DOI = {10.1007/s10107-019-01362-7},
       URL = {https://doi.org/10.1007/s10107-019-01362-7},
}

@article{XJY17,
  title={NEON+: Accelerated gradient methods for extracting negative curvature for non-convex optimization},
  author={Xu, Yi and Jin, Rong and Yang, Tianbao},
  journal={arXiv preprint arXiv:1712.01033},
  year={2017}
}

@InProceedings{JNJ18,
  title = 	 {Accelerated Gradient Descent Escapes Saddle Points Faster than Gradient Descent},
  author =       {Jin, Chi and Netrapalli, Praneeth and Jordan, Michael I.},
  booktitle = 	 {Proceedings of the 31st  Conference On Learning Theory},
  pages = 	 {1042--1085},
  year = 	 {2018},
  volume = 	 {75},
  month = 	 {06--09 Jul},
  url = 	 {https://proceedings.mlr.press/v75/jin18a.html},
}

@InProceedings{LL22,
  title = 	 {Restarted nonconvex accelerated gradient descent: No more polylogarithmic factor in the {$O(\varepsilon^{-7/4})$} complexity},
  author =       {Li, Huan and Lin, Zhouchen},
  booktitle = 	 {Proceedings of the 39th International Conference on Machine Learning},
  pages = 	 {12901--12916},
  year = 	 {2022},
  volume = 	 {162},
  month = 	 {17--23 Jul},
  url = 	 {https://proceedings.mlr.press/v162/li22o.html},
}

@inproceedings{JMP25,
author = {Jiang, Ruichen and Mokhtari, Aryan and Patitucci, Francisco},
title = {Improved Complexity for Smooth Nonconvex Optimization: A Two-Level Online Learning Approach with Quasi-Newton Methods},
year = {2025},
isbn = {9798400715105},
doi = {10.1145/3717823.3718308},
booktitle = {Proceedings of the 57th Annual ACM Symposium on Theory of Computing},
pages = {2225–2236},
numpages = {12},
}

@article {CGT10,
    AUTHOR = {Cartis, C. and Gould, N. I. M. and Toint, Philippe L.},
     TITLE = {On the complexity of steepest descent, {Newton}'s and regularized {Newton}'s methods for nonconvex unconstrained optimization problems},
   JOURNAL = {SIAM J. Optim.},
  FJOURNAL = {SIAM Journal on Optimization},
    VOLUME = {20},
      YEAR = {2010},
    NUMBER = {6},
     PAGES = {2833--2852},
      ISSN = {1052-6234,1095-7189},
   MRCLASS = {90C30 (49M37 65K05 90C60)},
  MRNUMBER = {2721157},
MRREVIEWER = {Klaus\ Meer},
       DOI = {10.1137/090774100},
       URL = {https://doi.org/10.1137/090774100},
}

@article {CDHS20,
    AUTHOR = {Carmon, Yair and Duchi, John C. and Hinder, Oliver and
              Sidford, Aaron},
     TITLE = {Lower bounds for finding stationary points {I}},
   JOURNAL = {Math. Program.},
  FJOURNAL = {Mathematical Programming},
    VOLUME = {184},
      YEAR = {2020},
    NUMBER = {1-2},
     PAGES = {71--120},
      ISSN = {0025-5610,1436-4646},
   MRCLASS = {90C26 (68Q25 90C06 90C30 90C60)},
  MRNUMBER = {4163541},
MRREVIEWER = {Klaus\ Meer},
       DOI = {10.1007/s10107-019-01406-y},
       URL = {https://doi.org/10.1007/s10107-019-01406-y},
}

@article {CDHS202,
    AUTHOR = {Carmon, Yair and Duchi, John C. and Hinder, Oliver and
              Sidford, Aaron},
     TITLE = {Lower bounds for finding stationary points {II}: first-order
              methods},
   JOURNAL = {Math. Program.},
  FJOURNAL = {Mathematical Programming},
    VOLUME = {185},
      YEAR = {2021},
    NUMBER = {1-2},
     PAGES = {315--355},
      ISSN = {0025-5610,1436-4646},
   MRCLASS = {90C26 (68Q25 90C06 90C30 90C60)},
  MRNUMBER = {4201716},
MRREVIEWER = {Guoyong\ Gu},
       DOI = {10.1007/s10107-019-01431-x},
       URL = {https://doi.org/10.1007/s10107-019-01431-x},
}

@inproceedings{AL18,
 author = {Allen-Zhu, Zeyuan and Li, Yuanzhi},
 booktitle = {Advances in Neural Information Processing Systems},
 pages = {},
 title = {NEON2: Finding Local Minima via First-Order Oracles},
 volume = {31},
 year = {2018}
}

@inproceedings{AABHM17,
author = {Agarwal, Naman and Allen-Zhu, Zeyuan and Bullins, Brian and Hazan, Elad and Ma, Tengyu},
title = {Finding approximate local minima faster than gradient descent},
year = {2017},
isbn = {9781450345286},
doi = {10.1145/3055399.3055464},
booktitle = {Proceedings of the 49th Annual ACM SIGACT Symposium on Theory of Computing},
pages = {1195–1199},
numpages = {5},
}

@misc{BDLLM25,
  title={On the equivalence of a {Hessian}-free inequality and {Lipschitz} continuous {Hessian}},
  author={Bo{\c{t}}, Radu I and Dao, Minh N and Liu, Tianxiang and Louren{\c{c}}o, Bruno F and Marumo, Naoki},
  year={2025},
  eprint = {2504.17193},
  archiveprefix = {arXiv},
  note = {arXiv:2504.17193},
  url = {https://arxiv.org/abs/2504.17193},
}

@article {OCBP14,
    AUTHOR = {Ochs, Peter and Chen, Yunjin and Brox, Thomas and Pock,
              Thomas},
     TITLE = {i{P}iano: inertial proximal algorithm for nonconvex
              optimization},
   JOURNAL = {SIAM J. Imaging Sci.},
  FJOURNAL = {SIAM Journal on Imaging Sciences},
    VOLUME = {7},
      YEAR = {2014},
    NUMBER = {2},
     PAGES = {1388--1419},
      ISSN = {1936-4954},
   MRCLASS = {94A08},
  MRNUMBER = {3218822},
       DOI = {10.1137/130942954},
       URL = {https://doi.org/10.1137/130942954},
}

@ARTICLE{JGR25,
  title     = "Computer-assisted design of accelerated composite optimization
               methods: {OptISTA}",
  author    = {Jang, Uijeong and Das Gupta, Shuvomoy and Ryu, Ernest K.},
  journal   = {Math. Program.},
  publisher = "Springer Science and Business Media LLC",
  year      =  {2026},
  language  = "en",
  doi={10.1007/s10107-025-02258-5},
  volume = {219},
  pages = {391--499},
  fjournal = {Mathematical Programming},
}

@article {OW19,
    AUTHOR = {O'Neill, Michael and Wright, Stephen J.},
     TITLE = {Behavior of accelerated gradient methods near critical points
              of nonconvex functions},
   JOURNAL = {Math. Program.},
  FJOURNAL = {Mathematical Programming},
    VOLUME = {176},
      YEAR = {2019},
    NUMBER = {1-2},
     PAGES = {403--427},
      ISSN = {0025-5610,1436-4646},
   MRCLASS = {90C26 (90C52)},
  MRNUMBER = {3960815},
MRREVIEWER = {Gonglin\ Yuan},
       DOI = {10.1007/s10107-018-1340-y},
       URL = {https://doi.org/10.1007/s10107-018-1340-y},
}

@InProceedings{FAL17,
  title = 	 {Model-Agnostic Meta-Learning for Fast Adaptation of Deep Networks},
  author =       {Chelsea Finn and Pieter Abbeel and Sergey Levine},
  booktitle = 	 {Proceedings of the 34th International Conference on Machine Learning},
  pages = 	 {1126--1135},
  year = 	 {2017},
  volume = 	 {70},
  month = 	 {06--11 Aug},
  url = 	 {https://proceedings.mlr.press/v70/finn17a.html},
}

@inproceedings{ADGHPSSd16,
 author = {Andrychowicz, Marcin and Denil, Misha and G{\'o}mez Colmenarejo, Sergio and Hoffman, Matthew W. and Pfau, David and Schaul, Tom and Shillingford, Brendan and {de Freitas}, Nando},
 booktitle = {Advances in Neural Information Processing Systems},
 pages = {},
 title = {Learning to learn by gradient descent by gradient descent},
 volume = {29},
 year = {2016},
}

@inproceedings{MPPS17,
title={Unrolled Generative Adversarial Networks},
author={Luke Metz and Ben Poole and David Pfau and Jascha Sohl-Dickstein},
booktitle={International Conference on Learning Representations},
year={2017},
url={https://openreview.net/forum?id=BydrOIcle}
}

@inproceedings{GL10,
author = {Gregor, Karol and LeCun, Yann},
title = {Learning fast approximations of sparse coding},
year = {2010},
isbn = {9781605589077},
booktitle = {Proceedings of the 27th International Conference on Machine Learning},
pages = {399--406},
numpages = {8},
  url = {https://icml.cc/2010/papers/449.pdf},
}

@inproceedings{YSLX16,
 author = {Yang, Yan and Sun, Jian and Li, Huibin and Xu, Zongben},
 booktitle = {Advances in Neural Information Processing Systems},
 pages = {},
 title = {Deep {ADMM-Net} for compressive sensing {MRI}},
 volume = {29},
 year = {2016}
}

@article{JY13,
  title={A literature survey of benchmark functions for global optimisation problems},
  author={Jamil, Momin and Yang, Xin-She},
  journal={International Journal of Mathematical Modelling and Numerical Optimisation},
  volume={4},
  number={2},
  pages={150--194},
  year={2013},
  publisher={Inderscience Publishers Ltd},
  doi = {10.1504/IJMMNO.2013.055204},
}

@article {UBT25,
    AUTHOR = {Upadhyaya, Manu and Banert, Sebastian and Taylor, Adrien B.
              and Giselsson, Pontus},
     TITLE = {Automated tight {L}yapunov analysis for first-order methods},
   JOURNAL = {Math. Program.},
  FJOURNAL = {Mathematical Programming},
    VOLUME = {209},
      YEAR = {2025},
    NUMBER = {1-2},
     PAGES = {133--170},
      ISSN = {0025-5610,1436-4646},
   MRCLASS = {90C25 (90C22 90C60 93D05)},
  MRNUMBER = {4851041},
       DOI = {10.1007/s10107-024-02061-8},
       URL = {https://doi.org/10.1007/s10107-024-02061-8},
}

@article {DVPGR24,
    AUTHOR = {Das Gupta, Shuvomoy and Van Parys, Bart P. G. and Ryu, Ernest
              K.},
     TITLE = {Branch-and-bound performance estimation programming: a unified
              methodology for constructing optimal optimization methods},
   JOURNAL = {Math. Program.},
  FJOURNAL = {Mathematical Programming},
    VOLUME = {204},
      YEAR = {2024},
    NUMBER = {1-2},
     PAGES = {567--639},
      ISSN = {0025-5610,1436-4646},
   MRCLASS = {90C25 (90C30 90C57)},
  MRNUMBER = {4706136},
       DOI = {10.1007/s10107-023-01973-1},
       URL = {https://doi.org/10.1007/s10107-023-01973-1},
}

@article {K21,
    AUTHOR = {Kim, Donghwan},
     TITLE = {Accelerated proximal point method for maximally monotone
              operators},
   JOURNAL = {Math. Program.},
  FJOURNAL = {Mathematical Programming},
    VOLUME = {190},
      YEAR = {2021},
    NUMBER = {1-2},
     PAGES = {57--87},
      ISSN = {0025-5610,1436-4646},
   MRCLASS = {47J25 (47H05 49J40 90C25 90C48)},
  MRNUMBER = {4322636},
       DOI = {10.1007/s10107-021-01643-0},
       URL = {https://doi.org/10.1007/s10107-021-01643-0},
}

@article {L21b,
    AUTHOR = {Lieder, Felix},
     TITLE = {On the convergence rate of the {H}alpern-iteration},
   JOURNAL = {Optim. Lett.},
  FJOURNAL = {Optimization Letters},
    VOLUME = {15},
      YEAR = {2021},
    NUMBER = {2},
     PAGES = {405--418},
      ISSN = {1862-4472,1862-4480},
   MRCLASS = {47J26 (49M29 90C22 90C48)},
  MRNUMBER = {4218746},
MRREVIEWER = {Vahid\ Mohebbi},
       DOI = {10.1007/s11590-020-01617-9},
       URL = {https://doi.org/10.1007/s11590-020-01617-9},
}

@InProceedings{PR22,
  title = 	 {Exact Optimal Accelerated Complexity for Fixed-Point Iterations},
  author =       {Park, Jisun and Ryu, Ernest K},
  booktitle = 	 {Proceedings of the 39th International Conference on Machine Learning},
  pages = 	 {17420--17457},
  year = 	 {2022},
  volume = 	 {162},
  month = 	 {17--23 Jul},
  url = 	 {https://proceedings.mlr.press/v162/park22c.html},
}

@InProceedings{YR21,
  title = 	 {Accelerated algorithms for smooth convex-concave minimax problems with {$O(1/k^2)$} rate on squared gradient norm},
  author =       {Yoon, Taeho and Ryu, Ernest K},
  booktitle = 	 {Proceedings of the 38th International Conference on Machine Learning},
  pages = 	 {12098--12109},
  year = 	 {2021},
  volume = 	 {139},
  month = 	 {18--24 Jul},
  url = 	 {https://proceedings.mlr.press/v139/yoon21d.html},
}

@article {DT16,
    AUTHOR = {Drori, Yoel and Teboulle, Marc},
     TITLE = {An optimal variant of {K}elley's cutting-plane method},
   JOURNAL = {Math. Program.},
  FJOURNAL = {Mathematical Programming},
    VOLUME = {160},
      YEAR = {2016},
    NUMBER = {1-2},
     PAGES = {321--351},
      ISSN = {0025-5610,1436-4646},
   MRCLASS = {90C25 (49M25 68Q25 90C20 90C22 90C60)},
  MRNUMBER = {3555391},
       DOI = {10.1007/s10107-016-0985-7},
       URL = {https://doi.org/10.1007/s10107-016-0985-7},
}

@article {KF18,
    AUTHOR = {Kim, Donghwan and Fessler, Jeffrey A.},
     TITLE = {Another look at the fast iterative shrinkage/thresholding
              algorithm ({FISTA})},
   JOURNAL = {SIAM J. Optim.},
  FJOURNAL = {SIAM Journal on Optimization},
    VOLUME = {28},
      YEAR = {2018},
    NUMBER = {1},
     PAGES = {223--250},
      ISSN = {1052-6234,1095-7189},
   MRCLASS = {90C25 (49M25 68Q25 90C22 90C60)},
  MRNUMBER = {3755677},
       DOI = {10.1137/16M108940X},
       URL = {https://doi.org/10.1137/16M108940X},
}

@article {BTB23,
    AUTHOR = {Barr\'e, Mathieu and Taylor, Adrien B. and Bach, Francis},
     TITLE = {Principled analyses and design of first-order methods with
              inexact proximal operators},
   JOURNAL = {Math. Program.},
  FJOURNAL = {Mathematical Programming},
    VOLUME = {201},
      YEAR = {2023},
    NUMBER = {1-2},
     PAGES = {185--230},
      ISSN = {0025-5610,1436-4646},
   MRCLASS = {90C22 (90C06 90C25 90C60)},
  MRNUMBER = {4620227},
       DOI = {10.1007/s10107-022-01903-7},
       URL = {https://doi.org/10.1007/s10107-022-01903-7},
}

@article {DTdB22,
    AUTHOR = {Dragomir, Radu-Alexandru and Taylor, Adrien B. and
              d'Aspremont, Alexandre and Bolte, J\'er\^ome},
     TITLE = {Optimal complexity and certification of {B}regman first-order
              methods},
   JOURNAL = {Math. Program.},
  FJOURNAL = {Mathematical Programming},
    VOLUME = {194},
      YEAR = {2022},
    NUMBER = {1-2},
     PAGES = {41--83},
      ISSN = {0025-5610,1436-4646},
   MRCLASS = {90C25 (90C06 90C22 90C60)},
  MRNUMBER = {4445450},
       DOI = {10.1007/s10107-021-01618-1},
       URL = {https://doi.org/10.1007/s10107-021-01618-1},
}

@article {RTB20,
    AUTHOR = {Ryu, Ernest K. and Taylor, Adrien B. and Bergeling, Carolina
              and Giselsson, Pontus},
     TITLE = {Operator splitting performance estimation: tight contraction
              factors and optimal parameter selection},
   JOURNAL = {SIAM J. Optim.},
  FJOURNAL = {SIAM Journal on Optimization},
    VOLUME = {30},
      YEAR = {2020},
    NUMBER = {3},
     PAGES = {2251--2271},
      ISSN = {1052-6234,1095-7189},
   MRCLASS = {47H05 (47H09 68Q25 90C22 90C25 90C60)},
  MRNUMBER = {4134368},
MRREVIEWER = {Liya\ Liu},
       DOI = {10.1137/19M1304854},
       URL = {https://doi.org/10.1137/19M1304854},
}

@InProceedings{DS20,
  title = 	 {The Complexity of Finding Stationary Points with Stochastic Gradient Descent},
  author =       {Drori, Yoel and Shamir, Ohad},
  booktitle = 	 {Proceedings of the 37th International Conference on Machine Learning},
  pages = 	 {2658--2667},
  year = 	 {2020},
  volume = 	 {119},
  month = 	 {13--18 Jul},
  url = 	 {https://proceedings.mlr.press/v119/drori20a.html},
}

@article {DFST25,
    AUTHOR = {Das Gupta, Shuvomoy and Freund, Robert M. and Sun, Xu Andy and
              Taylor, Adrien},
     TITLE = {Nonlinear conjugate gradient methods: worst-case convergence
              rates via computer-assisted analyses},
   JOURNAL = {Math. Program.},
  FJOURNAL = {Mathematical Programming},
    VOLUME = {213},
      YEAR = {2025},
    NUMBER = {1-2},
     PAGES = {1--49},
      ISSN = {0025-5610,1436-4646},
   MRCLASS = {90C25 (65K10 68T05 90C53 90C60)},
  MRNUMBER = {4952136},
       DOI = {10.1007/s10107-024-02127-7},
       URL = {https://doi.org/10.1007/s10107-024-02127-7},
}

@article {dGT17,
    AUTHOR = { {de Klerk}, Etienne and Glineur, Fran{\c c}ois and Taylor, Adrien B.},
     TITLE = {On the worst-case complexity of the gradient method with exact
              line search for smooth strongly convex functions},
   JOURNAL = {Optim. Lett.},
  FJOURNAL = {Optimization Letters},
    VOLUME = {11},
      YEAR = {2017},
    NUMBER = {7},
     PAGES = {1185--1199},
      ISSN = {1862-4472,1862-4480},
   MRCLASS = {90C52 (90C22 90C60)},
  MRNUMBER = {3702934},
       DOI = {10.1007/s11590-016-1087-4},
       URL = {https://doi.org/10.1007/s11590-016-1087-4},
}

@article {DGT20,
    AUTHOR = { {de Klerk}, Etienne and Glineur, Fran{\c c}ois and Taylor, Adrien B.},
     TITLE = {Worst-case convergence analysis of inexact gradient and {Newton} methods through semidefinite programming performance estimation},
   JOURNAL = {SIAM J. Optim.},
  FJOURNAL = {SIAM Journal on Optimization},
    VOLUME = {30},
      YEAR = {2020},
    NUMBER = {3},
     PAGES = {2053--2082},
      ISSN = {1052-6234,1095-7189},
   MRCLASS = {90C22 (65K05 90C26 90C30 90C51 90C52)},
  MRNUMBER = {4129979},
MRREVIEWER = {Zhongwen\ Chen},
       DOI = {10.1137/19M1281368},
       URL = {https://doi.org/10.1137/19M1281368},
}

@article{KSUMT24,
  title={Analysis of continuous dynamical system models with {Hessians} derived from optimization methods},
  author={Tomoya Kamijima and Shun Sato and Kansei Ushiyama and Takayasu Matsuo and Ken’ichiro Tanaka},
  journal={JSIAM Letters},
  volume={16},
  number={ },
  pages={29--32},
  year={2024},
  doi={10.14495/jsiaml.16.29}
}

@article{Z26,
  title={Sharp first-order lower bounds for higher-order smooth nonconvex optimization},
  author={Zhou, Dongruo},
  journal={arXiv preprint arXiv:2606.05438},
  year={2026}
}

@inproceedings{KB15,
  author    = {Diederik P. Kingma and Jimmy Ba},
  title     = {Adam: A Method for Stochastic Optimization},
  booktitle = {International Conference on Learning Representations},
  year      = {2015}
}

@misc{HADHRW26,
  author = {Hermant, Julien and Aujol, Jean-Fran{\c{c}}ois and Dossal, Charles and Huang, Lorick and Rondepierre, Aude and Waldspurger, Ir{\`e}ne},
  title = {Continuized {Nesterov} Momentum Achieves the {$O(\varepsilon^{-7/4})$} Complexity in Smooth Nonconvex Optimization},
  year = {2026},
  eprint = {2602.05504},
  archivePrefix = {arXiv},
  note = {arXiv:2602.05504},
  doi = {10.48550/arXiv.2602.05504},
}

\end{document}